\documentclass[12pt, a4]{amsart}

\usepackage{latexsym,amsmath,amssymb,amsthm,mathrsfs,color}
\usepackage{float}
\usepackage{graphicx}
\usepackage{marginnote}

\usepackage[bbgreekl]{mathbbol}
\usepackage{bbm}
\usepackage{tikz-cd}
\usepackage[all]{xy}
\usepackage{stmaryrd}
\usepackage{hyperref}
\usepackage{comment}
\hypersetup{colorlinks,%
    citecolor=brown,%
    filecolor=black,%
   linkcolor=blue,%
   urlcolor=black}

\usepackage[toc,page]{appendix}

\usepackage[shortalphabetic]{amsrefs}
 \makeatletter
\def\append@label@year@{%
    \safe@set\@tempcnta\bib@year
    \edef\bib@citeyear{\the\@tempcnta}%
    \ifnum\bib@citeyear>9
      \append@to@stem{%
          \ifx\bib@year\@empty
          \else
            \@xp\year@short \bib@citeyear \@nil
          \fi
      }%
    \fi
}
\makeatother

\let\oldtocsection=\tocsection
\renewcommand{\tocsection}[2]{\hspace{0em}\oldtocsection{#1}{#2}}

\usepackage[a4paper]{geometry}

\def\upddots{\mathinner{\mkern 1mu\raise 1pt \hbox{.}\mkern 2mu
\mkern 2mu \raise 4pt\hbox{.}\mkern 1mu \raise 7pt\vbox {\kern 7
pt\hbox{.}}} }

\numberwithin{equation}{section}

\begin{document}
\setlength{\unitlength}{2.5cm}

\newtheorem{thm}{Theorem}[section]
\newtheorem{lm}[thm]{Lemma}
\newtheorem{prop}[thm]{Proposition}
\newtheorem{cor}[thm]{Corollary}
\newtheorem{conj}[thm]{Conjecture}
\newtheorem{specu}[thm]{Speculation}

\theoremstyle{definition}
\newtheorem{dfn}[thm]{Definition}
\newtheorem{eg}[thm]{Example}
\newtheorem{rmk}[thm]{Remark}

\newcommand{\F}{\mathbf{F}}
\newcommand{\N}{\mathbbm{N}}
\newcommand{\R}{\mathbbm{R}}
\newcommand{\C}{\mathbbm{C}}
\newcommand{\Z}{\mathbbm{Z}}
\newcommand{\Q}{\mathbbm{Q}}
\newcommand{\Mp}{{\rm Mp}}
\newcommand{\Sp}{{\rm Sp}}
\newcommand{\GSp}{{\rm GSp}}
\newcommand{\GL}{{\rm GL}}
\newcommand{\PGL}{{\rm PGL}}
\newcommand{\SL}{{\rm SL}}
\newcommand{\SO}{{\rm SO}}
\newcommand{\Spin}{{\rm Spin}}
\newcommand{\GSpin}{{\rm GSpin}}
\newcommand{\Ind}{{\rm Ind}}
\newcommand{\Res}{{\rm Res}}
\newcommand{\Hom}{{\rm Hom}}
\newcommand{\End}{{\rm End}}
\newcommand{\msc}[1]{\mathscr{#1}}
\newcommand{\mfr}[1]{\mathfrak{#1}}
\newcommand{\mca}[1]{\mathcal{#1}}
\newcommand{\mbf}[1]{{\bf #1}}
\newcommand{\mbm}[1]{\mathbbm{#1}}
\newcommand{\into}{\hookrightarrow}
\newcommand{\onto}{\twoheadrightarrow}
\newcommand{\s}{\mathbf{s}}
\newcommand{\cc}{\mathbf{c}}
\newcommand{\bfa}{\mathbf{a}}
\newcommand{\id}{{\rm id}}
\newcommand{\g}{ \mathbf{g} }
\newcommand{\w}{\mathbbm{w}}
\newcommand{\Ftn}{{\sf Ftn}}
\newcommand{\p}{\mathbf{p}}
\newcommand{\bq}{\mathbf{q}}
\newcommand{\WD}{\text{WD}}
\newcommand{\W}{\text{W}}
\newcommand{\Wh}{{\rm Wh}}
\newcommand{\Whc}{{{\rm Wh}_\psi}}
\newcommand{\ggma}{\omega}
\newcommand{\sct}{\text{\rm sc}}
\newcommand{\Of}{\mca{O}^\digamma}
\newcommand{\gk}{c_{\sf gk}}
\newcommand{\Irr}{ {\rm Irr} }
\newcommand{\Irrg}{ {\rm Irr}_{\rm gen} }
\newcommand{\diag}{{\rm diag}}
\newcommand{\uchi}{ \underline{\chi} }
\newcommand{\Tr}{ {\rm Tr} }
\newcommand{\der}\de
\newcommand{\Stab}{{\rm Stab}}
\newcommand{\Ker}{{\rm Ker}}
\newcommand{\bfp}{\mathbf{p}}
\newcommand{\bfq}{\mathbf{q}}
\newcommand{\KP}{{\rm KP}}
\newcommand{\Sav}{{\rm Sav}}
\newcommand{\de}{{\rm der}}
\newcommand{\tnu}{{\tilde{\nu}}}
\newcommand{\lest}{\leqslant}
\newcommand{\gest}{\geqslant}
\newcommand{\tu}{\widetilde}
\newcommand{\tchi}{\tilde{\chi}}
\newcommand{\tomega}{\tilde{\omega}}
\newcommand{\Rep}{{\rm Rep}}
\newcommand{\cu}[1]{\textsc{\underline{#1}}}
\newcommand{\set}[1]{\left\{#1\right\}}
\newcommand{\ul}[1]{\underline{#1}}
\newcommand{\wt}[1]{\overline{#1}}
\newcommand{\wtsf}[1]{\wt{\sf #1}}
\newcommand{\anga}[1]{{\left\langle #1 \right\rangle}}
\newcommand{\angb}[2]{{\left\langle #1, #2 \right\rangle}}
\newcommand{\wm}[1]{\wt{\mbf{#1}}}
\newcommand{\elt}[1]{\pmb{\big[} #1\pmb{\big]} }
\newcommand{\ceil}[1]{\left\lceil #1 \right\rceil}
\newcommand{\floor}[1]{\left\lfloor #1 \right\rfloor}
\newcommand{\val}[1]{\left| #1 \right|}
\newcommand{\aff}{ {\rm a} }
\newcommand{\ex}{ {\rm ex} }
\newcommand{\exc}{ {\rm exc} }
\newcommand{\HH}{ \mca{H} }
\newcommand{\HKP}{ {\rm HKP} }
\newcommand{\std}{ {\rm std} }
\newcommand{\motimes}{\text{\raisebox{0.25ex}{\scalebox{0.8}{$\bigotimes$}}}}

\title[pro-$p$ Iwahori--Hecke algebras, Gelfand--Graev modules and applications]{Genuine pro-$p$ Iwahori--Hecke algebras, Gelfand--Graev representations, and some applications}

\author{Fan Gao, Nadya Gurevich, and Edmund Karasiewicz}
\address{F. Gao: School of Mathematical Sciences, Yuquan Campus, Zhejiang University, 38 Zheda Road, Hangzhou, China 310027}
\email{gaofan@zju.edu.cn}
\address{N. Gurevich: Department of Mathematics, Ben Gurion University of the Negev, Be'er Sheva,  Israel 8410501}
\email{ngur@math.bgu.ac.il }
\address{E. Karasiewicz: Department of Mathematics, University of Utah, Salt Lake City, USA 84112}
\email{karasiewicz@math.utah.edu}

\date{}
\subjclass[2010]{Primary 11F70; Secondary 22E50}
\keywords{covering groups, pro-$p$ Iwahori--Hecke algebra, Iwahori--Hecke algebras, Gelfand--Graev representations, Chinta--Gunnells representations, Whittaker dimensions}
\dedicatory{To Professor Gordan Savin on his 60th birthday.}
\maketitle

\begin{abstract} 
We study the Iwahori-component of the Gelfand--Graev representation of a central cover of a split linear reductive group and utilize our results for three applications. In fact, it is advantageous to begin at the pro-$p$ level. Thus to begin we study the structure of a genuine pro-$p$ Iwahori--Hecke algebra, establishing Iwahori--Matsumoto and Bernstein presentations. With this structure theory we first describe the pro-$p$ part of the Gelfand--Graev representation and then the more subtle Iwahori part. 

For the first application we relate the Gelfand--Graev representation to the metaplectic representation of Sahi--Stokman--Venkateswaran, which conceptually realizes the Chinta--Gunnells action from the theory of Weyl group multiple Dirichlet series. For the second we compute the Whittaker dimension of the constituents of regular unramified principal series; for the third we do the same for unitary unramified principal series.
\end{abstract}

\tableofcontents


\section{Introduction}
Let $F$ be a $p$-adic local field with ring of integers $O_F$ and residue field $\kappa$. Let $\mbf{G}$ be a smooth reductive linear algebraic group scheme over $O_F$. It is known that the generic fiber of $\mbf{G}$ is quasi-split and split over a finite unramified extension of $F$; for simplicity, we assume that it is  actually split over $F$. Write $G:=\mbf{G}(F)$ and $G_\kappa:=\mbf{G}(\kappa)$.

It is a well-known result of Rodier, Gelfand--Kazhdan and Shalika that every irreducible admissible representation $\pi \in \Irr(G)$ has at most one Whittaker model, i.e., 
\begin{equation} \label{E:M1}
\dim \Hom_G(\pi, \Ind_{U^-}^G \psi) \lest 1,
\end{equation}
where $U^-$ is the unipotent radical of a Borel subgroup of $G$ and $\psi$ is a nondegenerate character of $U^{-}$. This is a $p$-adic analogue of the finite field case, where it was first shown by Gelfand--Graev that the finite field analogue $\Ind_{U_\kappa^-}^{G_\kappa} \psi_\kappa$ of $\Ind_{U^-}^G \psi$ has similar multiplicity-one property. This multiplicity-one property for $\Ind_{U_\kappa^-}^{G_\kappa} \psi_\kappa$ follows from the commutativity of its endomorphism algebra $C(U_\kappa^-, \psi_\kappa \backslash G_\kappa /U_\kappa^-, \psi_\kappa)$, which can be proved by exhibiting a Chevalley--Steinberg involution, see \cite[\S 8.1]{Car} or \cite[Chapter 14]{Ste16}. The proof for $G$ essentially adapts this idea. The multiplicity one property \eqref{E:M1} for Whittaker models is important for the theory of $L$-functions, especially the Langlands--Shahidi method and some Rankin--Selberg integrals as well.

As a refinement of \eqref{E:M1}, it was shown in \cite[Theorem 4.3]{BuHe03}  that for simply generic Bernstein class $\mfr{s}$ with respect to $\psi$, there is a canonical isomorphism
\begin{equation} \label{E:BH}
\mfr{Z}^\mfr{s}(G) \longrightarrow {\rm End}_G({\rm ind}_{U^-}^G \psi)^\mfr{s},
\end{equation}
where $\mfr{Z}^\mfr{s}(G)$ is the Bernstein center of the subcategory of $\Irr(G)$ associated with $\mfr{s}$. Here ${\rm ind}_{U^-}^G \psi$ is the compact induction, whose dual is just ${\rm Ind}_{U^-}^G \psi^{-1}$. If one concentrates further on the unramified class $\mfr{s} = (T, \mbm{1})$, then further investigation of the structure of $(({\rm ind}_{U^-}^G \psi)^\mfr{s})^{I} = ({\rm ind}_{U^-}^G \psi)^I$ was given in several works \cite{BM1, Ree4, CS18} for example. Here $I \subset G$ is the Iwahori subgroup determined by the Borel subgroup $B \subset G$. In particular, it was shown in \cite{CS18} that if the conductor of $\psi$ is $\mfr{p}_F$, then one has
\begin{equation} \label{E:CS}
({\rm ind}_{U^-}^G \psi)^I \simeq \varepsilon_W \otimes_{\mca{H}_W} \mca{H}_I,
\end{equation}
where $\mca{H}_I = C_c^\infty(I\backslash G /I)$ is the Iwahori--Hecke algebra and $\varepsilon_W$ the sign character of the finite-dimensional subalgebra $\mca{H}_W$ deformed from $\C[W]$. Such a result was also obtained from a more general perspective by Brubaker--Bump--Friedberg \cite{BBF4}. For general $\mfr{s}=[T, \chi]$, the structure of 
$({\rm ind}_{U^-}^G \psi)^\mfr{s}$ was determined recently by \cite{MiPa21}.

More generally, one may consider the above problems for central covers of $G$. Assume that $F^\times$ contains the full group $\mu_n$ of $n$-th roots of unity. There are natural $n$-fold central covers
$$\begin{tikzcd}
\mu_n \ar[r, hook] & \wt{G} \ar[r, two heads] & G,
\end{tikzcd}$$
of which we consider only the genuine representations, i.e., $\mu_n$ acts via a fixed embedding $\epsilon: \mu_n \into \C^\times$. With the unipotent subgroup $U^-$ splitting uniquely in $\wt{G}$, one may consider the genuine Gelfand--Graev representation
$$\mca{V}:={\rm ind}_{\mu_n U^-}^{\wt{G}} (\epsilon\otimes \psi).$$
In this case, for fixed $\wt{G}$ we expect the right hand side of \eqref{E:M1} to be replaced by $\val{\wt{T}/Z(\wt{T})}^{1/2}$, where this upper bound is achieved, see \cite{GSS2}. In particular, multiplicity-one fails.

This multi-dimensionality of Whittaker models for genuine representation was observed decades ago. Though it gives obstacles to the classical theory of $L$-functions, it nevertheless motivates the theory of Weyl group multiple Dirichlet series (WMDS), which dates back at least to the work of Kubota.  These WMDS possess meromorphic continuation and functional equation, but are not Eulerian. Conjecturally, they are equal to certain Fourier coefficients of the Borel Eisenstein series of $\wt{G}$. See \cite{Bum12} for an exposition on this topic. There are various methods to construct such WMDS. One was given by Chinta--Gunnells \cite{CG10} by utilizing the so-called metaplectic $W$-representation afforded by $\C(Y)$, where $Y$ is the cocharacter lattice of $G$. Let $\mca{H}_I = C_{\epsilon, c}^\infty(I\backslash \wt{G} /I)$ be the Iwahori--Hecke algebra of $\wt{G}$. It was shown recently by  Sahi--Stokman--Venkateswaran \cite{SSV21} that the metaplectic $W$-representation of Chinta--Gunnells arises naturally from a certain $\mca{H}_I$-module afforded on the space $\C[Y]$, which we call the SSV representation of $\mca{H}_I$. In fact, the SSV representation is afforded by the bigger space $\C[P]$, where $P\supset Y$ is the coweight lattice. In any case, it should be highlighted that the construction in \cite{SSV21} is algebraic and uniform and does not involve computer-assisted checking for certain well-definedness of the Weyl group action.

In this paper we continue the study of Whittaker spaces for covering groups, focusing on the following problems:
\begin{enumerate}
\item[(P1)] Describe the Iwahori-component of the Gelfand--Graev representation, that is, to determine the $\mca{H}_I$-module structure of $\mca{V}^I={\rm ind}_{\mu_n U^-}^{\wt{G}} (\epsilon\otimes \psi)^I$. This will be a direct generalization of \cite{CS18} to covers.
\item[(P2)] Give a more natural and conceptual interpretation of the SSV representation in terms of $\mca{V}^I$. Relations between these two and also the metaplectic $W$-representation have been hinted at in several works including \cite{Mc2} \cite{CG10} \cite{CO} via the presence of the local scattering matrices, which describe the intertwining operators between the Whittaker models of principal series $I(\chi)$ and $I({}^w \chi)$.
\item[(P3)] Verify some speculative formulas regarding the Whittaker dimension of some Iwahori-spherical representations, especially those irreducible constituents of a regular or unitary unramified principal series, see \cite{Ga6, Ga7, Ga8}.
\end{enumerate}

A solution to (P1) would provide a valuable tool to investigate (P2) and (P3). Unfortunately, naively adapting the strategy of \cite{CS18} does not work. The fundamental obstacle is that the ``supports'' of $\mca{V}^I$ and $\mca{H}_I$ do not match. More precisely, for a linear group $G$, the Iwahori--Hecke algebra $\mca{H}_I$ is an affine Hecke algebra deformed from the extended affine Weyl group $Y\rtimes W$. In this case, the twisted Satake transform (see \cite{GuKa})
$$\mca{S}_\psi: C_c^\infty(I\backslash G/I) \longrightarrow ({\rm ind}_{U^-}^G \psi)^I$$
of $\mca{H}_I$-modules is surjective, and realizes \eqref{E:CS}. However, for $\wt{G}$, the support of the genuine Iwahori--Hecke algebra only corresponds to $Y_{Q,n} \rtimes W$, where $Y_{Q,n} \subset Y$ is a certain sublattice; yet the support of $\mca{V}^I$ still corresponds to $Y\rtimes W$. This is the support mismatch mentioned above. Consequently, $\mca{S}_\psi$ is not surjective for general covering groups. 

The support mismatch disappears if $I$ is replaced by its maximal pro-$p$ subgroup $I_{1}$ and we consider the algebra  $\mca{H}:=C_{\epsilon, c}^\infty(I_1 \backslash \wt{G} / I_1)$. Our first goal is then the following:
\begin{enumerate}
\item[(P0)] Study the larger pro-$p$ Hecke algebra $\mca{H}$ instead of $\mca{H}_I$, including its presentations; explicate the structure of $\mca{V}^{I_1}$ over $\mca{H}$. 
\end{enumerate}
We leverage these pro-$p$ results to investigate the Iwahori level and ultimately solve (P1)-(P3). 
The road map of the above topics and results in our paper is as follows:
$$
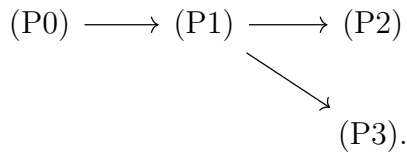
\begin{tikzcd}
\text{(P0)} \ar[r] & \text{(P1)} \ar[r] \ar[rd] & \text{(P2)} \\
& & \text{(P3)}.
\end{tikzcd}$$

\subsection{Main results}
To elaborate on the above, we give a brief outline of the paper and state our main results.

In \S \ref{S:Pre}, we introduce algebraic groups, their central covers, and fix some notations.

In \S \ref{S:proH}, we first study a $\Z$-model $\mca{H}_\Z$ of the pro-$p$ Iwahori--Hecke algebra $\mca{H}$. Using this, we give the Iwahori--Matsumoto presentation of the algebra $\mca{H}$. In \S \ref{S:proPS}, we give the Bernstein presentation of the Hecke algebra $\mca{H}$, which is based on the universal principal series in Haines--Kottwitz--Prasad \cite{HKP10}. Thus, we have the following result for \S \ref{S:proH}--\S \ref{S:proPS} as an amalgam of Theorem \ref{T:IM2} and Theorem \ref{T:BernR}:
\begin{thm}
The pro-$p$ Hecke algebra $\mca{H}$ has an explicit Iwahori--Hecke presentation and Bernstein presentation. 
\end{thm} 
We note that for linear groups the pro-$p$ Iwahori-Hecke algebra was studied by Vigneras \cite{Vig16} and Flicker \cite{Fl11}.  

The Bernstein presentation of $\mca{H}$ gives rise to the corresponding one for $\mca{H}_I$, which we explicate in \S \ref{SS:IHalg}. This presentation already appears in the earlier work \cite{Sav88, Sav04, Mc1, GG}. However, by using the pro-$p$ algebra we can easily prove Corollary \ref{PropInvert} and circumvent the technical complications faced by Savin in \cite[Proposition 6.5]{Sav04}.

In \S \ref{S:GG}, we first discuss the structure of 
$$\mca{V}_\kappa=(\Ind_{U_\kappa^-}^{G_\kappa} \psi)^{U_\kappa},$$
the $U_{\kappa}$-fixed vectors in the Gelfand--Graev representation of $G_\kappa$. We show that there is a decomposition 
$$\mca{V}_\kappa = \bigoplus_{\mca{O}} \mca{V}_{\kappa, \mca{O}}$$
into irreducible modules over $\mca{H}_\kappa = C(U_\kappa \backslash G_\kappa /U_\kappa)$, where the sum is taken over all $W$-orbits in $\Hom(T_\kappa, \C^\times)$. There is a $W$-stable sublattice $Y_{Q,n} \subset Y$ which dictates $Z(\wt{T})$, the center of the covering torus $\wt{T}$. We write  $\msc{X}_{Q,n}:=Y/Y_{Q,n}$. The main result pertaining to (P0) and (P1) is the following:

\begin{thm}[{Theorem \ref{GGProP}, Corollary \ref{C:VI-O}, Theorem \ref{Wequi}}] \label{T:V^I}
There is a natural isomorphism
$$\pmb{\gamma}: \mca{V}_\kappa \otimes_{\mca{H}_\kappa} \mca{H} \longrightarrow \mca{V}^{I_1}$$
of $\mca{H}$-modules. It gives rise to a decomposition
$$\mca{V}^I = \bigoplus_{\mca{O} \subset \msc{X}_{Q,n}} \mca{V}^I_\mca{O}$$
over all $W$-orbits in $\msc{X}_{Q,n}$. Moreover, if $\mca{O}=\mca{O}_{y} \subset \msc{X}_{Q,n}$ is splitting (see Definition \ref{D:splO}), then one has natural isomorphisms of $\mca{H}_I$-modules
$$\begin{tikzcd}
 \varepsilon_y \otimes_{\mca{H}_{I, y}} \mca{H}_{I} \ar[r] & \mca{V}_{\kappa, \mathcal{O}}\otimes_{\mca{H}_\kappa}\mca{H}*\mbm{1}_{I} \ar[r, "{\pmb{\gamma}}"] & \mca{V}^I_\mca{O},
 \end{tikzcd}$$
 where $\varepsilon_y$ is the sign character of the deformed subalgebra $\mca{H}_{I, y} \subset \mca{H}_W \subset \mca{H}_I$ associated with the parabolic Weyl subgroup $W_y=\mathrm{Stab}_{W}(y) \subset W$.
\end{thm}

The trivial orbit of $0$ and free orbits are always splitting, which immediately gives Corollary \ref{O-triv} and \ref{O-free}. In particular, if $n=1$, then $\wt{G}=G$ is a linear group, $\msc{X}_{Q,n}=\set{0}$, and in this case Corollary \ref{O-triv} implies \eqref{E:CS}.

From Theorem \ref{T:V^I}, it is thus an interesting question to determine which orbits $\mca{O} \subset \msc{X}_{Q,n}$ are splitting, and even more to determine the covering groups $\wt{G}$ for which every orbit $\mca{O}$ is splitting. We give a detailed discussion of several (nested) subclasses of covering groups in \S \ref{S:splO}. In particular, we show in Corollary \ref{C:aOsplt} that for an oasitic cover of an almost simple and simply connected $G$, every orbit $\mca{O}$ is splitting. Other examples with this splitting property include the Kazhdan--Patterson and Savin covers of $\GL_r$, which we discuss in Example \ref{E:KPS}.

In the remainder of \S \ref{S:splO}, we give a full account of the structure of $\mca{V}^I_\mca{O}$ for covers $\wt{\SL}_2$ of $\SL_2$. This example is already instructive. Indeed, if the invariant $n^*=n/\gcd(n, 2Q(\alpha^\vee))$ associated with $\wt{\SL}_2$ is odd, then every orbit $\mca{O}\subset \msc{X}_{Q,n}$ is splitting and thus Theorem \ref{T:V^I} applies. However, if $n^*$ is even, then there is one peculiar orbit which is non-splitting. This illustrates some subtleties in determining the structure of $\mca{V}^I_\mca{O}$ for general $\mca{O}$.

In \S \ref{S:SSV}, we identify the $\mca{H}_I$-module $\mca{V}^I$ as a submodule of the metaplectic representation $(\pi, \C[P])$ constructed by Sahi--Stokman--Venkateswaran \cite{SSV21}. As mentioned above, there is a natural $\mca{H}_I$-module $(\pi^0, \C[P])$ afforded by  $\C[P]$. The subspace $\C[Y]$ is invariant under $\pi$ and thus gives $(\pi^0, \C[Y])$. Moreover, there is a decomposition 
$$\pi^0 =\bigoplus_{\mca{O} \subset \msc{X}_{Q,n}} \pi^0_\mca{O},$$
where $\mca{O}$ is taken over all orbits in $\msc{X}_{Q,n}$. We have the following answer to (P2) above:

\begin{thm}[{Theorem \ref{T:GG=SSV0}}]
Let $\wt{G}$ be an $n$-fold cover of a semisimple group $G$. Assume $(-1, \varpi)_n=1$. Then for every $W$-orbit $\mca{O}\subset \msc{X}_{Q,n}$, one has
$$\mca{V}^I_\mca{O} \simeq \pi^0_\mca{O}$$
as $\mca{H}_I$-module; hence, $\mca{V}^I \simeq \pi^0$ as well.
\end{thm}

In fact, for each $z\in P$, the space $\C[Y]\cdot x^z \subset \C[P]$ is also $\mca{H}_I$-invariant and thus gives a representation $(\pi^z, \C[Y]\cdot x^z)$ of $\mca{H}_I$. We speculate that $\pi^z$ is related to $({\rm ind}_{\mu_n U^-}^G \epsilon \otimes {}^z \psi)^I$, see Conjecture \ref{C:pi-z}.

The remaining part of the paper concerns (P3). We apply Theorem \ref{T:V^I} to determine certain Whittaker dimensions and verify several formulas speculated in \cite{Ga6, Ga7}. In fact, for every orbit $\mca{O} \subset \msc{X}_{Q,n}$, one can define the $\mca{O}$-Whittaker subspace of any Iwahori-spherical $\pi\in \Irr(\wt{G})$ as
$$\Wh_\psi(\pi)_\mca{O} := \Hom_{\mca{H}_I}(\mca{V}^I_\mca{O}, \check{\pi}^I).$$
In the case that $\pi=I(\chi)$ is an unramified principal series, we expect that $\Wh_\psi(I(\chi))_\mca{O}$ is isomorphic to another naturally defined $\mca{O}$-Whittaker space $\Wh_\psi(I(\chi))_\mca{O}^\sharp$ discussed more frequently in the literature, see Conjecture \ref{C:iden}.

In any case, in the last part of \S \ref{S:Wdim}, we consider regular unramified genuine principal series $I(\chi)$. We assume that the set $\Phi(\chi)$ of reducibility of $I(\chi)$ is a subset of simple roots. Then by Rodier's result, the semisimplificiation of $I(\chi)^{\rm ss}$ is multiplicity free and there is a natural bijection
$$\msc{P}(\Phi(\chi)) \longrightarrow {\rm JH}(I(\chi)), \quad S \mapsto \pi_S$$
determined by the Jacquet module of $\pi_S$. Here $\msc{P}(\Phi(\chi))$ is the power set of $\Phi(\chi)$ and ${\rm JH}(I(\chi))$ is the Jordan--Holder set of $I(\chi)$. Every orbit $\mca{O} \subset \msc{X}_{Q,n}$ gives rise to a Weyl group permutation
$$\sigma_\mca{O}^\msc{X}: W \longrightarrow {\rm Perm}(\mca{O})$$
given by  $\sigma_\mca{O}^\msc{X}(w)(y):=w(y)$.  We have $\sigma^\msc{X}=\bigoplus_{\mca{O} \subset \msc{X}_{Q,n}} \sigma^\msc{X}_\mca{O}$, which is the permutation representation of $W$ realized on $\msc{X}_{Q,n}$.

\begin{thm}[{Theorem \ref{T:reg-ps}}] \label{T:Mreg}
Keeping the notation above, and let $S\subset \Phi(\chi)$. Then for every splitting orbit $\mca{O} \subset \msc{X}_{Q,n}$ one has
$$\dim \Whc(\pi_S)_\mca{O} = \angb{\sigma_S}{\sigma^\msc{X}_\mca{O}}_W,$$
where $\sigma_S \in {\rm Rep}(W)$ is a sum of certain Kazhdan--Lusztig representations of $W$ naturally associated with $S$.
Hence, for $\wt{G}$ such that every orbit $\mca{O}$ is splitting (for example, those as in Corollary \ref{C:aOsplt} and Example \ref{E:KPS}) one has 
$$\dim \Whc(\pi_S) = \angb{\sigma_S}{\sigma^\msc{X}}_W.$$
\end{thm}

The proof of Theorem \ref{T:Mreg} is essentially the same as that in \cite{Ga6} which concerns $\Wh_\psi(\pi_S)_\mca{O}^\sharp$. However, the crucial difference is that we know that the functor $\dim \Wh_\psi(-)_\mca{O}$ is exact, whereas the exactness is not clear for $\Wh_\psi(-)_\mca{O}^\sharp$. Thus, Theorem \ref{T:Mreg} verifies the analogue of \cite[Conjecture 1.1]{Ga6} for $\Wh_\psi(\pi_S)_\mca{O}$ considered in this paper, and also for $\Wh_\psi(\pi_S)_\mca{O}^\sharp$ if we assume Conjecture \ref{C:iden}. 

In \S \ref{S:uniPS}, we consider unitary unramified genuine principal series $I(\chi)$ with a decomposition
$$I(\chi) = \bigoplus_{\sigma\in \Irr(R_\chi)} \pi_\sigma,$$
where $R_\chi$ is the $R$-group of $I(\chi)$.

\begin{thm}[{Theorem \ref{T:uni-ps}}] \label{TM:uni}
Let $\wt{G}$ be a very saturated cover of an almost simple simply-connected  $G$ associated with $Q(\alpha^\vee)=1$ for any short simple coroot $\alpha^\vee$. Let $I(\chi)$ be a unitary $(K, s_K)$-unramified genuine principal series of $\wt{G}$. Let $\mca{O} \subset \msc{X}_{Q,n}$ be a $W$-orbit satisfying the S-property (see Definition \ref{D:S-ppty}). Then
$$\dim \Whc(\pi_\sigma)_\mca{O} =  \angb{\sigma \otimes \zeta_{\tilde{\rho}} }{  \sigma_\mca{O}^\msc{X} }_{R_\chi}$$ 
for every $\sigma \in \Irr(R_\chi)$.
\end{thm}

In the above theorem, $\zeta_{\tilde{\rho}}$ is a character of $R_\chi$ given in \eqref{D:z-rho}. The proof relies on explicating the Kazhdan--Lusztig--Reeder parametrization of $\pi_\sigma$ in terms of a homology space. If $\wt{G}$ is an oasitic cover of almost simple simply-connected group, then every $W$-orbit $\mca{O}$ satisfies the condition in Theorem \ref{TM:uni} and this gives Corollary \ref{C:uni-ps}.

At the end of \S \ref{S:uniPS}, we also prove a result regarding the variation of Whittaker dimension with respect to changing the additive character from $\psi$ of conductor $\mfr{p}_F$ to ${}^{\rho} \psi$ of conductor $O_F$. This could be considered as a covering analogue of the linear case as discussed in \cite[\S 4]{Kal4} or \cite[\S 9]{GGP1}. It in particular verifies a special case of \cite[Conjecture 5.7]{GSS3}. See Corollary \ref{C:Wh-equi} for details.

In summary, Theorem \ref{T:Mreg}, \ref{TM:uni}  and Corollary \ref{C:uni-ps} constitute the main results for (P3) above.

\subsection{Acknowledgements}  We would like to thank Gordan Savin for his influence on this work. We have learned a great deal from Gordan over the years; our results on Hecke algebras and the Gelfand--Graev representation are directly inspired by his work. Throughout the preparation of this manuscript, Gordan offered numerous insights that greatly improved the presentation and the quality of our results. 

Thanks are also due to Dani Szpruch for communications on the general local coefficients and some results pertaining to our work. The first-named author is partially supported by NSFC-12171422. The second author would like to thank Sergey Lysenko for drawing her attention to the special role of the pro-$p$ Hecke algebra in the case of covering groups.

\section{Preliminaries} \label{S:Pre}

\subsection{Algebraic groups}
Let $p$ be a prime number. Let $F$ be a finite extension of $\Q_{p}$ with ring of integers $O=O_{F}$ and maximal ideal $\frak{p}=\frak{p}_{F}$. Let $\varpi\in O_F$ be a fixed uniformizer, and thus $\mfr{p}=\varpi O_F$. Let $q$ denote the size of the residue field $\kappa=O_F/\mfr{p}$.

Let $\mbf{G}$ be a connected reductive linear algebraic group over $O_F$. We assume that the generic fiber at ${\rm Spec}(F)$, which is still denoted by $\mbf{G}$, is split over $F$ with maximal split torus $\mbf{T}$. We write 
$$X=\Hom(\mbf{T}, \mbf{G}_m) \text{ and } Y=\Hom(\mbf{G}_m, \mbf{T})$$
for the lattice of characters and cocharacters of $\mbf{T}$ respectively. Here $X$ and $Y$ form a perfect pairing
$$\angb{-}{-}: X\times Y \rightarrow \Z$$
given by $x\circ y(t)=t^{\langle x,y\rangle}$ for $t\in \mbf{G}_m$. Let $\Phi$ and $\Phi^{\vee}$ denote the roots and coroots of $(\mbf{G},\mbf{T})$, respectively. We write $Y^{\mathrm{sc}} \subset Y$ for the sublattice generated by the coroots. 

Let $N(\mbf{T}) \subset \mbf{G}$ be the normalizer of $\mbf{T}$ in $\mbf{G}$. This gives the Weyl group 
$$W=N(\mbf{T})/\mbf{T},$$
which we identify with the Weyl group of the coroot system. Given $\alpha^{\vee}\in \Phi^{\vee}$, we write $w_{\alpha}$ for the associated reflection of $Y\otimes \Q$. 

 We fix a Borel subgroup $\mbf{B}$ containing $\mbf{T}$ with unipotent radical $\mbf{U}$. This choice of $\mbf{B}$ identifies a set of positive roots $\Phi_{+}$ (resp. positive coroots $\Phi^{\vee}_{+}$) and simple roots $\Delta\subset \Phi_{+}$ (resp. simple coroots $\Delta^{\vee}\subset \Phi^{\vee}_{+}$). The choice of simple roots induces a length function 
 $$\ell: W \longrightarrow \Z_{\gest 0}.$$
 Let $\mbf{B}^{-} =\mbf{T} \mbf{U}^-$ denote the opposite Borel subgroup. We fix a Chevalley--Steinberg system $\set{e_\alpha: \mbf{G}_a \to \mbf{U}_\alpha}_{\alpha \in \Phi}$ of pinnings for $(\mbf{G},\mbf{T})$.
For each $\alpha\in \Phi$ we can define the  map 
$$w_{\alpha}:\mbf{G}_m \longrightarrow N(\mbf{T})$$ 
given by $w_{\alpha}(a)=e_{\alpha}(a)\cdot e_{-\alpha}(-a^{-1})\cdot e_{\alpha}(a)$. For $\alpha\in \Phi$ we write $h_{\alpha}(a):=\alpha^{\vee}(a)$ which satisfies $h_{\alpha}(a)=w_{\alpha}(a)\cdot w_{\alpha}(-1)$.

For simplicity of notation, for a group $\mbf{H}$ over $\kappa$, we  write
$$H_\kappa:=\mbf{H}(\kappa).$$
Consider the hyperspecial maximal compact subgroup $K:=\mbf{G}(O_F)$.
The reduction mod $\frak{p}$ map
 $${\rm red}_\mfr{p}: K \onto G_\kappa =\mbf{G}(\kappa)$$
 is surjective (see Tits \cite[\S 3.4.4]{Tit79}). Let
 $$I:={\rm red}_\mfr{p}^{-1}(B_\kappa) \subset K$$ be the Iwahori subgroup. Inside $I$ one has a unique maximal pro-$p$ normal subgroup
 $$I_1:={\rm red}_\mfr{p}^{-1}(U_\kappa).$$
We write $G, B, T, U$ for the $F$-rational points of $\mbf{G},\mbf{B},\mbf{T},\mbf{U}$. For any subgroup $H\subset G$ let 
$$H_{1}:=H\cap I_{1}.$$

The group $U$ decomposes as a product of the subgroups $U_{\alpha}$, where $\alpha\in \Phi_{+}$; similarly for the opposite unipotent radical $U^-$ with respect to $\Phi_-$. Each $U_{\alpha}$ has a filtration by subgroups 
$$U_{\alpha}^l=\set{e_{\alpha}(u): u\in\frak{p}^l }$$
indexed by $l \in \Z$ and satisfying $U_\alpha^l \supset U_\alpha^{l+1}$. See \cite{Tit79} for details.

 \subsection{Root system and affine root system} \label{SS:rsys}
Consider
$$\msc{A}=Y\otimes\R \text{ and } \msc{A}^{\rm sc}=Y^{\rm sc} \otimes \R.$$
The coroot system $\Phi^{\vee}\subset \msc{A}^{\mathrm{sc}}$ decomposes into irreducible coroot systems $\Phi_j^{\vee} \subset \msc{A}_j:= {\rm span}_\R(\Phi_j^\vee)$ satisfying 
$$\Phi^\vee=\bigcup_{j=1}^d \Phi_{j}^\vee, \quad \msc{A}^{\mathrm{sc}}=\bigoplus_{j=1}^d \msc{A}_{j}.$$
This gives a partition of $\Delta^{\vee}$ into  $\Delta_{j}^{\vee}=\Delta^{\vee}\cap \Phi_{j}^{\vee}$, the sets of simple roots for the $\Phi_{j}^{\vee}$'s. Each $\Phi_j$ has a unique highest root $\alpha_j^\dag$ with respect to $\Delta_j$.

For any $\alpha\in\Phi$ and $k\in \Z$, consider 
$$H_{\alpha, k}=\set{ v \in \msc{A}:\ \angb{\alpha}{v}=k}.$$
Let $w_{\alpha, k}$ be the reflection of $\msc{A}$ fixing the hyperplane $H_{\alpha, k}$, i.e., $w_{\alpha, k}(v)= v -(\langle\alpha,v\rangle - k)\alpha^\vee$. 
The set of reflections $\{w_{\alpha, k}| \alpha \in \Phi, k\in \Z\}$ generates the affine Weyl group $W_\aff = Y^{\rm sc} \rtimes W$. Our choice of $\Phi_{+}$ identifies the alcove $\mca{C}$ that is contained in the positive Weyl chamber such that $0 \in \overline{\mca{C}}$. The set of reflections associated to the walls of $\mca{C}$ 
$$S_\aff=\set{ w_{\alpha}: \alpha\in \Delta } \cup \set{w_{\alpha_j^\dag, 1}: 1\lest j \lest d }$$
is a minimal set of generators for $W_{\aff}$ and realizes $W_{\aff}$ as a Coxeter group.
The group $W_\aff$ acts simply transitively on the set of alcoves, which are the connected components of 
$$\msc{A}-\bigcup_{\alpha \in \Phi, k \in \Z} H_{\alpha, k}.$$
Let $W_{\rm ex}:=Y \rtimes W$ be the extended affine Weyl group. One has 
$$W_{\rm ex} = W_\aff \rtimes \Omega,$$
where $\Omega=\set{w \in W_{\rm ex}: w(\mca{C}) = \mca{C}}$.

A minimal expression of the element $w\in W_\aff$ with respect to $S_\aff$ is a factorization $w=w_{1}\ldots w_{k}$ with $w_j \in S_\aff$ such that there is no factorization of $w$ using fewer elements of $S_\aff$. The (strong) Bruhat order $w<w'$ on the Coxeter group $W_\aff$ with respect to $S_\aff$ is defined as follows:  $w<w'$ if there are minimal expressions $w=w_{1}\ldots w_{k}$ and $w'=w_{1}'\ldots w_{\ell}'$ such that the sequence $w_{1},\ldots, w_{k}$ is a subsequence of $w_{1}',\ldots, w_{k}'$.

 \subsection{Covering groups} 
 By abuse of notation, we still use $\mbf{G}$ to denote the generic fiber of $\mbf{G}$ at the point ${\rm Spec}(F)$.
Consider the $\mbf{K}_{2}$-extension  \cite{BD}
$$\begin{tikzcd}
\mbf{K}_2 \ar[r, hook] & \wt{\mbf{G}} \ar[r, two heads, "\wp"] & \mbf{G}
\end{tikzcd}$$
over $F$ which is incarnated by the pair $(D, \eta = \mbm{1})$, see \cite{We3} or \cite[\S 2.6]{GG}. Here 
$$D: Y \times Y \to \Z$$
is a (not necessarily symmetric) bilinear form such that the quadratic form 
$$Q(y):=D(y,y)$$ is Weyl-invariant. We have a Weyl-invariant symmetric bilinear form $B_Q$ given by 
$$B_Q(y, z) =D(y, z) + D(z, y).$$

The extension $\wt{\mbf{G}}$ splits canonically and $\mbf{G}$-equivariantly over any unipotent subgroup of $\mbf{G}$. Thus, we write
$$\wt{e}_{\alpha}:\mbf{G}_{a} \to \wm{G}$$
for the splitting of $e_{\alpha}$ for any $\alpha \in \Phi$. With this, one gets 
$$\wt{w}_{\alpha}(a):=\wt{e}_{\alpha}(a)\cdot \wt{e}_{-\alpha}(-a^{-1})\cdot \wt{e}_{\alpha}(a) \in \wt{N(\mbf{T})} \text{ and } \wt{h}_{\alpha}(a):=\wt{w}_{\alpha}(a)\cdot \wt{w}_{\alpha}(-1)\in \wm{T}.$$
Also, there is a section $\s$ of the quotient map $\wm{T} \onto \mbf{T}$ such that for any $a_{j}\in \mbf{G}_m, y_{j}\in Y, i=1, 2$ one has
\begin{equation}\label{ppty1}
\s(y_{1}(a_{1}))\cdot \s(y_{2}(a_{2}))=\{a_{1},a_{2}\}^{D(y_{1},y_{2})}\cdot \s(y_{1}(a_{1})\cdot y_{2}(a_{2})),
\end{equation}
where $\{a_{1},a_{2}\}\in \mbf{K}_{2}$ as in \cite[\S 0.N.5]{BD}. Furthermore, since we have assumed $\eta=\mbm{1}$, for $\alpha \in \Delta$ one has
\begin{equation}\label{ppty2}
\wt{h}_{\alpha}(a)=\s(h_{\alpha}(a))\in\wm{T}.
\end{equation}
Writing $\wt{w}_{\alpha}$ for $\wt{w}_{\alpha}(1)$ for every $\alpha\in \Phi$, it then follows from \cite[Proposition 11.3]{BD} that
\begin{equation}\label{ppty3}
\wt{w}_{\alpha}\cdot \s(y(a))\cdot \wt{w}_{\alpha}^{-1}=\s(y(a)) \cdot \wt{h}_{\alpha}(a^{-\langle \alpha,y\rangle}).
\end{equation}
 for every $y\in Y$ and $a\in \mbf{G}_{m}$.

Assume that $F^\times$ contains the full group $\mu_n$ of $n$-th roots of unity. By push-out via the Hilbert symbol $(-,-)_n: \mbf{K}_2(F) \to \mu_n$, one obtains from $\wm{G}(F)$ a topological central extension
$$\begin{tikzcd}\label{CExt}
\mu_n \ar[r, hook] & \wt{G} \ar[r, two heads, "\wp"] & G.
\end{tikzcd}$$
For any subset $H \subset G$, we may write $\wt{H}:=\wp^{-1}(H)$. Fixing an embedding $\epsilon: \mu_n \into \C^\times$ we consider only $\epsilon$-genuine representation of $\wt{G}$, i.e., when $\mu_n$ acts via $\epsilon$.

All the properties in \eqref{ppty1}--\eqref{ppty3} specialize to give corresponding relations on elements in $\wt{G}$. For example, from \eqref{ppty1} we obtain 
\begin{equation}\label{ppty1'}
\s(y_{1}(a_{1}))\cdot \s(y_{2}(a_{2}))=(a_{1},a_{2})_{n}^{D(y_{1},y_{2})}\cdot \s(y_{1}(a_{1})\cdot y_{2}(a_{2})),
\end{equation}
for any $y_{j}\in Y$, $a_{j}\in F^{\times}$. For $y\in Y$ denote
$$\s_y:=\s(y(\varpi)).$$
The commutator on $\overline{T}$ factors through $T$ and defines a map
\begin{equation}
[-,-]:T\times T\rightarrow \mu_{n}.
\end{equation}

We collect below some relations in $\wt{G}$ which we use in our computations. Let $\alpha\in \Phi$, $t,t_{j}\in F^{\times}$, $u\in F$, $y_{j}\in Y$. Then:
\begin{subequations}
\begin{align}
 \wt{h}_{-\alpha}(t)& =(t,-1)_n^{Q(\alpha^{\vee})} \cdot \wt{h}_{\alpha}(t^{-1}) \label{Rel1}   \\ 
 \s(y_{1}(t_{1}))\cdot \s(y_{2}(t_{2})) & =(t_{1},t_{2})_n^{D(y_{1},y_{2})}\cdot \s(y_{1}(t_{1})\cdot y_{2}(t_{2})) \label{Rel2} \\
 [y_{1}(t_{1}),y_{2}(t_{2})] & =(t_{1},t_{2})_{n}^{B_{Q}(y_{1},y_{2})}   \label{Rel3} \\
 \wt{w}_{-\alpha}(t) & =\wt{w}_{\alpha}(-t^{-1}) \label{Rel4} \\
 \wt{w}_{\alpha}(t_{1})\wt{w}_{\alpha}(t_{2})& =(-t_{1},-t_{2})_n^{Q(\alpha^{\vee})} \cdot \wt{h}_{\alpha}(-t_{1}t_{2}^{-1}) \label{Rel5} \\
\wt{w}_{\alpha}(t)\wt{e}_{\pm\alpha}(u)\wt{w}_{\alpha}(-t) & =\wt{e}_{\mp\alpha}(-t^{\mp2}u) \label{Rel6} \\
\wt{w}_{\alpha} \cdot \wt{y}(t) \cdot \wt{w}_{\alpha}^{-1} & =\wt{y}(t) \cdot \wt{h}_{\alpha}(t^{-\langle \alpha,y\rangle}). \label{Rel7}
 \end{align}
 \end{subequations}

\subsection{Dual group}
Consider the sublattice
$$Y_{Q,n}:=Y\cap nY^{*}$$
of $Y$, where $Y^{*}\subset Y\otimes \Q$ is the lattice dual to $Y$ with respect to $B_{Q}$. The lattice $Y_{Q,n}$ dictates the center $Z(\wt{T})$ of $\wt{T}$, see \cite{We1}. 
We write 
$$\msc{X}_{Q,n}:=Y/Y_{Q,n}$$
 or just $\msc{X}$ for simplicity. For $y\in Y$, we write $\hat{y} \in \msc{X}_{Q,n}$ for the coset of $y$.

We also write $X_{Q,n}:=\Hom_\Z(Y_{Q,n}, \Z)$. For every $\alpha\in \Phi$ we set
$$\alpha_{Q,n}^{\vee}=n_{\alpha}\alpha^{\vee} \text{ and } \alpha_{Q,n}=n_\alpha^{-1} \alpha,$$
where $n_\alpha = n/\gcd(n, Q(\alpha^\vee))$. This gives the modified simple roots $\Delta_{Q,n}=\set{\alpha_{Q,n}: \alpha\in \Delta}$ and coroots $\Delta_{Q,n}^\vee$; similarly for $\Phi_{Q,n}$ and $\Phi_{Q,n}^\vee$. Let $Y_{Q,n}^{sc}$ denote the coroot lattice spanned by $\Delta_{Q,n}^{sc}$. One has
$$Y_{Q,n}^{sc} \subset Y_{Q,n} \subset Y.$$

The tuple
\begin{equation*}
(Y_{Q,n},\ \Phi_{Q,n}^{\vee};\ X_{Q,n},\ \Phi_{Q,n})
\end{equation*}
forms a root datum. Let $\wm{G}_{Q,n}^\vee$ be the associated reductive group over $\Z$ with character lattice $Y_{Q,n}$. Write $\wt{G}_{Q,n}^\vee$ or simply $\wt{G}^\vee$ for its complex group, which is called the dual group of $\wt{G}$.

\subsection{Tame covers and splittings}
The commutator of $\overline{T}$ induces a bi-multiplicative map $[-,-]: T \times T \to \mu_n$, given explicitly on generators in \eqref{Rel3}. Consider the map
$$\varphi: T \longrightarrow \Hom(T_\kappa, \C^\times)$$
 given by 
 $$\varphi(t)(s):=[t, s']$$
 where $s' \in \mbf{T}(O_F)$ is any lifting of $s\in T_\kappa$ with respect to the reduction map ${\rm red}_\mfr{p}$. Since $T_1 = I_1 \cap \mbf{T}(O_F)$, which is the kernel of ${\rm red}_\mfr{p}$ restricted to $\mbf{T}(O_F)$, lies in the center of $\wt{T}$, we see that the map $\varphi$ is a well-defined group homomorphism.
 
Throughout this paper, we assume that $\wt{G}$ is a tame cover, meaning $p\nmid n$. For tame covers, one has ${\rm Ker}(\varphi)=Z(\overline{T}) \cdot \mbf{T}(O_F)$, since $p\nmid n$; as $\msc{X}_{Q,n} \simeq \wt{T}/Z(\wt{T})\mbf{T}(O_F)$, this gives a well-defined injective homomorphism
\begin{equation} \label{D:varphi}
\begin{tikzcd}
\varphi: \msc{X}_{Q,n} \ar[r, hook] & \Hom(T_\kappa, \C^\times).
\end{tikzcd}
\end{equation}
In explicit terms, $\varphi(\hat{y})(s) =[y(\varpi), s']$.

The covering group $\wt{G}$ splits over a subgroup $H\subset G$ if there is a group homomorphism $s:H\rightarrow \wt{G}$ such that $\wp\circ s=id_{H}$. If $\wt{G}$ splits over a subgroup $H\subset G$, then the set of all such splittings is a torsor over $\Hom(H, \mu_n)$. For tame covers, the group $\wt{G}$ splits over $K=\mbf{G}(O_F)$. In fact, since $\eta =\mbm{1}$ by our assumption, the $\mbf{K}_2$-extension $\wm{G}$ over $F$ arises from a $\mbf{K}_2$-extension of $\mbf{G}$ over $O_F$ (see \cite[Theorem 4.3]{We4}), which also entails the splitting of $K$. Note that for non-tame covers the groups $K$ and $I$ may not split. For examples of double covers of $G$ over $\Q_{2}$, see \cite{Kar21}. 

Throughout, we fix a splitting 
$$s_K: K \into \wt{G}.$$
To simplify notation, we may write $K$ instead of $s_K(K)$. The splitting $s_K$ gives rise to splittings of $I$ and $I_{1}$ by restriction. Since $I_{1}$ is a pro-$p$ group and $p\nmid n$, it follows that $\Hom(I_1,\mu_{n})=\set{1}$. Thus $I_1$ has a unique splitting afforded by $s_K|_{I_1}$. In contrast, for $I$ we have 
$$\Hom(I,\mu_{n})\simeq \Hom(I/I_{1},\mu_{n})\simeq  \Hom(T_\kappa, \mu_{n})\simeq (\mu_{n})^r,$$
where $r$ is the rank of the $\Z$-lattice $Y$ and the last isomorphism follows because $\mbf{T}$ is split.

We remark that the above contrast between the uniqueness of the splitting for $I_{1}$ and the non-uniquness for $I$ is partly responsible for the simpler descriptions and proofs of our results at the pro-$p$ Iwahori--Hecke algebra level than at the usual Iwahori--Hecke level.

\section{Pro-$p$ Iwahori--Hecke algebra $\mca{H}$} \label{S:proH}

In this section we will establish an Iwahori--Matsumoto presentation, following Vigneras \cite{Vig16}, for the integral $\epsilon$-genuine pro-$p$ Hecke algebra 
$$\mca{H}_{\Z}:=C^{\infty}_{c,\epsilon}(I_{1}\backslash \wt{G}/I_{1}, \Z[\mu_{n}]),$$
which consists of locally constant and compactly supported functions 
$$f: \wt{G} \rightarrow \Z[\mu_{n}]$$
such that $f(\gamma_{1}g\gamma_{2}\zeta)=\epsilon(\zeta)^{-1}f(g)$ for every $g\in \wt{G},\gamma_{j}\in I_{1}$, and $\zeta\in \mu_{n}$. The multiplication is given by convolution, i.e., for $f_{1},f_{2}\in \mca{H}_{\Z}$ one has
\begin{equation*}
f_{1}*f_{2}(g)=\int_{\wt{G}}f_{1}(h)f_{2}(h^{-1}g)dh=\int_{\wt{G}}f_{1}(gh)f_{2}(h^{-1})dh,
\end{equation*} 
where $dh$ is the Haar measure on the unimodular group $\wt{G}$ so that $\int_{I_{1}}dh=1$. We define 
$$\mca{H}:=\mca{H}_\Z \otimes_{\Z[\mu_{n}]}\C=C^{\infty}_{c,\epsilon}(I_{1}\backslash \wt{G}/I_{1},\C).$$

\subsection{Several groups}
We begin by describing a $\Z[\mu_{n}]$-basis for $\mca{H}_\Z$. This basis is in bijection with the group 
$$W(1)=N(T)/T_{1}\simeq \wt{N(T)}/\mu_{n} T_{1}.$$
Recall that for the extended affine Weyl group one has
$$W_{\rm ex}\simeq  N(T)/\mbf{T}(O_F),$$
where the affine reflection $w_{\alpha, k}$ corresponds to the class of $w_{\alpha}(\varpi^{-k}) \in N(T)$. This gives the exact sequence
$$\begin{tikzcd}
 T_\kappa \ar[r, hook] & W(1) \ar[r, two heads, "f"] & W_{\rm ex}.
\end{tikzcd}$$
For more details, see \cite[Page 696]{Vig16}. We have
$$\Omega\simeq (N_{G}(I)\cap N(T))/\mbf{T}(O_F) \simeq N_{G}(I)/I.$$

Since $W_\aff, \Omega \subset W_{\rm ex}$, we define
$$\Omega(1)=f^{-1}(\Omega),\quad W_\aff(1) = f^{-1}(W_\aff).$$
Consider $\wt{W}(1):=\wt{N(T)}/T_1$ with the natural quotient
$$\wp: \wt{W}(1) \onto W(1).$$
From this we define
$$\wt{\Omega}(1):=\wp^{-1}(\Omega(1)) = (f\circ \wp)^{-1}(\Omega),\quad  \wt{W}_\aff(1):= \wp^{-1}(W_\aff(1)) = (f\circ \wp)^{-1}(W_\aff).$$
The various groups above are illustrated in the following diagram
$$\begin{tikzcd}
&  \Omega(1) \ar[r, two heads, "f"] \ar[d, hook] & \Omega  \ar[d, hook] \\
\wt{\Omega}(1) \ar[d, hook] \ar[ru, two heads, "\wp"] & W(1) \ar[r, two heads, "f"] & W_{\rm ex} \\
\wt{W}(1) \ar[ru, two heads, "\wp"] & W_\aff(1) \ar[r, two heads, "f"] \ar[u, hook] & W_\aff  \ar[u, hook] \\
\wt{W}_\aff(1) \ar[u, hook] \ar[ru, two heads, "\wp"],
\end{tikzcd}$$
where $\Ker(f) = T_\kappa$ and $\Ker(\wp) = \mu_n$. Since $W_{\rm ex} =\Omega \ltimes W_\aff$, one has
$$ W(1)\simeq \Omega(1)\ltimes_{T_\kappa}W_\aff(1):= \frac{\Omega(1)\ltimes W_\aff(1)}{\nabla(T_\kappa)},$$
where $\nabla(t)=(t, t^{-1})$ is the anti-diagonal embedding. Similarly, there is an isomorphism
$$ \wt{W}(1)\cong\overline{\Omega}(1)\ltimes_{\mu_{n}T_\kappa}\wt{W}_\aff(1).$$
Note that the group $\wt{W}_\aff(1)$ is generated by the following set
$$\left(\mu_{n}\mbf{T}(O_F)/T_{1} \right) \bigcup \set{\wt{w}_{\alpha}(1)T_{1}: \alpha\in \Delta} \bigcup \set{ \wt{w}_{\alpha_j^\dag}(\varpi^{-1})T_{1}: 1\lest j \lest d}.$$

\begin{lm}\label{DoubCosetReps}
The natural map $W(1) \to \mu_{n}I_{1}\backslash\wt{G}/I_{1}$ is a bijection; similarly, one has a natural bijection $W_{\rm ex} \to \mu_{n}I\backslash\wt{G}/I$.
\end{lm}
\begin{proof}
This follows immediately from \cite[Proposition 3.35]{Vig16}.
\end{proof}

For $g\in \wt{G}$ let 
$$\mathcal{T}_{g} \in \mca{H}_\Z$$ be the unique element such that $\mathrm{supp}(\mathcal{T}_{g})=\mu_{n}I_{1}gI_{1}$ and $\mathcal{T}_{g}(g)=1$. Note that $\mathcal{T}_{g}$ is well-defined. To prove this it suffices to show that if 
$$g=\zeta \cdot s_K(\gamma_{1})gs_K(\gamma_{2})$$
with $\zeta\in \mu_{n}$ and $\gamma_{i}\in I_{1}$, then $\zeta=1$. Since any open subgroup of $I_1$ is a pro-$p$ group and thus has a unique splitting into $\wt{G}$, it follows that the two splittings
$$s_K, {}^gs_K: \ I_{1}\cap (g I_1 g^{-1}) \rightarrow\wt{G}$$
are equal. Therefore $\zeta=1$ and thus $\mca{T}_g$ is well-defined. By Lemma \ref{DoubCosetReps}, any section $\sigma$ of the map $\wt{N(T)}\onto W(1)$ yields a $\Z[\mu_n]$-basis $\{\mathcal{T}_{\sigma(w)}: w\in W(1)\}$ of $\mca{H}_\Z$. 

\subsection{Iwahori--Matsumoto presentation}
In this subsection we show that $\mca{H}_\Z$ admits an Iwahori--Matsumoto presentation. For any $g\in \wt{G}$ we set 
$$\mfr{q}_{g}=[I_{1}\wp(g)I_{1}:I_{1}].$$
By definition $\mfr{q}_{g}$ is constant on the double cosets of $\mu_{n}I_{1}$ inside $\wt{G}$. By reduction to the linear case as in \cite[Prop 6.2]{Sav04} and arguing as in \cite[\S 4.1]{Vig16}, the following holds:

\begin{lm}[Braid relations]
Let $g,g'\in \wt{G}$. If $\mfr{q}_{g}\cdot \mfr{q}_{g'}=\mfr{q}_{gg'}$, then $\mathcal{T}_{g}*\mathcal{T}_{g'}=\mathcal{T}_{gg'}$.
\end{lm}

Let 
$$\mca{H}_{\Z, 0} \subset \mca{H}_\Z$$ be the subalgebra of functions supported on $\overline{K}$. The splitting $s_{K}$ relates $\mca{H}_{\Z,0}$ to functions on the linear group $K$. 
The next lemma follows from a direct calculation.
\begin{lm}\label{HeckeSplit}
The map $s_K^\mca{H}: \mca{H}_{\Z, 0} \rightarrow C^\infty(I_{1}\backslash K/I_{1}, \Z[\mu_n])$ defined by $f\mapsto f \circ s_K$ is an isomorphism of algebras.
\end{lm}

Now we prove the quadratic relations. Suppose that $\Delta=\set{\alpha_1, ..., \alpha_r}$ is the set of simple roots and $\set{\alpha_j^\dag: 1\lest j \lest d}$ are the highest roots of the irreducible root systems $\Phi_{j}$. Let
$$\Delta_\aff=\Delta\cup\set{-\alpha_j^\dag: 1\lest j \lest d}.$$
For $\alpha\in \Phi$ and $\chi\in \Hom(\kappa^{\times},\mu_{n})$ we define $c_\alpha(\chi) \in \mca{H}$ to be such that
$$(q-1) \cdot c_{\alpha}(\chi)=\sum_{u\in\kappa^{\times}}\chi(u) \mca{T}_{\wt{h}_{\alpha(u)}}\in\mca{H}_\Z.$$
For every $\alpha \in \Phi$, we write $\mca{T}_\alpha:= \mca{T}_{\wt{w}_\alpha(1)}$ and let 
$$\varepsilon:=(\varpi,\varpi)_n=(-1,\varpi)_n \in \set{\pm 1}.$$

\begin{prop}[Quadratic relations] \label{QuadRel} 
Keep notation as above.
\begin{enumerate}
\item[(i)] For every $\alpha\in \Delta$ one has
\begin{equation} \label{Quadfin} 
\mathcal{T}_{\alpha}^{2}=q\mathcal{T}_{\wt{h}_{\alpha}(-1)}+(q-1)c_{\alpha}(\mbm{1})\mathcal{T}_{\alpha}.
\end{equation}
\item[(ii)] For any $\alpha^\dag \in \set{\alpha_1^\dag, ..., \alpha_d^\dag }$, write $\mathcal{T}_{\alpha^\dag}=\mathcal{T}_{\wt{w}_{\alpha^\dag}(\varpi^{-1})}$. Then 
\begin{equation} \label{Quadaff} 
\mathcal{T}_{\alpha^\dag}^{2}=q\varepsilon^{Q(\alpha^{\dag, \vee})}\mathcal{T}_{\wt{h}_{\alpha^\dag}(-1)}+(q-1)c_{\alpha^\dag}((-,\varpi)_n^{Q(\alpha^{\dag, \vee})})\mathcal{T}_{\alpha^\dag}.
\end{equation}
\end{enumerate}
\end{prop}
\begin{proof}
Both assertions can be verified directly. We elaborate only for \eqref{Quadaff}. For \eqref{Quadfin} we can reduce to the linear case by Lemma \ref{HeckeSplit} and then apply \cite[Theorem 2.2]{Vig16}. 

To begin we bound ${\rm supp}(\mca{T}_{\alpha^\dag}^{2})$. The computation reduces to the linear case, which gives 
\begin{equation*}
{\rm supp}(\mca{T}_{\alpha^\dag}^{2})  \subseteq \mu_{n}\wt{h}_{\alpha^\dag}(-1)I_{1} \cup \bigcup_{u\in O_F^{\times}/1+\frak{p}}\mu_{n}I_{1}\wt{h}_{\alpha^\dag}(u)\wt{w}_{\alpha^\dag}(\varpi^{-1})I_{1}.
\end{equation*}
Now we evaluate $\mathcal{T}_{\alpha^\dag}^{2}$ at the elements $\wt{h}_{\alpha^\dag}(-1)$  and $\wt{w}_{\alpha^\dag}(\varpi^{-1})\wt{h}_{\alpha^\dag}(v)$ where  $v\in O_F^{\times}/1+\frak{p}$.

A preliminary calculation gives
\begin{equation} \label{QuadEqn1}
\begin{aligned}
\mca{T}_{\alpha^\dag}^{2}(g) = &\int_{\wt{G}}\mca{T}_{\alpha^\dag}(h) \cdot \mca{T}_{\alpha^\dag}(h^{-1}g)dh \\
=&\sum_{\gamma\in \mu_{n}I_{1}/\mu_{n}I_{1}\cap \wt{w}_{\alpha^\dag}(\varpi^{-1})I_{1}\wt{w}_{\alpha^\dag}(\varpi^{-1})^{-1}} \mathcal{T}_{\alpha^\dag}(\gamma \wt{w}_{\alpha^\dag}(\varpi^{-1}))\mathcal{T}_{\alpha^\dag}(\wt{w}_{\alpha^\dag}(\varpi^{-1})^{-1}\gamma^{-1}g)dh \\
= & \sum_{\gamma\in I_{1}/I_{1}\cap w_{\alpha^\dag}(\varpi^{-1})I_{1}w_{\alpha^\dag}(\varpi^{-1})^{-1}} \mathcal{T}_{\alpha^\dag}(\gamma \wt{w}_{\alpha^\dag}(\varpi^{-1}))\mathcal{T}_{\alpha^\dag}(\wt{w}_{\alpha^\dag}(\varpi^{-1})^{-1}\gamma^{-1}g)dh \\
= & \mathcal{T}_{\alpha^\dag}(\wt{w}_{\alpha^\dag}(\varpi^{-1})^{-1}g)+\sum_{u\in O_F^{\times}/1+\frak{p}} \mathcal{T}_{\alpha^\dag}(\wt{w}_{\alpha^\dag}(\varpi^{-1})^{-1}\wt{e}_{-\alpha^\dag}(\varpi u)g),
\end{aligned}
\end{equation}
where the last equality follows from $I_{1}/I_{1}\cap \wt{w}_{\alpha^\dag}(\varpi^{-1})I_{1}\wt{w}_{\alpha^\dag}(\varpi^{-1})^{-1}= U_{-\alpha^\dag}^{1}/U_{-\alpha^\dag}^{2}$.

It follows from \eqref{Rel2} and \eqref{Rel7} that
$$\wt{w}_{\alpha^\dag}(\varpi^{-1})^{-1}\wt{h}_{\alpha^\dag}(-1)=\varepsilon^{Q(\alpha^{\dag, \vee})} \cdot \wt{w}_{\alpha^\dag}(\varpi^{-1})$$ 
and $\wt{e}_{-\alpha^\dag}(\varpi u)\in I_{1}$ for all $u\in O_F^\times$. Thus from \eqref{QuadEqn1} we have 
\begin{equation*}
\mathcal{T}_{\alpha^\vee}^{2}(\wt{h}_{\alpha^\dag}(-1))=q\varepsilon^{Q(\alpha^{\dag, \vee})}.
\end{equation*}

Next, we consider the element $\wt{w}_{\alpha^\dag}(\varpi^{-1})\wt{h}_{\alpha^\dag}(v)$, where $v\in O_F^{\times}/1+\frak{p}$. For $u\in O_F^{\times}$, the left and right $U_{-\alpha^\dag}^1$-invariance of $\mathcal{T}_{\alpha^\dag}$ and the equations \eqref{Rel2}, \eqref{Rel6} and \eqref{Rel7} imply
$$\mathcal{T}_{\alpha^\dag}(\wt{w}_{\alpha^\dag}(\varpi^{-1})^{-1}\wt{e}_{-\alpha^\dag}(\varpi u)\wt{w}_{\alpha^\dag}(\varpi^{-1})\wt{h}_{\alpha^\dag}(v))=(-u,\varpi)_n^{Q(\alpha^{\dag, \vee})} \cdot \mathcal{T}_{\alpha^\dag}(\wt{w}_{\alpha^\dag}(\varpi^{-1})\wt{h}_{\alpha^\dag}(-u^{-1}v)).$$
Thus by line \eqref{QuadEqn1} we obtain
\begin{equation*}
\mathcal{T}_{\alpha^\dag}^{2}(\wt{w}_{\alpha^\dag}(\varpi^{-1})\wt{h}_{\alpha^\dag}(v))=(v,\varpi)_n^{Q(\alpha^{\dag,\vee})},
\end{equation*}
which completes the proof.
\end{proof}

The braid relations and the quadratic relations imply the following.
\begin{cor}\label{PropInvert}
For every $g\in \wt{G}$, the element $\mca{T}_g \in \mca{H}_\Z$ is invertible.
\end{cor}

Setting 
$$\wt{N}(I_1):=\wt{N_{G}(I_{1})},$$
a direct calculation and comparison of measures give that
$$\wt{N}(I_1)=\set{g\in \wt{G}: \mfr{q}_g=1}.$$
It also follows that $\wm{T}(O_F) \subset \wt{N}(I_1)$. 
Setting 
$$e_{\epsilon}=\sum_{\zeta\in\mu_{n}}\epsilon^{-1}(\zeta)\zeta\in \C[\mu_{n}],$$
which is the idempotent associated to $\epsilon$.
The map 
$$\Z[\mu_{n},\wt{N}(I_1)/I_{1}]e_\epsilon \longrightarrow \mca{H}_\Z$$
given by $gI_1 e_\epsilon \mapsto \mca{T}_g$ is an embedding of $\Z[\mu_n]$-algebras. Let 
$$\mca{H}_\Z^\aff \subset \mca{H}_\Z$$
 be the subalgebra  generated by  
 $$A:=\Z[\mu_{n}, \wm{T}(O_F)/T_{1}]e_{\epsilon}$$
  and $\set{\mathcal{T}_{\alpha}: \alpha\in \Delta}$, and also $\set{\mca{T}_{\alpha_j^\dag}: 1\lest j \lest d}$. We have
$$\mca{H}^\aff_\Z = \langle \mca{T}_{w}: w\in \wt{W}^\aff(1) \rangle,$$
as an algebra generated over $A$. Indeed, this follows from support considerations and also the inclusion $wIs \subset (IwI\cup IwsI)$, where $w\in W_\aff$ and $s$ is one of the generators of $W_\aff$ as a Coxeter group.

To describe a presentation for $\mca{H}^\aff_\Z$, let $\msc{F}'$ denote the free $\Z[\mu_{n}]$-algebra generated by the symbols $\msc{T}_{\alpha}$, where $\alpha\in \Delta_\aff$. Let
 $$\msc{F}=A*\msc{F}'$$
 be the free product of the associative $\Z[\mu_{n}]$-algebras $A$ and $\msc{F}'$. For $\alpha,\beta\in\Delta_\aff$, write $\theta_{\alpha\beta}=\theta_{\beta\alpha}$ for the angle between the roots $\alpha$ and $\beta$. Let $\msc{I} \subset \msc{F}$ be the ideal generated by the following elements:
\begin{enumerate}
\item[$\bullet$] $\msc{T}_{\alpha}\msc{T}_{\beta}-\msc{T}_{\beta}\msc{T}_{\alpha}$, for any $\alpha,\beta\in \Delta_\aff $ such that $\theta_{\alpha\beta}=\frac{\pi}{2}$;
\item[$\bullet$] $\msc{T}_{\alpha}\msc{T}_{\beta}\msc{T}_{\alpha}-\msc{T}_{\beta}\msc{T}_{\alpha}\msc{T}_{\beta}$, for any $\alpha,\beta\in \Delta_\aff$ such that $\theta_{\alpha\beta}=\frac{2\pi}{3}$;
\item[$\bullet$] $(\msc{T}_{\alpha}\msc{T}_{\beta})^{2}-(\msc{T}_{\beta}\msc{T}_{\alpha})^{2}$, for any $\alpha,\beta\in \Delta_\aff$ such that $\theta_{\alpha\beta}=\frac{3\pi}{4}$;
\item[$\bullet$] $(\msc{T}_{\alpha}\msc{T}_{\beta})^{3}-(\msc{T}_{\beta}\msc{T}_{\alpha})^{3}$, for any $\alpha,\beta\in \Delta_\aff$ such that $\theta_{\alpha\beta}=\frac{4\pi}{5}$;
\item[$\bullet$] $\msc{T}_{\alpha}h-(w_{\alpha}(1)hw_{\alpha}(-1))\msc{T}_{\alpha}$, for any $\alpha\in \Delta$ and $h\in T_\kappa$;
\item[$\bullet$] $\msc{T}_{\alpha^\dag}h-(w_{\alpha^\dag}(\varpi^{-1})hw_{\alpha^\dag}(-\varpi^{-1}))\msc{T}_{\alpha^\dag}$ for any $-\alpha^\dag \in\Delta_\aff-\Delta$ and $h\in T_\kappa$;
\item[$\bullet$] $\msc{T}_{\alpha}^{2}-qh_{\alpha}(-1)-(q-1)c_{\alpha_{j}}(1)\msc{T}_{\alpha}$, for any $\alpha\in \Delta$;
\item[$\bullet$] $\msc{T}_{\alpha^\dag}^{2}-q\varepsilon^{Q(\alpha^{\dag, \vee})}h_{\alpha^\dag}(-1)-(q-1)c_{\alpha^\dag}((-,\varpi)^{Q(\alpha^{\dag, \vee})})\msc{T}_{\alpha^\dag}$, for any $-\alpha^\dag \in\Delta_\aff -\Delta$.
\end{enumerate}

\begin{prop}\label{P:IM1}
The map of $\Z[\mu_n]$-algebras 
$$\msc{F}'\rightarrow \mca{H}^\aff_\Z$$
given by 
$$\msc{T}_{\alpha}\mapsto \mathcal{T}_{\wt{w}_{\alpha}(-1)}, \quad \msc{T}_{\alpha^\dag}\mapsto \mathcal{T}_{\wt{w}_{\alpha^\dag}(\varpi^{-1})} \text{ and } he_\epsilon \mapsto \mca{T}_h$$
for $\alpha\in \Delta,  -\alpha^\dag \in \Delta_\aff -\Delta$, and $h e_\epsilon \in A$,  induces an isomorphism of $\Z[\mu_{n}]$-algebras $\msc{F}/\mathscr{I} \rightarrow \mca{H}^\aff_\Z$. 
\end{prop}
\begin{proof}
The proof of \cite[Theorem 3.5]{IM65} by Iwahori--Matsumoto  adapts directly to this case.
\end{proof}

Recall the definition of a twisted tensor product of algebras. Let $\Gamma$ be a  group, $L$ a commutative ring, and $B$ a ring  that is a free $L$-module with $L$-basis $\set{b_j: j\in J}$. Suppose that 
$$\phi: \Gamma \longrightarrow {\rm Aut}_{\text{$L$-alg}}(B)$$
is a group homomorphism from $\Gamma$ into the $L$-algebra automorphisms of $B$. We define the twisted tensor product $L[\Gamma]\hat{\motimes}_{L}B$ to be the free $L$-module with basis $\{\gamma\otimes b_j: \gamma\in \Gamma,\, j\in J \}$ and multiplication defined by
\begin{equation*}
(\gamma_{1}\otimes b_{j_{1}})\cdot (\gamma_{2}\otimes b_{j_{2}})\mapsto (\gamma_{1}\gamma_{2}\otimes \phi_{\gamma_{2}}(b_{j_1})b_{j_2}).
\end{equation*}

\begin{thm}\label{T:IM2}
The map 
$$\Z[\mu_{n},\wt{N}(I_1)/T_1]e_{\epsilon}\hat{\motimes}_{A}\mca{H}^\aff_\Z \longrightarrow \mca{H}_\Z$$
defined by $g \otimes \mca{T}_{w}\mapsto \mca{T}_{gw}=\mca{T}_{g}*\mca{T}_{w}$, where $g\in \wt{N}(I_{1})$ and $w\in \wt{W}(1)$ is an isomorphism of $\Z[\mu_{n}]$-algebras. Together with Proposition \ref{P:IM1}, this gives an Iwahori--Matsumoto presentation of the pro-$p$ Iwahori--Hecke algebra $\mca{H}_\Z$.
\end{thm}

\section{Universal pro-$p$ principal series and $\mca{H}$}  \label{S:proPS}

In this section we prove that $\mca{H}=C^{\infty}_{c,\epsilon}(I_{1}\backslash \wt{G}/I_{1},\C)$ admits a Bernstein presentation. Our proof is based on the universal principal series, following Haines--Kottwitz--Prasad \cite{HKP10}.

\subsection{A first reduction}
Consider the space
$$\mca{M}:=C^{\infty}_{c,\epsilon}(T_{1}U\backslash \wt{G}/I_{1}, \C)$$
consisting of locally constant functions $f: \wt{G}\rightarrow \C $ such that 
$$f(\zeta tug\gamma)=\epsilon(\zeta)f(g) \text{ for } \zeta\in \mu_{n},t\in T_{1}, u\in U, \gamma\in I_{1}$$ and is also compactly supported on $U\backslash \wt{G}$. Let
$$\mca{R}':=\C[\wt{T}/T_{1}].$$

Here $\mca{M}$ is an $(\mca{R}',\mca{H})$-bimodule, where the right action of $\mca{H}$ is given by convolution of functions. More specifically, given $F\in \mca{M}$ and $f\in \mca{H}$ we define
\begin{equation*}
F*f(g):=\int_{\wt{G}}F(h)f(h^{-1}g)dh=\int_{\wt{G}}F(gh)f(h^{-1})dh,
\end{equation*}
where $dh$ is the Haar measure of $\wt{G}$ normalized so that the measure of $I_{1}$ is 1. There is also a natural left action of $\wt{T}$ on $\mca{M}$ given by (for $t\in \overline{T}$ and $F\in M$) 
\begin{equation*}
(t\cdot F)(g):=\delta_{\overline{B}}(t)F(t^{-1}g),
\end{equation*} 
where $\delta_{\overline{B}}$ is the modular character of $\wt{B}$. Since $T_{1}$ is central in $\overline{T}$, a consequence of $\gcd(p,n)=1$, this left action descends to $\wt{T}/T_{1}$ and can be extended $\C$-linearly to an action of $\mca{R}'$.  Since the functions in $\mca{M}$ are $\epsilon$-genuine it suffices to work with 
$$\mca{R}:=\mca{R}'e_\epsilon.$$
Note that $\mca{R}$ has a $\C$-basis indexed by $T/T_1$.

Consider the natural map $N(T)\rightarrow T_{1}U\backslash G/I_{1}$. In view of the Iwasawa decomposition $G=UTK$ and the Bruhat decomposition for $K$ with respect to $I_{1}$, it induces a bijection 
$$W(1)=T_{1}\backslash N(T) \longrightarrow T_{1}U\backslash G/I_{1}.$$
For any $g\in\wt{G}$, define
$$\nu_g  \in \mca{M}$$ to be the unique function such that ${\rm supp}(\nu_{g})=\mu_{n} T_{1} Ug I_1$ and $\nu_{g}(g)=1$. 

\begin{lm}
The function $\nu_g\in \mca{M}$ is well-defined.
\end{lm}
\begin{proof}
For $g=k\in s_K(K)$ this follows from a direct computation. Now for $k\in s_K(K)$ and $t\in \wt{T}$ we consider 
$$\nu_{tk}':=t\cdot \nu_k.$$
The function $\nu_{tk}'$ has support $\mu_{n}T_{1}UtkI_{1}$ and $\nu_{tk}'(tk)\neq 0$. Since the set $\overline{T} s_K(K)$ contains a set of representatives for $\mu_{n}T_{1}U\backslash \wt{G}/I_{1}$ the functions $\nu_g$ for any $g\in \wt{G}$ are well defined, being a nonzero scalar multiple of $\nu_{tk}'$, for some $t\in\wt{T}$ and $k\in s_K(K)$.
\end{proof}

Fix a section $\sigma$ of the map $\overline{N(T)} \rightarrow W(1)$ that sends the identity to the identity. For each $w\in W(1)$ define 
$$\nu_w=\nu_{\sigma(w)}\in \mca{M}.$$
We see that the set $\set{\nu_w: w\in W(1)}$ constitutes a $\C$-basis for $\mca{M}$.

\begin{lm}\label{HMIso}
The map 
$$\mca{H} \longrightarrow \mca{M}$$
 given by $h\mapsto \nu_{1}*h$ is an isomorphism of $\mca{H}$-modules. Furthermore, this induces an isomorphism of rings 
 $$\mca{H} \longrightarrow {\rm End}_{\mca{H}}(\mca{M})$$
   via the map $h\mapsto (\nu_{1}*h'\mapsto \nu_{1}*h*h')$.
\end{lm}
\begin{proof} 
Given a $\wt{G}\times\wt{G}$-module $A$, we write $_{U}A$ for the space of left $U$-coinvariants and $^{T_{1}}A$ for the space of left $T_{1}$-invariants.

The assertions follow from Haines--Kottwitz--Prasad \cite[Lemma 1.6.1]{HKP10}, since the Bruhat ordering on $W(1)$ is inflated from the underlying affine Weyl group (see \cite[Page 697]{Vig16}). 

Nevertheless, we give another proof suggested by Gordan Savin based on Bushnell--Kutzko theory. Consider $C^{\infty}_{c,\epsilon}(\wt{G})$ as a $\wt{G}\times \wt{G}$-module with left and right translation. By Lemma \ref{PropInvert} the elements $\mathcal{T}_{t}, t\in\overline{T}$, are invertible. Thus by Bushnell--Kutzko theory \cite{BK98}, we have an isomorphism of $\wt{G}$-modules 
$$C^{\infty}_{c,\epsilon}(I_{1}\backslash \wt{G})\simeq {_{U}^{T_{1}}C^{\infty}_{c,\epsilon}}(\wt{G}).$$
Integrating on the left over $U$ gives that ${_{U}^{T_{1}}C^{\infty}_{c,\epsilon}}(\wt{G})\simeq C^{\infty}_{c,\epsilon}(UT_{1}\backslash\wt{G})$. Thus we have 
$$C^{\infty}_{c,\epsilon}(I_{1}\backslash \wt{G})\simeq C^{\infty}_{c,\epsilon}(UT_{1}\backslash\wt{G})$$
 as $\overline{T}\times \wt{G}$-modules. Taking right $I_{1}$-invariants, we get $\mca{H} \simeq \mca{M}$ as right $\mca{H}$-modules. Moreover, integration on the left over $U$ sends the identity element of $\mca{H}$ to $\nu_1\in \mca{M}$. This completes the proof.
\end{proof}

Next we record three useful identities. An element $t\in\overline{T}$ is called dominant if for any $l \in \Z$ we have $tU_{\alpha}^l t^{-1}\subseteq U_{\alpha}^l$ for all $\alpha\in \Phi_+$. 

\begin{lm}\label{ModId} We have the following:
\begin{enumerate}
\item[(i)] $\nu_1*\mca{T}_w=\nu_w$ for any $w\in K\cap N(T)$;
\item[(ii)] $\nu_t*\mca{T}_w=\nu_{tw}$ for any $t\in \wt{T}/T_1$ and $w\in K\cap N(T)$;
\item[(iii)] $\nu_1*\mca{T}_t=\nu_{t}$ for any $t\in \wt{T}/T_1$ dominant.
\end{enumerate}
\end{lm}
\begin{proof}
We begin with two general calculations. First, for functions $f_{j}$ on the group $\wt{G}$, we have 
\begin{equation}\label{consupp}
\mathrm{supp}(f_{1}*f_{2})\subseteq \mathrm{supp}(f_{1})\cdot \mathrm{supp}(f_{2}).
\end{equation}
Second, for $\mathcal{T}_{g}\in\mca{H}$ and $\nu \in \mca{M}$ we have 
\begin{equation}  \label{ModIdEqn1}
\begin{aligned} 
\nu*\mathcal{T}_{g}(g')& =\int_{\wt{G}} \nu(g'h^{-1})\mathcal{T}_{g}(h)dh \\
&= \int_{\mu_{n} I_{1}gI_{1}} \nu(g'h^{-1})\mathcal{T}_{g}(h)dh \\
&=\sum_{\mu_{n}I_{1}\backslash\mu_{n}I_{1}gI_{1}} \nu(g'h^{-1})\mathcal{T}_{g}(h) \\
&= \sum_{h\in (I_{1}\cap g^{-1}I_{1}g)\backslash I_{1}} \nu(g'h^{-1}g^{-1})\mathcal{T}_{g}(gh).
\end{aligned}
\end{equation}

Now we prove (i). By \eqref{consupp} we have $\mathrm{supp}(\nu_1*\mca{T}_w)\subseteq \mu_{n}UI_{1}wI_{1}$. We use the Iwahori factorization and $w^{-1}(I_{1}\cap U^-)w\subseteq I_{1}$ to see that $\mu_{n}UI_{1}wI_{1}=\mu_{n}UwI_{1}$. Thus $\nu_1*\mca{T}_{w}=c \cdot \nu_w$ for some $c \in \C$. Using \eqref{ModIdEqn1} we compute $\nu_1*\mca{T}_{w}(w)=1$, thus $\nu_{1}*\mca{T}_{w}= \nu_{w}$. Specifically, 
\begin{equation*}
\nu_{1}*\mca{T}_{w}(w)=\sum_{(I_{1}\cap w^{-1}I_{1}w)\backslash I_{1}}  \nu_{1}(wh^{-1}w^{-1})\mathcal{T}_{w}(wh)=1.
\end{equation*}
The last equality follows because $(I_{1}\cap w^{-1}I_{1}w)\backslash I_{1}\cong \prod_{\alpha\in \Phi_{+}\cap w^{-1}\Phi_{-}}U_{\alpha}^{1}\backslash U_{\alpha}^{0}$ and the factorization of $UI_{1}$ into $U(T\cap I_{1})(U^{\mathrm{op}}\cap I_{1})$ is unique.

Item (ii) follows from item (i) and $\nu_t=\delta_{\overline{B}}^{-1}(t)(t\cdot \nu_1)$. For (iii), by \eqref{consupp} we have
\begin{equation*}
\mathrm{supp}(\nu_1*\mca{T}_t)\subseteq \mu_{n}T_{1}UI_{1}tI_{1}
=\mu_{n}T_{1}UU^-(\frak{p})tI_{1}
\subseteq \mu_{n}T_{1}UtI_{1},
\end{equation*}
where the last containment follows because $t$ is dominant. Thus it suffices to show that $\nu_1*\mathcal{T}_{t}(t)=1$. Using \eqref{ModIdEqn1} we have 
\begin{equation*}
\nu_{1}*\mathcal{T}_{t}(t)=\sum_{(I_{1}\cap t^{-1}I_{1}t)\backslash I_{1}} \nu_{1}(th^{-1}t^{-1})\mathcal{T}_{t}(th)=1.
\end{equation*}
The last equality follows since $(I_{1}\cap t^{-1}I_{1}t)\backslash I_{1}\simeq \prod_{\alpha\in \Phi_{-}}U_{\alpha}^{1+\ell_{t,\alpha}}\backslash U_{\alpha}^{1}$, where $\ell_{t,\alpha}\in\Z_{\gest 0}$.
\end{proof}

\begin{lm}\label{TInj}
The $\C$-algebra map $\mca{R} \rightarrow {\rm End}_{\mca{H}}(\mca{M}) \simeq \mca{H}$ defined by $r\mapsto (m\mapsto r\cdot m)$ is an injection.
\end{lm}

\begin{proof}
Let $t\in \overline{T}$. Then $t\cdot \nu_1$ is a function supported on $\mu_{n}T_{1}UtI_{1}$. Any two elements of $\overline{T}$ have distinct supports unless they are in the same $\mu_{n}T_{1}$ coset.
\end{proof}
Using Lemmas \ref{HMIso} and \ref{TInj}, we can associate to any $t\in\overline{T}/T_{1}$ an element 
$$\Theta_{t}\in \mca{H},$$
defined by $t\cdot \nu_{1}=\nu_{1}*\Theta_{t}$.  Let 
$$\mca{H}_\kappa :=\mca{H}_{\Z, 0}\otimes_{\Z[\mu_n]} \C = C_{c, \epsilon}^\infty(I_1 \backslash \wt{K} /I_1)$$
be the subalgebra of $\mca{H}$ consisting of functions with support contained in $\overline{K}$. For any $\chi\in \Hom(T_\kappa,\C^\times)$ we define 
$$c(\chi)=\frac{1}{\val{T_\kappa}} \sum_{t\in T_\kappa} \chi(t)\mathcal{T}_{t}\in\mca{H}_\kappa.$$
The next lemma contains some useful identities, which follow from the braid relations in $\mca{H}_\Z$. Note that the Weyl group $W$ acts on $\Hom(T_\kappa, \C^\times)$ through its action on $T_\kappa$.

\begin{lm}\label{BasicHeckeRel}
The following holds.
\begin{enumerate}
\item[(i)] As a $\C$-algebra $\mca{H}_\kappa$ is generated by the elements $\mathcal{T}_{\alpha}$, where $\alpha\in \Delta$, and $c(\chi)$, where $\chi\in\Hom(T_\kappa, \C^{\times})$.
\item[(ii)] Let $\alpha\in \Delta$ and $\chi\in\Hom(T_\kappa, \C^\times)$. Then
\begin{equation*}
\mathcal{T}_{\alpha}*c(\chi)=c({}^{w_\alpha}\chi)*\mathcal{T}_{\alpha}.
\end{equation*}
\item[(iii)] Let $t\in\overline{T}$ and $\chi\in \Hom(T_\kappa,\C^\times)$. Then 
\begin{equation*}
\Theta_{t}*c(\chi)=c(\varphi(t)\cdot\chi)*\Theta_{t}.
\end{equation*}
\end{enumerate}
\end{lm}

We end this section with a Bernstein decomposition.

\begin{prop}\label{BernDecomp}
Multiplication in $\mca{H}$ induces a vector space isomorphism
\begin{equation*}
\mca{R} \otimes_{\mca{R} \cap \mca{H}_\kappa} \mca{H}_\kappa \rightarrow \mca{H}
\end{equation*}
sending $te_{\epsilon}\otimes h$ to $\Theta_{t}*h$. Furthermore, the same is true with the factors reversed.
\end{prop}

\begin{proof}
The result follows from  \cite[Lemma 1.7.1]{HKP10}. Consider the composite of the above map  with the isomorphism $\mca{H} \longrightarrow \mca{M}$ from Lemma \ref{HMIso}:
\begin{equation*}
\mca{R} \otimes_{\mca{R} \cap \mca{H}_\kappa} \mca{H}_\kappa \longrightarrow \mca{H} \longrightarrow \mca{M}.
\end{equation*}
For $t\in \overline{T}/T_{1}$ and $w\in N(T) \cap K$, this map sends $te_{\epsilon}\otimes \mathcal{T}_{w}$ to $\nu_1*\Theta_{t}*\mathcal{T}_{w}$. By Lemma \ref{ModId}, $\nu_1*\Theta_{t}*\mathcal{T}_{w}=C \cdot \nu_{tw}$, where $C \in \C^\times$. Thus the map is a $\C$-isomorphism.
\end{proof}

\subsection{Decomposing $\mca{H}$}
In this subsection we describe the $\C$-algebra structure of $\mca{H}$ as a product of matrix rings, see Proposition \ref{MatProd} below.

To begin, we decompose $\mca{H}$ as a right $\mca{H}$-module using the idempotents $c(\chi) \in \C[T_\kappa]$. The decomposition $\C[T_\kappa]=\bigoplus_\chi c(\chi) \C[T_\kappa]$, where the direct sum is taken over $\chi\in \Hom(T_\kappa,\mu_{q-1})$, immediately gives the following:

\begin{lm}\label{HeckeIdemDecomp}
One has a decomposition
$$\mca{H} \simeq \bigoplus_{\chi} c(\chi)\mca{H},$$
where $\chi$ is taken over elements in the set $\Hom(T_\kappa,\mu_{q-1})$.
\end{lm}
We view $\mca{H}$ as a left $\Hom_{\mca{H}}(\mca{H}, \mca{H})$-module. In particular, the map 
$$\mca{H} \longrightarrow \Hom_{\mca{H}}(\mca{H},\mca{H})$$
defined by $h\mapsto (h'\mapsto h*h')$ gives an isomorphism of rings. For $\chi\in \Hom(T_\kappa,\mu_{q-1})$ we set
$$\mca{H}_{\chi}:=c(\chi)\mca{H} c(\chi).$$

\begin{lm} \label{HeckeIdemHom}
Let $\chi,\chi'\in \Hom(T_\kappa,\mu_{q-1})$. Then the map 
$$c(\chi')\mca{H} c(\chi)\longrightarrow \Hom_{\mca{H}}(c(\chi)\mca{H}, c(\chi')\mca{H})$$
defined by $c(\chi')hc(\chi)\mapsto (f\mapsto c(\chi')hc(\chi)*f)$ is an isomorphism of $(\mca{H}_{\chi'},\mca{H}_{\chi})$-bimodules.
\end{lm}
\begin{proof}
There is an embedding $\Hom_{\mca{H}}(c(\chi)\mca{H}, c(\chi')\mca{H})\hookrightarrow \Hom_{\mca{H}}(\mca{H}, \mca{H})\simeq \mca{H}$. It is simple to check that the image of this embedding is $c(\chi')\mca{H} c(\chi)$.
\end{proof}

There is an action of $W$ on $\Hom(T_\kappa,\mu_{q-1})$ via the action of $W$ on $T_\kappa$. There is also an action of $\msc{X}_{Q,n} = Y/Y_{Q,n}$ on $\Hom(T_\kappa,\mu_{q-1})$ via the embedding 
\begin{equation}
\begin{tikzcd}
\varphi: \msc{X}_{Q,n} \ar[r, hook] & \Hom(T_\kappa,\mu_{n}).
\end{tikzcd}
\end{equation}
These actions are compatible in the sense that they give rise to a well-defined action of  $W \ltimes \msc{X}_{Q,n}$ on $\Hom(T_\kappa, \mu_{q-1})$.

\begin{lm}\label{HeckeIdemnon0}
The space $c(\chi')\mca{H} c(\chi)$ is nonzero if and only if $\chi$ and $\chi'$ are in the same $W\ltimes \msc{X}_{Q,n}$-orbit.
\end{lm}
\begin{proof}
Let $\wt{w}\in \overline{N(T)}$ represent $w\in W\ltimes Y$. Let $t\in \mbf{T}(O_F)$. The result follows from the identity 
$$\mathcal{T}_{\wt{w}}\mathcal{T}_{t}=\mathcal{T}_{wtw^{-1}}\mathcal{T}_{\wt{w}},$$
which is due to the fact that $\mbf{T}(O_F) \subset N_G(I_1)$. Specifically, $\mathcal{T}_{\wt{w}}c(\chi)=c({}^w\chi)\mathcal{T}_{\wt{w}}$. Here $w$ is acting on $\chi$ through its image in $W\ltimes \msc{X}_{Q,n}$.
\end{proof}

\begin{cor}\label{HeckeIdemIso}
The module $c(\chi)\mca{H}$ is isomorphic to $c(\chi')\mca{H}$ if and only if $\chi$ and $\chi'$ are in the same $W\ltimes \msc{X}_{Q,n}$-orbit.
\end{cor}
\begin{proof}
By Corollary \ref{PropInvert}, every $\mathcal{T}_{g}$ is invertible. Thus, if $c(\chi')\mathcal{T}_{g}c(\chi)\neq 0$, then $c(\chi')\mathcal{T}_{g}c(\chi)=c(\chi')\mathcal{T}_{g}=\mathcal{T}_{g}c(\chi)$, and in fact $c(\chi)\mathcal{T}_{g}^{-1}c(\chi')=c(\chi)\mathcal{T}_{g}^{-1}=\mathcal{T}_{g}^{-1}c(\chi')$ is its inverse.
\end{proof}

Given a ring $S$ and $k\in\Z_{\gest 1}$, we denote by $\mbf{M}(k,S)$ be the ring of $k\times k$ matrices with entries in $S$.

\begin{prop} \label{HeckeRingDecomp}
There is an isomorphism of rings
\begin{equation*}
\mca{H} \simeq \prod_{_{\mathcal{O}_{\chi}\subset \Hom(T_\kappa,\mu_{q-1})}} \mbf{M}(\val{\mca{O}_\chi},\mca{H}_{\chi}).
\end{equation*}
The product is taken over the $W\ltimes  \msc{X}_{Q,n}$-orbits.
\end{prop}
\begin{proof} 
By Lemmas \ref{HeckeIdemDecomp}, \ref{HeckeIdemHom}, and \ref{HeckeIdemnon0} and Corollary \ref{HeckeIdemIso} we have
\begin{align*}
\mca{H} & \simeq \Hom_{\mca{H}}(\mca{H}, \mca{H}) \\
& \simeq  \Hom_{\mca{H}}\Big(\bigoplus_{\chi\in \mathcal{O}}c(\chi)\mca{H},\bigoplus_{\chi'\in\mathcal{O}}c(\chi')\mca{H} \Big)\\
& \simeq \prod_{\mathcal{O}_{\chi}}\Hom_{\mca{H}}(c(\chi)\mca{H}^{\val{\mca{O}_\chi}}, c(\chi)\mca{H}^{\val{\mca{O}_\chi}}),
\end{align*}
where the product is taken over the $W\ltimes \msc{X}_{Q,n}$-orbits in $\Hom(T_\kappa,\mu_{q-1})$. Finally, since
\begin{equation*}
\Hom_{\mca{H}}(c(\chi)\mca{H}^{\val{\mca{O}_\chi}}, c(\chi)\mca{H}^{\val{\mca{O}_\chi}}) \simeq \mbf{M}(\val{\mathcal{O}}, \Hom_{\mca{H}}(c(\chi)\mca{H} , c(\chi)\mca{H} )),
\end{equation*}
where $\Hom_{\mca{H}}(c(\chi)\mca{H}, c(\chi)\mca{H})\simeq \mca{H}_{\chi}$ by Lemma \ref{HeckeIdemHom}, the result follows.
\end{proof}

\subsection{The case of covering tori} 
The results of the preceding subsection can be refined in the case $G=T$. 

\begin{lm}\label{HeckeTorus3} 
Let $\chi,\chi'\in\Hom(T_\kappa, \C^\times)$.
\begin{enumerate}
\item[(i)] The center of $\mca{R}$ is equal to $\C[Z(\overline{T})/T_{1}]e_{\epsilon}$.
\item[(ii)] Suppose that $y_{0}\in Y$ satisfies $\chi=\varphi(y_{0})\cdot \chi'$. Then 
$$c(\chi)\mca{R} c(\chi')=Z(\mca{R})c(\chi) \Theta_{\s_{y_0}}c(\chi').$$
In particular, if $\chi=\chi'$, then $c(\chi)\mca{R} c(\chi)=Z(\mca{R})c(\chi)$. Hence, $c(\chi) \mca{R} c(\chi)\simeq Z(\mca{R})c(\chi)$ as rings.
\end{enumerate}
\end{lm}
\begin{proof}
For (i),  let $\set{f}$ be a $\C$-basis for $\C[\wt{T}/T_{1}]e_{\epsilon}$ given by a choosing a set of representatives for the coset space $\wt{T}/\mu_{n}T_{1}$. Suppose that $\sum c_f f$ lies in the center of $\C[\wt{T}/T_{1}]e_\epsilon$. Then for any $t\in \overline{T}/T_{1}$ we have 
$$t\Big(\sum c_{f}f\Big)=\Big(\sum c_{f}f\Big)t.$$
On the other hand we have $t(\sum c_{f}f)=(\sum c_{f}[t,f]f)t$. Thus for any $t\in \wt{T}/T_{1}$ one has 
$$[t,f]=1$$
 for any $f$ such that $c_f\neq 0$. Thus the center of $\C[\wt{T}/T_{1}]e_\epsilon$ is equal to $\C[Z(\overline{T})/T_{1}]e_{\epsilon}$.

Now we prove (ii). One has
$$c(\chi)\Theta_{\s_y} c(\chi') \ne 0$$
 if and only if $y-y_{0}\in Y_{Q,n}$. Thus the set 
 $$\set{c(\chi)\Theta_{\s_{y_0}}\Theta_{\s_y}c(\chi'): y\in Y_{Q,n}}$$
 is a $\C$-basis for $c(\chi)\mca{R} c(\chi')$. Since $Z(\overline{T})\mbf{T}(O_F)/\mu_{n}\mbf{T}(O_F) \simeq  Y_{Q,n}$, the elements in 
 $$\set{\Theta_{\s_y}c(\chi): y \in Y_{Q,n}}$$ 
form a $\C$-basis for $Z(\mca{R})c(\chi)$ and  the result follows.
\end{proof}


Using Lemma \ref{HeckeTorus3} we refine Proposition \ref{HeckeRingDecomp} as follows.

\begin{prop}\label{MatProd}
One has an isomorphism of rings
$$\mca{R} \simeq \prod_{\mca{O}_\chi} \mbf{M}(\val{\msc{X}_{Q,n}}, Z(\mca{R})c(\chi)),$$
where the product is taken over the cosets in $\Hom(T_\kappa,\C^\times)/{\rm Im}(\varphi)$.
\end{prop}


With the choice of a genuine character of $Z(\overline{T})$ we can describe $Z(\mca{R})c(\chi)$ as a ring. 

\begin{lm}\label{centerIso}
Let $\sigma \in \Irrg(Z(\overline{T}))$. One has an injection of $\C$-algebras
$$\C[Y_{Q,n}]\hookrightarrow Z(\mca{R})$$
given by $y\mapsto \sigma(\s_y)\Theta_{\s_y}$. Furthermore, for $\chi\in\Hom(T_\kappa,\C^\times)$ the composition  of this map with the projection $Z(\mca{R}) \onto Z(\mca{R})c(\chi)$ defined by right multiplication by $c(\chi)$ gives an isomorphism 
$$\C[Y_{Q,n}] \longrightarrow Z(\mca{R}) c(\chi)$$
of $\C$-algebras.
\end{lm}
\begin{proof}
Since the set $\set{\Theta_{\s_y}: y\in Y_{Q,n}}$ is $\C$-linearly independent in $Z(\mca{R})$, the map is an injection of $\C$-vector spaces. By relation \eqref{Rel2} and the fact that $\sigma$ is genuine, we see that the map is a $\C $-algebra homomorphism. The composition of this map with the projection $Z(\mca{R})\onto Z(\mca{R})c(\chi)$ sends the $\C $-linearly independent set 
$$\set{\Theta_{\s_y}: y\in Y_{Q,n}}\subset Z(\mca{R})$$
to the $\C $-basis 
$$\set{\Theta_{\s_y}c(\chi): y\in Y_{Q,n}} \subset Z(\mca{R})c(\chi).$$
Thus the composition is an isomorphism of $\C $-vector spaces. Again relation (\ref{Rel2}) shows that the map is a $\C $-algebra homomorphism.
\end{proof}

\subsection{Intertwining operators}
In this subsection we consider algebraic intertwining operators, following Haines--Kottwitz--Prasad. The construction of these operators mostly involves minor modifications to \cite[\S 1.10]{HKP10}. Thus we will be brief and focus primarily on the modifications required in the context of covering groups.   

To begin, we define a completion of $\mca{M}$. Since $\overline{T}$ is in general non-commutative, it is more convenient to define our completions using $Z(\mca{R})$ instead of $\mca{R}$. Let $J$ be a subset of modified coroots contained in some system of positive modified coroots, and let $\C [J]$ be the subalgebra of $\C [Y_{Q,n}]$ generated by $J$. We write $\C [J]^{\wedge}$ for the completion of $\C [J]$ with respect to the maximal ideal generated by the elements of $J$. We define the $\C $-algebra 
$$\mca{R}_J=\C[J]^\wedge\otimes_{\C[J]}\mca{R},$$
where we view $\mca{R}$ as an $\C[J]$-module via the isomorphism of Lemma \ref{centerIso}. Set
$$\mca{M}_J:=\mca{R}_J \otimes_\mca{R} \mca{M} \text{ and } \mca{H}_J:=\mca{R}_J\otimes_R \mca{H}.$$
The rest of \cite[\S 1.10]{HKP10} applies without essential change.

Since our objective is to derive the Bernstein relations, we consider only the intertwining operators associated to simple reflections. For $\alpha\in \Delta$, let 
$$I_{\alpha}=I_{s_{\alpha}}: \mca{M}_{\{-\alpha_{Q,n}^\vee\}}\longrightarrow \mca{M}_{\{\alpha_{Q,n}^\vee\}}$$ 
be the intertwining operator given by 
$$I_{s_{\alpha}}(f)(g)=\int_{U_{\alpha}}f(\wt{w}_{\alpha}(-1)ug)du.$$
The Bernstein relations follow from an analogue of Haines--Kottwitz--Prasad \cite[Lemma 1.13.1 (i)]{HKP10}. We write $U^-(\frak{p}):=U^-\cap I$.

\begin{lm} \label{DoubleCosetRedRank}
Let $\alpha\in \Delta$ and $T_{\alpha}:=\alpha^\vee(F^\times)$. Then 
\begin{equation*}
\phi_{\alpha}(\SL_2(F))\cap UI_1=U_\alpha(T_\alpha\cap I_1)U_{-\alpha}^{1}.
\end{equation*}
\end{lm}
\begin{proof}
The result follows from the Bruhat decomposition. Specifically, 
\begin{equation*}
\phi_{\alpha}(\SL_2(F))=U_{\alpha}T_{\alpha}w_{\alpha}(1)\cup U_{\alpha}T_{\alpha}U_{-\alpha},
\end{equation*}
and $U_{\alpha}T_{\alpha}w_{\alpha}\cap UU^-(\frak{p})\subseteq UTw_{\alpha}U^- \cap UTU^-=\emptyset$.
\end{proof}

\begin{lm}\label{CoverDoubleCosetId}
Let $\alpha\in \Delta$ and $b\in O_F^\times$. Let $j\in \Z_{\gest 0}$ and $u\in F^{\times}$ be such that ${\rm val}(u)\lest 2j-1$. The following identity of double cosets holds in $\wt{G}$.
\begin{equation*}
U\wt{w}_{\alpha}(-1)\wt{e}_{\alpha}(u)\wt{h}_{\alpha}(\varpi^{j})\wt{h}_{\alpha}(b)U^-(\frak{p})=(u^{-1},\varpi^{j})_n^{Q(\alpha^\vee)} (u^{-1}\varpi^{j},b)_n^{Q(\alpha^{\vee})} U\wt{h}_{\alpha}(u^{-1}\varpi^{j}b)U^-(\frak{p}).
\end{equation*}
\end{lm}
\begin{proof} 
This follows from a direct calculation by using the relations (\ref{Rel2}), (\ref{Rel4}), (\ref{Rel6}), and (\ref{Rel7}).
\end{proof}

\begin{lm}
We have the following equality
\begin{equation*}
I_{s_{\alpha}}(\nu_{1})=q^{-1} \nu_{\wt{w}_{\alpha}(1)}+\sum_{j\in\Z_{\gest1}} q^{-j-1}\sum_{b \in \kappa^{\times}} (\varepsilon^{j}(b,\varpi^{j})_n)^{Q(\alpha^{\vee})} \nu_{\wt{h}_{\alpha}(\varpi^{j})\wt{h}_{\alpha}(b)}.
\end{equation*}
\end{lm}

\begin{proof}
This reduces to a rank-one calculation, which we discuss in detail. First, we determine the support of the function $I_{\alpha}(\nu_{1})(g)$. It suffices to consider representatives of the double cosets 
$$\mu_{n}T_{1}U\backslash \wt{G}/I_{1}\cong W(1).$$
In fact, we can work in $G$ instead of $\wt{G}$.

Suppose that $\wt{g}\in \wt{G}$ and $I_{\alpha}(\nu_{1})(\wt{g})\neq 0$. Let $g=\wp(\wt{g})$. We may assume that $g=wy(\varpi)t$, where $y\in Y, t\in \mbf{T}(O_F)$ and $w\in K\cap N(T)$ represents an element of the Weyl group. We claim that $w$ must represent $1$ or $w_{\alpha}$. Since $I_{\alpha}(\nu_{1})(\wt{g})\neq 0$, it follows that $w_{\alpha}(-1)Ug\cap UI_{1}\neq \emptyset$ and so $w_{\alpha}(-1)Uw\cap UTU^-\neq \emptyset$. Thus 
$$UTw_{\alpha}Uww_{\ell}U\cap UTw_{\ell}U\neq \emptyset.$$
Since 
$$UTw_{\alpha}Uww_{\ell}U\subseteq UTw_{\alpha}ww_{\ell}U\cup UTww_{\ell}U$$
 the element $w$ must represents $1$ or $w_{\alpha}$ by the Bruhat decomposition.  Thus we may assume that $w=1$ or $w=w_{\alpha}(1)$.

Now we determine $y\in Y$ and $t\in \mbf{T}(O_F)$. There are two cases as follows.
\begin{enumerate}
\item[$\bullet$] If $w=w_{\alpha}(1)$, then
$$w_{\alpha}(-1)Uw_{\alpha}(1)y(\varpi)t\cap UI_{1}\neq \emptyset$$
 if and only if $U_{-\alpha}y(\varpi)t\cap UI_{1}\neq \emptyset$, if and only if $y=0$ and $t\in T_{1}$.
\item[$\bullet$]
Suppose that $w=1$. In this case, $w_{\alpha}(-1)Uy(\varpi)t\cap UI_{1}\neq \emptyset$, thus 
$$U_{-\alpha}w_{\alpha}(-1)y(\varpi)\cap UK\neq \emptyset.$$
Let $t_{s_{\alpha}\cdot y}=y(\varpi)h_{\alpha}(\varpi^{-\langle y,\alpha\rangle})$. Then we have 
$$U_{-\alpha}w_{\alpha}y(\varpi)=t_{s_{\alpha}\cdot y}U_{-\alpha}w_{\alpha},$$
and  thus $U_{-\alpha}\cap Ut_{s_{\alpha}(y)}^{-1}K\neq \emptyset$. Since the elements in $U_{-\alpha}$ admit an Iwasawa decomposition in the $\SL_2$ corresponding to $\alpha$ we see that $y=j\alpha^{\vee}$ for some $j\in \Z$. Furthermore, a computation in $\SL_2$ shows that $y=j\alpha^{\vee}$, where $j\in \Z_{\gest 0}$. Similarly, we see that $t=h_{\alpha}(b)$, where $b\in O_F^{\times}$. 
\end{enumerate}

We have just shown that $I_{\alpha}(\nu_{1})$ is determined by its values on the elements $g=\wt{w}_{\alpha}(1),\, \wt{h}_{\alpha}(\varpi^{j})\wt{h}_{\alpha}(b)$, where $j\in \Z_{\gest 0}$ and $b\in O_F^{\times}$. 

Now we compute $I_{\alpha}(\nu_{1})(\wt{w}_{\alpha}(1))$. A direct computation yields
$$ I_{\alpha}(\nu_{1})(\wt{w}_{\alpha}(1))=\int_{F}\nu_{1}(\wt{e}_{-\alpha}(u))du=\mathrm{vol}(\frak{p}).$$
To compute $I_{\alpha}(\nu_{1})(\wt{h}_{\alpha}(\varpi^{j})\wt{h}_{\alpha}(b))$ with $j\in\Z_{\gest 0}$ and $b\in O_F^{\times}$, a calculation in $\SL_2$ shows that if this is nonzero, then $j\gest 1$. By definition,
$$
I_{\alpha}(\nu_1)(\wt{h}_{\alpha}(\varpi^{j})\wt{h}_{\alpha}(b))=\int_{F} \nu_{1}(\wt{w}_{\alpha}(-1)\wt{e}_{\alpha}(u)\wt{h}_{\alpha}(\varpi^{j})\wt{h}_{\alpha}(b))du.
$$
We apply Lemmas \ref{DoubleCosetRedRank} and \ref{CoverDoubleCosetId} to get
\begin{align*}
\int_{F} \nu_{1}(\wt{w}_{\alpha}(-1)\wt{e}_{\alpha}(u)\wt{h}_{\alpha}(\varpi^{j})\wt{h}_{\alpha}(b))du =& \int_{\varpi^{j}b(1+\frak{p})}((u^{-1},\varpi^{j})_n (u^{-1}\varpi^{j},b)_n)^{-Q(\alpha^{\vee})}du\\
=&q^{-j}\mathrm{vol}(1+\frak{p})(b\varpi^{j},\varpi^{j})_n^{Q(\alpha^{\vee})}.
\end{align*}
Combining the above gives
\begin{equation*}
I_{\alpha}(\nu_{1})=q^{-1}\nu_{\wt{w}_{\alpha}(1)}+\sum_{j\in\Z_{\gest1}}q^{-j-1}\sum_{b\in \kappa^{\times}}(b\varpi^{j},\varpi^{j})_n^{Q(\alpha^{\vee})} \nu_{\wt{h}_{\alpha}(\varpi^{j})\wt{h}_{\alpha}(b)}.
\end{equation*}
This completes the proof.
\end{proof}

\subsection{Bernstein presentation}\label{ProPBR}
For every $j\in \Z$, we set
$$c_{\alpha}(j)=\frac{1}{q-1}\sum_{b \in\kappa^{\times}}(b,\varpi)_n^j\Theta_{\wt{h}_{\alpha}(b)}.$$
The intertwining operator $I_{s_{\alpha}}$ for any $\alpha\in \Delta$ is represented by the element 
$$
q^{-1}\mathcal{T}_{\wt{w}_{\alpha}(1)}+(q-1)q^{-1}\sum_{j\in\Z_{\gest 1}}\varepsilon^{jQ(\alpha^{\vee})}\Theta_{\wt{h}_{\alpha}(\varpi^{j})}c_{\alpha}(jQ(\alpha^{\vee}))\in \mca{H}_{\{ \alpha_{Q,n}^\vee \}}.
$$

Since for any $t\in \overline{T}$ the intertwining operator satisfies $I_{s_{\alpha}}\circ\Theta_{t}=\Theta_{w_{\alpha}(1)\cdot t}\circ I_{s_{\alpha}}$, where $w_\alpha(1)\cdot t= w_\alpha(1) t w_\alpha(-1)$, we get the following proposition.

\begin{thm}[Bernstein relations] \label{T:BernR}
Let $\alpha\in\Delta$, then 
$$\begin{aligned}
& \Big( \mathcal{T}_{\wt{w}_{\alpha}(1)}+(q-1)\sum_{j\in\Z_{\gest 1}}\varepsilon^{jQ(\alpha^{\vee})}\Theta_{\wt{h}_{\alpha}(\varpi^{j})}c_{\alpha}(jQ(\alpha^{\vee})) \Big)*\Theta_{\s_y}\\
=& \Theta_{\wt{w}_{\alpha}(1)\cdot \s_y}*\Big( \mathcal{T}_{\wt{w}_{\alpha}(1)}+(q-1)\sum_{j\in\Z_{\gest 1}}\varepsilon^{jQ(\alpha^{\vee})}\Theta_{\wt{h}_{\alpha}(\varpi^{j})}c_{\alpha}(jQ(\alpha^{\vee})) \Big) \in\mca{H}_{\{\alpha_{Q,n}^{\vee}\}}.
\end{aligned}$$
Furthermore, this identity in $\mca{H}_{\{\alpha_{Q,n}^{\vee}\}}$ simplifies to an identity in $\mca{H}$ as follows: 
\begin{equation}\label{BR1}
\begin{aligned}
& \mathcal{T}_{\wt{w}_{\alpha}(1)}*\Theta_{\s_y}+(q-1)\sum_{j\in\Z_{\gest 1}}\varepsilon^{jQ(\alpha^{\vee})(1+\langle y,\alpha\rangle)}\Theta_{\s_y\wt{h}_{\alpha}(\varpi^{j})}c_{\alpha}((j+\langle y,\alpha\rangle)Q(\alpha^{\vee}))\\
= & \Theta_{w_{\alpha}(1)\cdot \s_y}*\mathcal{T}_{\wt{w}_{\alpha}(1)}+(q-1)\sum_{j\in\Z_{\gest 1-\langle y,\alpha\rangle}}\varepsilon^{jQ(\alpha^{\vee})(1+\langle y,\alpha\rangle)}\Theta_{\s_y\wt{h}_{\alpha}(\varpi^{j})}c_{\alpha}((j+\langle y,\alpha\rangle)Q(\alpha^{\vee})).
\end{aligned}
\end{equation}
Thus, in view of Proposition \ref{BernDecomp} this gives a Bernstein presentation for $\mca{H}$.
\end{thm}

For later, we record some formulas derived from Theorem \ref{T:BernR}.
If we substitute $w_{\alpha}(-1)\cdot \s_y$ in place of $\s_y$ and apply the identity $\wt{w}_{\alpha}(-1)\cdot \s_y=\varepsilon^{B(y,\alpha^{\vee})}\s_y\wt{h}_{\alpha}(\varpi^{-\langle y,\alpha\rangle})$ as from \eqref{Rel2}, \eqref{Rel3} and \eqref{Rel7}, then formula \eqref{BR1} becomes
 \begin{equation}\label{BR2}
 \begin{aligned}
& \mathcal{T}_{\wt{w}_{\alpha}(1)}*\Theta_{w_{\alpha}(-1)\cdot \s_y}+(q-1)\sum_{j\in\Z_{\gest 1}}\varepsilon^{jQ(\alpha^{\vee})(1-\langle y,\alpha\rangle)}\Theta_{(\wt{w}_{\alpha}(-1)\cdot \s_y)\wt{h}_{\alpha}(\varpi^{j})}c_{\alpha}((j-\langle y,\alpha\rangle)Q(\alpha^{\vee}))\\
= & \Theta_{\s_y}*\mathcal{T}_{\wt{w}_{\alpha}(1)}+(q-1)\sum_{j\in\Z_{\gest 1+\langle y,\alpha\rangle}}\varepsilon^{jQ(\alpha^{\vee})(1-\langle y,\alpha\rangle)}\Theta_{(w_{\alpha}(-1)\cdot \s_y)\wt{h}_{\alpha}(\varpi^{j})}c_{\alpha}((j-\langle y,\alpha\rangle)Q(\alpha^{\vee})).
\end{aligned}
\end{equation}
This gives
 \begin{equation}\label{BR3}
 \begin{aligned}
& \mathcal{T}_{\wt{w}_{\alpha}(1)}*\Theta_{w_{\alpha}(-1)\cdot \s_y}+(q-1)\sum_{j\in\Z_{\gest 1-\langle y ,\alpha\rangle}}\varepsilon^{jQ(\alpha^{\vee})}\Theta_{\s_y\wt{h}_{\alpha}(\varpi^{j})}c_{\alpha}(jQ(\alpha^{\vee}))\\
=& \Theta_{\s_y}*\mathcal{T}_{\wt{w}_{\alpha}(1)}+(q-1)\sum_{j\in\Z_{\gest 1}}
\varepsilon^{jQ(\alpha^{\vee})}\Theta_{\s_y\wt{h}_{\alpha}(\varpi^{j})}c_{\alpha}(jQ(\alpha^{\vee})).
\end{aligned}
\end{equation}
Multiply \eqref{BR3} by $\mbm{1}_{I} = c(\mbm{1})$ on the right to get
\begin{equation} \label{BR3*I}
\begin{aligned}
& \mathcal{T}_{\wt{w}_{\alpha}(1)}*\Theta_{\wt{w}_{\alpha}(-1)\cdot \s_y}*\mbm{1}_{I}+(q-1)\sum_{
\substack{
j\in\Z_{\gest 1-\langle y ,\alpha\rangle}\\
j\equiv 0\,\text{mod }n_{\alpha}
}}\varepsilon^{jQ(\alpha^{\vee})}\Theta_{\s_y\wt{h}_{\alpha}(\varpi^{j})}*\mbm{1}_{I}\\
= & \Theta_{\s_y}*\mathcal{T}_{\wt{w}_{\alpha}(1)}*\mbm{1}_{I}+(q-1)\sum_{\begin{smallmatrix}
j\in\Z_{\gest 1}\\
j\equiv 0\,\text{mod }n_{\alpha}
\end{smallmatrix}}\varepsilon^{jQ(\alpha^{\vee})}\Theta_{\s_y\wt{h}_{\alpha}(\varpi^{j})}*\mbm{1}_{I}.
\end{aligned}
\end{equation}
Henceforth, we set
$$\mca{A}:=\mbm{1}_{I}*\mca{R}*\mbm{1}_{I}.$$
Note that $j\equiv 0\,\text{mod }n_{\alpha}$ implies $\wt{h}_{\alpha}(\varpi^{j})\in Z(\overline{T})$ and thus 
$$\Theta_{\wt{h}_{\alpha}(\varpi^{j})}*\mbm{1}_{I}\in \mathcal{A}.$$

\begin{lm}\label{BRandA}
One has
\begin{equation*}
\Theta_{\s_y}*\mathcal{T}_{\wt{w}_{\alpha}(1)}*\mbm{1}_{I}\in \mathcal{T}_{\wt{w}_{\alpha}(1)}*\Theta_{w_{\alpha}(-1)\cdot \s_y}*\mbm{1}_{I}+\Theta_{\s_y}*\mathcal{A}.
\end{equation*}
Moreover, for any $\wt{w}\in K$ representing an element $w\in W$
\begin{equation*}
\Theta_{\s_y}*\mathcal{T}_{\wt{w}}*\mbm{1}_{I}\in \mathcal{T}_{\wt{w}}*\Theta_{\wt{w}^{-1}\cdot \s_y}*\mbm{1}_{I}+\sum_{w'<w}\mathcal{T}_{\wt{w}'}*\Theta_{(\wt{w}')^{-1}\cdot \s_y}*\mathcal{A}.
\end{equation*}
\end{lm}
\begin{proof}
The first equality is clear, while the second is follows from induction on the length of $w$.
\end{proof}

\subsection{Iwahori--Hecke algebra} \label{SS:IHalg}
Recall that 
$$\mbm{1}_{I}=c(\mbm{1})=\frac{1}{\val{T_\kappa}}\sum_{h\in T_\kappa}\mathcal{T}_h \in\mca{H}_\kappa.$$
The Iwahori--Hecke algebra for $\wt{G}$ with respect to $(I, s_K|_I)$ is
$$\mca{H}_{I}:=\mbm{1}_{I}*\mca{H}*\mbm{1}_{I}.$$
Our results for $\mca{H}$ can thus be projected onto $\mca{H}_{I}$. For $g\in \wt{G}$ and $t\in \wt{T}$ we write 
$$\mathcal{T}_{g}^{I}=\mbm{1}_{I}*\mathcal{T}_{g}*\mbm{1}_{I}, \quad \Theta_{t}^{I}=\mbm{1}_{I}*\Theta_{t}*\mbm{1}_{I}.$$
Set
$$\mca{H}_W:=\mbm{1}_{I}*\mca{H}_\kappa*\mbm{1}_{I} \subset \mca{H}_\kappa.$$

\begin{lm} The following holds:
\begin{enumerate}
\item[(i)] The map 
$$\mca{H}_W \longrightarrow C^{\infty}(I\backslash K/I)$$
 defined by $f\mapsto f|_{s(K)}$ is a $\C $-algebra isomorphism.\label{FinIwahoriIso} 
\item[(ii)] The element $\Theta_{t}^{I}$ is nonzero if and only if $t\in Z(\overline{T})$. Moreover, the choice of a genuine character $\sigma \in \Irrg(Z(\wt{T}))$ induces a map  
$$\C[Y_{Q,n}] \longrightarrow \mathcal{A}, \quad y\mapsto \sigma(\s_y)\Theta_{\s_y}^{I},$$
which is an isomorphism of $\C $-algebras. \label{BAlg}
\item[(iii)] The map 
$$\mca{H}_W\otimes_{\C} \mathcal{A}\longrightarrow \mca{H}_{I}$$ given by $h\otimes a\mapsto h*a$ is an isomorphism of $\C $-vector spaces.\label{IwahoriBF}
\item[(iv)] (Bernstein Relation) Let $\alpha\in \Delta$ and $y\in Y_{Q,n}$. Then\label{IwahoriBR}
\begin{equation*}\label{I*BR3*I}
\begin{aligned}
& \mathcal{T}_{w_{\alpha}(1)}^{I}*\Theta_{w_{\alpha}(-1)\cdot \s_y}^{I}+(q-1)\sum_{
\substack{
j\in\Z_{\gest 1-\langle y ,\alpha\rangle}\\
j\equiv 0\, {\rm mod }n_{\alpha}
}}\varepsilon^{jQ(\alpha^{\vee})}\Theta_{\s_yh_{\alpha}(\varpi^{j})}^{I}\\
= & \Theta_{\s_y}*\mathcal{T}_{w_{\alpha}(1)}*\mbm{1}_{I}+(q-1)\sum_{
\substack{
j\in\Z_{\gest 1}\\
j\equiv 0\,{\rm mod }n_{\alpha}
}}\varepsilon^{jQ(\alpha^{\vee})}\Theta_{\s_yh_{\alpha}(\varpi^{j})}^{I}.
\end{aligned}
\end{equation*}
\end{enumerate}
\end{lm}
\begin{proof}
For (i) the proof is identical to that of Lemma \ref{HeckeSplit}.
To prove (ii), one has
$$\Theta_{t}*\mbm{1}_{I}=c([t,-])*\Theta_{t}$$
by Lemma \ref{BasicHeckeRel}. Thus we see that $\mbm{1}_{I}*\Theta_{t}*\mbm{1}_{I}\neq 0$ if and only if $t\in Z(\overline{T})\mbf{T}(O_F)$. Note that the isomorphism is the same as Lemma \ref{centerIso}, when $\chi=\mbm{1}$.

The map  $\mca{H}_W \otimes_{\C } \mathcal{A}\longrightarrow \mca{H}_{I}$ is the restriction of the isomorphsim 
$$\mca{H}_\kappa\otimes_{\mca{H}_\kappa\cap \mca{R}}\mca{R}\longrightarrow \mca{H}$$
from Proposition \ref{BernDecomp}; thus it is injective. The map is also surjective since for any $h\in\mca{H}_\kappa$ we have $\mbm{1}_{I}*h=h*\mbm{1}_{I}$. This gives (iii).

The assertion (iv) follows by multiplying equation (\ref{BR3*I}) on the left by $\mbm{1}_{I}$.
\end{proof}


\section{Gelfand--Graev Representation} \label{S:GG}

In this section we study the Gelfand--Graev representation 
$$\mca{V}:=\mathrm{ind}_{\mu_{n}U^-}^{\wt{G}}(\epsilon \otimes\psi),$$
where 
$$\psi:U^- \longrightarrow \C ^{\times}$$
 is a nondegenerate character of conductor $\frak{p}$ and $\wt{G}$ acts by right translation on $\mca{V}$. Our objective is to understand $\mca{V}^I$ as an $\mca{H}_{I}$-module, for which we begin by studying the finite Gelfand--Graev representation.

\subsection{Finite Gelfand--Graev representation}
Since $\psi$ has conductor $\frak{p}$, it descends to a non-degenerate character (still denoted by)
$$\psi: U_\kappa^- \longrightarrow \C^\times.$$
Set
$$\mca{V}_\kappa:={\rm Ind}_{U_\kappa^-}^{G_\kappa}(\psi)^{U_\kappa},$$
the $U_\kappa$-fixed vectors in the Gelfand--Graev representation for the finite group $G_\kappa$.

By Lemma \ref{HeckeSplit}, for the finite Hecke algebra one has 
$$\mca{H}_\kappa\simeq C^{\infty}_{c}(I_{1}\backslash K/I_{1}) \simeq \C[U_\kappa \backslash G_\kappa/ U_\kappa],$$
where the measure used to define convolution in $\C [U_\kappa \backslash G_\kappa/U_\kappa]$ is $q^{-1}$ times the counting measure. Thus $\mca{H}_\kappa$ acts on $\mca{V}_\kappa$ by convolution and so does the subalgebra of functions supported on $U_\kappa T_\kappa U_\kappa=U_\kappa T_\kappa = T_\kappa U_\kappa$, which is isomorphic to $\C[T_\kappa]$.

\begin{lm}\label{FinGGT}
The map $\mca{V}_\kappa \longrightarrow \C[T_\kappa]$ given by $f\mapsto f|_{T_\kappa}$ is an isomorphism of $\C[T_\kappa]$-modules.
\end{lm}

\begin{proof}
The functions in $\mca{V}_\kappa$ have support contained in $U_\kappa^- T_\kappa U_\kappa$. By the Bruhat decomposition the map is an isomorphism of vector spaces. It is a $\C[T_\kappa]$-module homomorphism because $U_\kappa T_\kappa =T_\kappa U_\kappa$.
\end{proof}

Next we describe the structure of $\mca{V}_\kappa$ as an $\mca{H}_\kappa$-module. For $\chi\in\Hom(T_\kappa,\C ^{\times})$, recall that $c(\chi)\in \C[T_\kappa]$ is the idempotent associated with $\chi$. We also write $c(\chi)$ for the element of $\mca{V}_\kappa \simeq \C[T_\kappa]$, under the isomorphism of Lemma \ref{FinGGT}. Let 
$$\mathcal{O} \subset \Hom(T_\kappa,\C^{\times})$$
be a $W$-orbit with respect to the action  induced from that on $T_\kappa$. Let
$$\mca{V}_{\kappa, \mca{O}}:={\rm Span}_{\C}\set{c(\chi): \chi\in \mca{O} } \subseteq \mca{V}_\kappa.$$
For $\chi \in \Hom(T_\kappa, \C^\times)$ and $\alpha\in \Delta$, we set
\begin{equation} \label{D:GS}
\g_{\alpha}(\psi,\chi):=\sum_{u\in\kappa^{\times}}\psi(e_{-\alpha}(u)) \cdot \chi(h_{-\alpha}(u)).
\end{equation}

\begin{prop} 
The finite Gelfand--Graev representation $\mca{V}_\kappa = {\rm ind}_{U_\kappa^-}^{G_\kappa}(\psi)^{U_\kappa}$ decomposes as
$$
\mca{V}_\kappa=\bigoplus_{\mca{O}} \mca{V}_{\kappa, \mca{O}},
$$
where the direct sum is taken over the $W$-orbits in $\Hom(T_\kappa, \C^\times)$. Moreover, each $\mca{V}_{\kappa, \mca{O}}$ is an irreducible $\mca{H}_\kappa$-module.
\end{prop}
\begin{proof}
By Lemma \ref{FinGGT} we know that $ {\rm ind}_{U_\kappa^-}^{G_\kappa}(\psi)^{U_\kappa}$ decomposes as a $\C[T_\kappa]$-module into a direct sum over all of the characters of $T_\kappa$.  For $g\in G_\kappa$ we write 
$$\mca{T}^\kappa_{g}:= {\rm ch}_{U_\kappa g U_\kappa} \in\C[U_\kappa\backslash G_\kappa/U_\kappa],$$
the characteristic function of $U_\kappa g U_\kappa$. For $\alpha\in\Delta$, we also set
$$\mca{T}^\kappa_\alpha: = \mca{T}^\kappa_{w_{\alpha}(1)}.$$

For $\chi\in\Hom(T_\kappa,\C ^{\times})$ and $\alpha\in \Delta$ we  view $c(\chi) \in \mca{V}_\kappa$.
A direct computation gives
\begin{align*}
c(\chi)*\mca{T}^\kappa_{\alpha}(t)=&\int_{Uw_{\alpha}(1)U}c(\chi)(th^{-1})\mca{T}^\kappa_{w_{\alpha}(1)}(h)dh\\
=&\sum_{u\in \kappa^{\times}}c(\chi)(te_{\alpha}(u)w_{\alpha}(-1))\\
=&\Big( \sum_{u\in\kappa^{\times}}c(\chi)(e_{\alpha}(u)w_{\alpha}(-1))\Big) {}^{w_\alpha}\chi(t) \\
=&\Big( \sum_{u\in\kappa^{\times}}\psi^{-1}(e_{-\alpha}(-u^{-1}))c(\chi)(h_{\alpha}(u))\Big) {}^{w_\alpha}\chi(t)\\
=&\Big( \sum_{u\in\kappa^{\times}}\psi(e_{-\alpha}(u))\chi(h_{-\alpha}(u)) \Big) {}^{w_\alpha}\chi(t) \\
=&\g_{\alpha}(\psi,\chi) \cdot {}^{w_\alpha}\chi(t)
\end{align*}
This gives
\begin{equation}\label{FinGGAction}
c(\chi)*\mca{T}^\kappa_{\alpha}=\g_{\alpha}(\psi,\chi) \cdot c({}^{w_\alpha} \chi).
\end{equation}

The algebra $\mca{H}_\kappa$ is generated by the elements $\mca{T}^\kappa_{\alpha}$ and $\mca{T}^\kappa_{t}$, where $\alpha\in \Delta$ and $t\in T_\kappa$. Thus
$$
\mca{V}_\kappa=\bigoplus_{\mca{O}} \mca{V}_{\kappa, \mca{O}},
$$
where the direct sum is taken over $W$-orbits in $\Hom(T_\kappa,\C ^{\times})$. 

In fact, each $\mca{V}_{\kappa, \mca{O}}$ is irreducible. Let $v\in \mca{V}_{\kappa, \mca{O}}$ be a nonzero vector. By Lemma \ref{FinGGT}, there is a character $\chi\in \Hom(T_\kappa,\C ^{\times})$ such that 
$$v*c(\chi)=z \cdot c(\chi)$$ for some $z \in \C^{\times}$. Finally, the equality \eqref{FinGGAction} implies $v*\mca{H}_\kappa=\mca{V}_{\kappa, \mca{O}}$. Thus $\mca{V}_{\kappa, \mca{O}}$ is irreducible.
\end{proof}

Now we show that each $\mca{V}_{\kappa,\mathcal{O}}$ is an induced representation. Let $\chi\in\Hom(T_\kappa,\C ^{\times})$. Then consider the algebra 
$$\mca{H}_{\kappa,\chi}=c(\chi)\mca{H}_\kappa c(\chi).$$
A $\C $-basis for $\mca{H}_{\kappa,\chi}$ is given by 
$$\set{c(\chi)\mca{T}^\kappa_w c(\chi): w\in W_\chi},$$
where $w\in G_\kappa$ is the natural representative of an element in $W_{\chi}:={\rm Stab}_{W}(\chi)$. The action in line (\ref{FinGGAction}) shows that $\mca{H}_{\kappa,\chi}$ acts on the one-dimensional subspace $\C\cdot c(\chi)\subset \mca{V}_\kappa$. Henceforth, we write 
$$\tau_{\chi}:=\C \cdot c(\chi)$$
 for this one-dimensional representation of $\mca{H}_{\kappa,\chi}$.

\begin{lm}\label{OrbitInduced}
Let $\chi\in \Hom(T_\kappa,\C ^{\times})$ and let $\mathcal{O}_{\chi}$ be the $W$-orbit of $\chi$. Then the linear map 
$$\tau_{\chi}\otimes_{\mca{H}_{\kappa,\chi}}\mca{H}_\kappa\longrightarrow \mca{V}_{\kappa, \mathcal{O}_{\chi}}$$
 defined by $c(\chi)\otimes h\mapsto c(\chi)*h$ is an $\mca{H}_\kappa$-module isomorphism.
\end{lm}
\begin{proof}
It follows from \eqref{FinGGAction} that the linear map $\tau_{\chi} \longrightarrow \mca{V}_{\mathcal{O}}$ defined by $c(\chi)\mapsto c(\chi)$ is an $\mca{H}_{\kappa,\chi}$-module map and induces an $\mca{H}_\kappa$-module homomorphism $\tau_{\chi}\otimes_{\mca{H}_{\kappa,\chi}}\mca{H}_\kappa\rightarrow \mca{V}_{\kappa, \mathcal{O}_{\chi}}$ defined by $c(\chi)\otimes h\mapsto c(\chi)*h$. This map is surjective by \eqref{FinGGAction}. The map is an isomorphism because both spaces have dimension $\val{W_\chi\backslash W}$.
\end{proof}

Lemma \ref{OrbitInduced} can be refined to an isomorphism
$$\tau_{\chi}\otimes_{\mca{H}_{\kappa,\chi}}c_{\mathcal{O}}\mca{H}_\kappa c_{\mathcal{O}}\simeq \mca{V}_{\kappa, \mathcal{O}},$$
where 
$$c_{\mathcal{O}}:=\sum_{\chi'\in \mathcal{O}}c(\chi')$$ is the central idempotent corresponding to the orbit $\mathcal{O}$. In particular, for the trivial orbit we have the following.

\begin{cor}\label{finiteTrivial}
Let $\mathcal{O}_\mbm{1}$ be the orbit of the trivial character $\mbm{1}$ of $T_\kappa$. Then $\mca{V}_{\kappa, \mca{O}_\mbm{1}}=\C\cdot c(\mbm{1})$ and 
\begin{equation*}
\mca{V}_{\kappa, \mca{O}_\mbm{1}}\simeq {\rm sign}
\end{equation*}
as $C(B_\kappa \backslash G_\kappa/B_\kappa)$-modules.
\end{cor}
\begin{proof}
By definition $\mca{V}_{\kappa, \mca{O}_\mbm{1}}=\C \cdot c(\mbm{1})$ and by a direct computation we see that $$c(\mbm{1})*\mca{T}^\kappa_{w_{\alpha}(1)}=c(\mbm{1})*\mca{T}^\kappa_{\alpha}*c(\mbm{1})\in C(B_\kappa\backslash G_\kappa/B_\kappa).$$
The result then follows from \eqref{FinGGAction}.
\end{proof}

\subsection{The pro-$p$-fixed vectors}
We follow the approach of Chan--Savin \cite{CS18} to study $\mca{V}^{I_{1}}$. For any $U$-module $S$, we write $S_{U}$ for the $\overline{T}$-module of $U$-coinvariants.

\begin{lm}\label{prop1} The natural map $\mca{V}\longrightarrow \mca{V}_{U}$ induces an isomorphism of $\overline{T}$-modules
\begin{equation}\label{TypeIso}
\mca{V}^{I_{1}}\simeq (\mca{V}_{U})^{T_{1}}.
\end{equation}
\end{lm}
\begin{proof}
This follows from the theory of Bushnell--Kutzko \cite{BK98}, which is applicable because $I_{1}$ possesses an Iwahori factorization and $\mathcal{T}_{t}\in \mca{H}$ is invertible for any $t\in \overline{T}$.
\end{proof}

\begin{lm}\label{prop2}
Let $\mca{V}_{0} \subset \mca{V}$ be the subspace consisting of functions supported on $U^-\wt{T}U$. The inclusion $\mca{V}_{0}\subseteq \mca{V}$ induces an isomorphism of $\overline{T}$-modules $(\mca{V}_{0})_{U}\simeq \mca{V}_U$.
\end{lm}
\begin{proof}
The proof of Lemma 4.1 from Chan--Savin \cite{CS18} adapts without change.
\end{proof}

\begin{lm}\label{prop3}
The map $\mathcal{S}: (\mca{V}_{0})_{U}\longrightarrow C^{\infty}_{c,\epsilon}(\overline{T})$ defined by 
$$\mathcal{S}(f)(t)=\delta_{U}(t)^{-1/2}\int_{U}f(tu)du$$ is an isomorphism of $\overline{T}$-modules. 
\end{lm}
\begin{proof}
Again this adapts directly from Chan--Savin \cite[Proposition 4.2]{CS18}.
\end{proof}

Let $\mathrm{ch}_{t}\in C^{\infty}_{c,\epsilon}(\overline{T}/T_{1})$ be the unique function with support equal to $t\mu_{n}T_{1}$ and $\mathrm{ch}_{t}(t)=1$. The map 
$$\overline{T}/T_{1}\rightarrow C^{\infty}_{c,\epsilon}(\overline{T}/T_{1})=C^{\infty}_{c,\epsilon}(\overline{T})^{T_{1}}$$
defined by $t\mapsto {\rm ch}_{t}$ is a $\overline{T}/T_{1}$-equivariant injection. This map extends linearly to give an isomorphism of $\mca{R}$-modules $\mca{R}\longrightarrow C^{\infty}_{c,\epsilon}(\overline{T}/T_{1})$. 
Let 
$${\rm ch}_{I_{1}}^{\psi}\in \mca{V}^{I_{1}}$$ be the function supported on $\mu_{n}U^- I_{1}$ such that ${\rm ch}_{I_{1}}^{\psi}(1)=1$. Note that ${\rm ch}_{I_{1}}^{\psi} \in \mca{V}_{0}$.
Since the factorization $\mu_{n}U^- I_{1}=\mu_{n}U^- T_{1}(U\cap I_{1})$ is unique, a direct calculation shows that
$$\mca{S}(\mathrm{ch}_{I_{1}}^{\psi}) = z \cdot \mathrm{ch}_{1}\in C^{\infty}_{c,\epsilon}(\overline{T}/T_{1})$$
with $z \in \C^\times$.

Combining Lemmas \ref{prop1}, \ref{prop2}, and \ref{prop3}, we obtain the following:

\begin{prop}\label{prop4}
There is an isomorphism of $\overline{T}$-modules 
$$\vartheta_{1}: \mca{V}^{I_{1}}\rightarrow \mca{R}$$ induced by the isomorphisms of Lemmas \ref{prop1}, \ref{prop2}, and $\ref{prop3}$. Under this isomorphism, the functions in $\mca{V}^{I_{1}}$ with support contained in $\mu_{n}U^- K$ are mapped to $\C[\mu_{n},T_\kappa]$ and thus generate $\mca{R}$ as an $\mca{R}$-module.
\end{prop}

We define a map
$$\vartheta_2: \mca{V}_\kappa \longrightarrow \mca{V}^{I_1}$$
by setting
$$\vartheta(f)(\zeta u t k) =
\begin{cases}
\epsilon(\zeta)\psi(u)f(t k) & \text{ if $t\in T\cap K$},\\
0 & \text{ otherwise}.
\end{cases}
$$

\begin{lm}\label{Embed}
The map $\vartheta_2$ is an embedding  of $\mca{H}_\kappa$-modules.  The image $\vartheta_2(\mca{V}_\kappa) \subset \mca{V}^{I_1}$ consists of the functions  with support contained in $\mu_{n}U^- K$.
\end{lm}
\begin{proof}
A direct computation using the Iwasawa decomposition $\wt{G}=U\wt{T}K$ shows that the map is well-defined and is an embedding of $\mca{H}_\kappa$-modules. The second claim follows by the definition of $\vartheta_2(f)$.
\end{proof}

\begin{thm}\label{GGProP}
The $\mca{H}$-module map 
$$\pmb{\gamma}: \mca{V}_\kappa \otimes_{\mca{H}_\kappa}\mca{H} \longrightarrow \mca{V}^{I_{1}}$$
 induced by the embedding $\vartheta_{2}:\mca{V}_\kappa \longrightarrow \mca{V}^{I_{1}}$ from Lemma \ref{Embed} is an isomorphism.
\end{thm}

\begin{proof}
By Frobenius reciprocity and Lemma \ref{Embed} there is a nonzero map $\mca{V}_\kappa \otimes_{\mca{H}_\kappa}\mca{H} \rightarrow \mca{V}^{I_{1}}$ of $\mca{H}$-modules, which we also call $\vartheta_{2}$. It suffices to show that 
$$\vartheta_{1}\circ\vartheta_{2}: \mca{V}_\kappa \otimes_{\mca{H}_\kappa}\mca{H} \rightarrow \mca{R}$$
 is an isomorphism of $\mca{R}$-modules.

The map $\vartheta_{1}\circ\vartheta_{2}: \mca{V}_\kappa \longrightarrow \mca{R}$ is injective with image equal to $\C [\mu_{n},T_\kappa]=\mca{H}_\kappa\cap \mca{R}\subset \mca{R}$ by Proposition \ref{prop4} and Lemma \ref{Embed}. By Proposition \ref{BernDecomp} we have $\mca{H} \simeq  \mca{H}_\kappa\otimes_{\mca{R}\cap \mca{H}_\kappa} \mca{R}$ as $\C$-vector spaces. It follows that 
$$\begin{tikzcd}
\mca{V}_\kappa \otimes_{\mca{H}_\kappa\cap \mca{R}} \mca{R} \simeq \mca{V}_\kappa \otimes_{\mca{H}_\kappa} \mca{H} 
\ar[r, "{\vartheta_1 \circ \vartheta_2}"] & \mca{R}
\end{tikzcd} $$ 
is an isomorphism. This completes the proof.
\end{proof}

\begin{rmk}
Although our arguments in this subsection require only minor modification to those of Chan--Savin \cite{CS18}, the usage of the pro-$p$ subgroup $I_{1}$ is essential. The key point is that the pro-$p$ Hecke algebra $\mca{H}$ contains functions supported on all elements of $\overline{T}$, while the Iwahori--Hecke algebra $\mca{H}_{I}$ does not, a fact which is closely related to the failure of multiplicity one for Whittaker models. If one carries out the above discussion using $I$ in place of $I_{1}$, then the analogue of Theorem \ref{GGProP} fails. This failure is a direct consequence of the constraints on the support of functions in $\mca{H}_{I}$.
\end{rmk}

\subsection{The Iwahori-fixed vectors}

The main objective of this subsection is to describe $\mca{V}^{I}$ as an $\mca{H}_{I}$-module. Our work from the last subsection provides the foundation. Let $\mbm{1}_{I} \in \mca{H}_I$ be the identity element.

\begin{cor} \label{C:VI-O}
The $\mca{H}$-module isomorphism $\pmb{\gamma}: \mca{V}_\kappa \otimes_{\mca{H}_\kappa}\mca{H} \rightarrow \mca{V}^{I_{1}}$ from Theorem \ref{GGProP} induces an isomorphism of $\mca{H}_{I}$-modules
$$\pmb{\gamma}: \mca{V}_\kappa \otimes_{\mca{H}_\kappa}\mca{H} *\mbm{1}_{I}\longrightarrow \mca{V}^{I}.$$
Furthermore,
\begin{equation*}
\mca{V}_\kappa \otimes_{\mca{H}_\kappa}\mca{H} *\mbm{1}_{I}=\bigoplus_{\mathcal{O}} \big( \mca{V}_{\kappa, \mathcal{O}}\otimes_{\mca{H}_\kappa}\mca{H}*\mbm{1}_{I} \big),
\end{equation*}
where the direct sum is taken over all the $W$-orbits in $\Hom(T_\kappa,\C ^{\times})$.
\end{cor}

By the above corollary, we can focus on the $\mca{H}_{I}$-module structure of each $\mca{V}_{\kappa, \mathcal{O}}\otimes_{\mca{H}_\kappa}\mca{H}*\mbm{1}_{I}$ individually. To begin, we first characterize the orbits $\mathcal{O}$ such that 
$$\mca{V}_{\kappa, \mathcal{O}}\otimes_{\mca{H}_\kappa}\mca{H}*\mbm{1}_{I}\neq 0.$$
Recall that we have an injective map 
$$\varphi: \msc{X}_{Q,n} \longrightarrow  \Hom(T_\kappa,\C ^{\times})$$
given by $\varphi(y)(-)=[\s_y,-]|_{\mbf{T}(O_F)}$. This map is $W$-equivariant.

\begin{lm}\label{OrbitSupp}
Let $\mathcal{O} \subset \Hom(T_\kappa,\C ^{\times})$ be a $W$-orbit. Then $\mca{V}_{\kappa, \mathcal{O}}\otimes_{\mca{H}_\kappa}\mca{H}*\mbm{1}_{I}\neq 0$ if and only if $\mathcal{O} \subset {\rm Im}(\varphi)$.
\end{lm}
\begin{proof}
Let $\chi\in \Hom(T_\kappa,\C ^{\times})$ and $y\in Y$. Since $\mbm{1}_{I}=c(1)$ and $\Theta_{\s_y}*c(1)=c([\s_y,-])*\Theta_{\s_y}$, we have
\begin{equation*}
c(\chi)\otimes \Theta_{\s_y}*\mbm{1}_{I}=c(\chi)c(\varphi(y))\otimes \Theta_{\s_y}*\mbm{1}_{I},
\end{equation*}
which is zero unless $\chi=\varphi(y)$. By Proposition \ref{BernDecomp}, the elements of the form $c(\chi)\otimes \Theta_{\s_y}*\mbm{1}_{I}$ generate $\mca{V}_{\mathcal{O}}\otimes_{\mca{H}_\kappa}\mca{H}*\mbm{1}_{I}$ as a $\C $-vector space. Thus if $\mca{V}_{\kappa, \mathcal{O}}\otimes_{\mca{H}_\kappa}\mca{H}*\mbm{1}_{I}\neq 0$, then for every $\chi\in\mathcal{O}$ we must have $\chi=\varphi(y)$ for some $y\in Y$.

Conversely, for $y\in Y$ the element $c(\varphi(y))\otimes \Theta_{\s_y}*\mbm{1}_{I}\neq 0$.  Indeed, by Lemma \ref{Embed} and Theorem \ref{GGProP}, $c([\s_y,-])\otimes 1\in \mca{V}_\kappa \otimes_{\mca{H}_\kappa} \mca{H}$ is nonzero. Since $\Theta_{\s_y}$ is invertible, the element $c(\varphi(y))\otimes \Theta_{\s_y}=c(\varphi(y))\otimes \Theta_{\s_y}*\mbm{1}_{I}$ is nonzero.
\end{proof}

\begin{dfn}
For every $W$-orbit $\mca{O} \subset \msc{X}_{Q,n}$, we call
$$\mca{V}^I_\mca{O}:= \pmb{\gamma}\left(\mca{V}_{\kappa, \mathcal{O}}\otimes_{\mca{H}_\kappa}\mca{H}*\mbm{1}_{I}\right)$$
the $\mca{O}$-component of $\mca{V}^I$.
\end{dfn}
We have the decomposition
$$\mca{V}^I = \bigoplus_{\mca{O} \subset \msc{X}_{Q,n}} \mca{V}^I_\mca{O}$$
over all $W$-orbits in $\msc{X}_{Q,n}$.

Next we describe $\mca{V}_{\kappa, \mathcal{O}}\otimes_{\mca{H}_\kappa}\mca{H}*\mbm{1}_{I}$ (and thus also $\mca{V}^I_\mca{O}$) as an $\mathcal{A}$-module. Note that $\Theta_{\s_y}^{I}\neq 0$ if and only if $y\in Y_{Q,n}$. Thus $\mca{R}*\mbm{1}_{I}$ is a free $\mathcal{A}$-module of rank $\val{\msc{X}_{Q,n}}$. In particular, if $\set{y_j}_{j\in I} \subset Y$ is a set of representatives for $\msc{X}_{Q,n}$, then $\set{\Theta_{\s_{y_j}}*\mbm{1}_{I}}_{j\in I}$ is an $\mathcal{A}$-basis for $\mca{R}*\mbm{1}_{I}$. 

\begin{lm} \label{OrbitA}
Let $\mathcal{O} \subset \msc{X}_{Q,n}$ be a $W$-orbit and let $\set{y_j}_j\subset Y$ be a set of representatives for $\mca{O}$. Then 
$$\mca{V}_{\kappa, \mathcal{O}}\otimes_{\mca{H}_\kappa}\mca{H}*\mbm{1}_{I} \simeq \mca{A}^{\oplus \val{\mca{O}}}$$
as $\mca{A}$-module with an $\mca{A}$-basis given by
$$\set{c(\varphi(y_{j}))\otimes \Theta_{\s_{y_j}}*\mbm{1}_{I}: \varphi(y_{j})\in \mathcal{O}}.$$
Also, $\mca{V}_\kappa \otimes_{\mca{H}_\kappa}\mca{H}*\mbm{1}_{I}$ is a free $\mathcal{A}$-module of rank $\val{\msc{X}_{Q,n}}$.
\end{lm}

\begin{proof}
By Proposition \ref{BernDecomp} we have 
$$\mca{V}_{\kappa, \mathcal{O}}\otimes_{\mca{H}_\kappa}\mca{H}*\mbm{1}_{I}\simeq \mca{V}_{\kappa, \mathcal{O}}\otimes_{\mca{H}_\kappa\cap \mca{R}} \mca{R}*\mbm{1}_{I}$$
as $\mca{A}$-modules. The result follows because $\mca{R}*\mbm{1}_{I}$ is a free $\mca{A}$-module of rank $\val{\msc{X}_{Q,n}}$. Indeed, for a choice of representatives $\{y_{j}\}$ for $\mathscr{X}_{Q,n}$, the set $\{\Theta_{s_{y_j}}*\mbm{1}_{I}\}$ is an $\mathcal{A}$-basis for $\mca{R}*\mbm{1}_{I}$. However, for $\chi\in \Hom(T_\kappa,\mu_{n})$ the element $c(\chi)\otimes \Theta_{\s_{y_j}}*\mbm{1}_{I}$ is nonzero if and only if $\chi=\varphi(y_{j})$. 
\end{proof}

\subsection{Further simplification} \label{SS:sdes}
To obtain a unique and simple description of $\mca{V}^I_\mca{O} \simeq \mca{V}_{\kappa, \mathcal{O}}\otimes_{\mca{H}_\kappa}\mca{H}*\mbm{1}_{I}$, we will consider $\mca{O}$ with a special ``splitting" property as follows.

For use later, we introduce a slight generalization of the $W$-orbits on $Y$ and $\msc{X}_{Q,n}$. 
Let $L \subset Y$ be a $W$-stable sublattice of the same rank as $Y$ with respect to the usual reflection action of $W$ on $Y$. For every element $z$ in the coweight lattice $P$, one considers the twisted Weyl action on $Y$ given by
$$w[y]_z:=w(y+z) - z.$$
It is well-defined since $z\in P$. This gives a finite dimensional permutation representation
\begin{equation*} 
\sigma_{[z]}^{Y/L}: W \longrightarrow {\rm Perm}(Y/L).
\end{equation*}

An orbit $\mca{O} \subset Y$ with respect to the action $w[-]_z$ given above is called a $(W, z)$-orbit. Clearly, the $(W, 0)$-orbits are just the $W$-orbits. The $w[-]_z$ action is also well-defined on $Y/L$ and the quotient map $Y \onto Y/L$ is $(W, z)$-equivariant. Specializing to the special case $L=Y_{Q,n}$, we have

\begin{dfn} \label{D:splO}
A $(W, z)$-orbit $\mca{O} \subset \msc{X}_{Q,n}=Y/Y_{Q,n}$ is called splitting if there exists a section of the quotient map
$$Y \onto \msc{X}_{Q,n},$$
which is equivariant with respect to the $w[\cdot]_z$-action on both $Y$ and $\msc{X}_{Q,n}$.
\end{dfn}
The orbit of $\hat{0} \in \msc{X}_{Q,n}$ is always $(W, 0)$-splitting. Also, the element $y:=\rho- \rho_{Q,n}$, if it lies in $Y$, gives rise to a splitting $(W, -\rho)$-orbit $\mca{O}_{\hat{y}}=\set{\hat{y}}$. This latter example plays a crucial role as a base point for a special form of the covering Casselman--Shalika formula, see \cite[\S 5.3]{GSS2}. On the other hand, if $\mca{O}_{\hat{y}} \subset \msc{X}_{Q,n}$ is a free $(W, z)$-orbit, then it is always splitting.

For the rest of this section, we only consider the $W$-orbits.
The next lemma collects some basic consequences of a $W$-orbit $\mca{O}_{\hat{y}}$ possessing a splitting. Given $y\in Y$ we write 
$$W_y:={\rm Stab}_{W}(y) \text{ and } W_{\hat{y}}:={\rm Stab}_{W}(\hat{y}).$$

\begin{lm}\label{WeqSplitting}
Let $\mathcal{O}\subseteq \mathscr{X}_{Q,n}$ be a splitting $W$-orbit. Let $s: \mca{O} \into Y$ be such a splitting.
\begin{enumerate}
\item[(i)] For every $y\in s(\mca{O})$, one has $W_{y}=W_{\hat{y}}$, which is a parabolic subgroup of $W$. 
\item[(ii)] If $y\in Y$ and $\alpha\in \Phi$ such that $w_{\alpha}\in W_{y}$, then $\langle\alpha,y \rangle=0$.
\item[(iii)] Let $W'\subseteq W$ be a parabolic subgroup such that $w_{j}\in W$ are of minimal length in $W'w_{j}$, $j=1,2$. If $w\in W'w_{1}$ and $w\leq w_{2}$, then $w_{1}\leq w_{2}$.\label{CosetOrder}
\end{enumerate}
\end{lm}
\begin{proof}
Here (i) follows from \cite[Proposition 1.15]{Hum}, while (iii) from Proposition 1.10 of loc. cit. and the description of the Bruhat order in terms of subexpressions. Statement (ii) is clear by the definition of reflection.
\end{proof}

In general, we let
$$\tilde{W}_{\aff} = Y_{Q,n}^{sc} \rtimes W$$
be the modified affine Weyl group, which gives that
$$Y_{Q,n} \rtimes W = \tilde{W}_{\aff} \rtimes \Omega,$$
where $\Omega \subset Y_{Q,n} \rtimes W$ consists of the elements which fix the fundamental alcove in $Y_{Q,n}\otimes \R$ associated to $\tilde{W}_\aff$ as in \S \ref{SS:rsys}. Thus
$$\mca{H}_I = \mca{H}_{\tilde{W}_\aff} \otimes_\Z \Z[\Omega],$$
where the algebra law for the right hand is given as in \cite[Page 47]{IM65}. For any $y \in Y$, let 
$$\mca{H}_{I, y}:=\mca{H}_{\tilde{W}_{\aff, y}} \subset \mca{H}_I$$
be the subalgebra of $\mca{H}_I$ associated with $\tilde{W}_{\aff, y}:={\rm Stab}_{\tilde{W}_\aff}(y)$. In particular, if $\tilde{W}_{\aff, y} \subset W$ is actually a parabolic subgroup of $W$, then $\mca{H}_{I, y}$ is generated by the elements in
$$\set{\mbm{1}_{I}*\mathcal{T}_{\wt{w}}*\mbm{1}_{I}: w\in \tilde{W}_{\aff, y}};$$
in this case, $\mca{H}_{I, y} \subset \mca{H}_W$.

\begin{thm}\label{Wequi}
Let $\mathcal{O}\subseteq \mathscr{X}_{Q,n}$ be a splitting $W$-orbit with a splitting given by $s: \mathcal{O} \into Y$. Let $y\in Y$ be in the closure of the positive Weyl chamber with respect to $\Delta$ and $\hat{y} \in \mathcal{O}$. 
\begin{enumerate}
\item[(i)] The algebra $\mca{H}_{I, y}$ acts on the one-dimensional vector space $\C\cdot c(\varphi(y))\otimes \Theta_{\s_y}*\mbm{1}_{I}$ via the sign character $\varepsilon_{y}$. 
\item[(ii)] The $\mca{H}_{I, y}$-module map $\varepsilon_y \rightarrow \mca{V}_{\kappa, \mathcal{O}} \otimes_{\mca{H}_\kappa}\mca{H}*\mbm{1}_{I}$  induces isomorphisms of $\mca{H}_{I}$-modules
\begin{equation*}
\begin{tikzcd}
 \varepsilon_y \otimes_{\mca{H}_{I, y}} \mca{H}_{I} \ar[r] & \mca{V}_{\kappa, \mathcal{O}}\otimes_{\mca{H}_\kappa}\mca{H}*\mbm{1}_{I} \ar[r, "{\pmb{\gamma}}"] & \mca{V}^I_\mca{O},
\end{tikzcd}
\end{equation*}
where the first map is given by $1\mapsto c(\varphi(y))\otimes \Theta_{\s_y}*\mbm{1}_{I}$.
\end{enumerate}
\end{thm}

\begin{proof} 
Since $\mca{O}$ is splitting, we see that $\tilde{W}_{\aff, y} = W_y$ is a standard parabolic Weyl subgroup. Now the fact that $\mca{H}_{I, y}$ acts on $\C \cdot c(\varphi(y))\otimes \Theta_{\s_y}*\mbm{1}_{I}$ follows from Lemma \ref{WeqSplitting}, \eqref{BR3*I} of the Bernstein relations, and the equality \eqref{FinGGAction}. 
It follows from  \eqref{FinGGAction} and Lemma \ref{WeqSplitting} (ii) that each simple reflection in $\mca{H}_{I, y}$ acts by the sign character. Specifically, for $\alpha\in \Delta$ such that $w_{\alpha}\in \tilde{W}_{\aff, y}$ we have 
$$\varphi(y)(h_{\alpha}(u))=(\varpi,u)_n^{B_{Q}(y,\alpha^{\vee})}=1,$$
where the last equality is due to the fact that $B_{Q}(y,\alpha^{\vee}) = Q(\alpha^\vee)\angb{\alpha}{y} =0$.

Thus, we have an $\mca{H}_{I}$-module homomorphism 
$$f: \varepsilon_y \otimes_{\mca{H}_{I, y}} \mca{H}_{I}\longrightarrow \mca{V}_{\kappa, \mathcal{O}}\otimes_{\mca{H}_\kappa}\mca{H}*\mbm{1}_{I}$$ defined by 
\begin{equation*}
1\otimes h\mapsto c(\varphi(y))\otimes \Theta_{\s_y}*h.
\end{equation*}
We claim that this map is an isomorphism of $\mca{H}_{I}$-modules. To prove this, we show that this map is upper triangular with respect to a natural choice of $\mathcal{A}$-bases.

For each coset $W_{y}w$ we choose the element of minimal length. Let $\set{w_j}_j$ be the set of these minimal elements. Let $\wt{w}_{j}\in K\cap \wt{N(T)}$ represent $w_{j}$. Then the Bernstein decomposition implies that 
$$\set{1\otimes \mathcal{T}_{\wt{w}_{j}}^{I}: 1\lest j\lest \val{W_{\chi}\backslash W}}$$ 
is an $\mathcal{A}$-basis for $\varepsilon_y \otimes_{\mca{H}_{I, y}} \mca{H}_{I}$.

Writing $y_{j}:=w_{j}^{-1}\cdot y$, the set 
$$\set{c(\varphi(y_{j}))\otimes \Theta_{\s_{y_j}}*\mbm{1}_{I}: 1\lest j \lest \val{W_{\chi}\backslash W}}$$
is an $\mathcal{A}$-basis for $\mca{V}_{\kappa, \mathcal{O}}\otimes_{\mca{H}_\kappa}\mca{H}*\mbm{1}_{I}$.
Under the map $\varepsilon_y \otimes_{\mca{H}_{I, y}} \mca{H}_{I}\rightarrow \mca{V}_{\kappa, \mathcal{O}}\otimes_{\mca{H}_\kappa}\mca{H}*\mbm{1}_{I}$ we have 
$$1\otimes \mathcal{T}_{\wt{w}_{j}}^{I}\mapsto c(\varphi(y))\otimes \Theta_{\s_y} \mathcal{T}_{\wt{w}_{j}}*\mbm{1}_{I}.$$
By Lemma \ref{BRandA}, we have
\begin{equation*}
c(\varphi(y))\otimes \Theta_{\s_y} \mathcal{T}_{\wt{w}_{j}}*\mbm{1}_{I}\in c(\varphi(y))\otimes \big( \mathcal{T}_{\wt{w}_{j}}*\Theta_{\wt{w}_{j}^{-1}\cdot \s_y}*\mbm{1}_{I}+\sum_{w'<w_{j}}\mathcal{T}_{\wt{w}'}*\Theta_{(\wt{w}')^{-1}\cdot \s_y}\mathcal{A} \big),
\end{equation*}
and applying \eqref{FinGGAction} again gives
\begin{equation}\label{uppertri}
\begin{aligned}
& c(\varphi(y))\otimes \big( \mathcal{T}_{\wt{w}_{j}}*\Theta_{\wt{w}_{j}^{-1}\cdot \s_y}*\mbm{1}_{I}+\sum_{w'<w_{j}}\mathcal{T}_{\wt{w}'}*\Theta_{(\wt{w}')^{-1}\cdot \s_y}\mathcal{A}\big) \\
=& (-1)^{\ell(w_{j})}c(\varphi(y_{j}))\otimes \Theta_{\wt{w}_{j}^{-1}\cdot \s_y}*\mbm{1}_{I}+\sum_{w'<w_{j}}c(\varphi((w')^{-1}\cdot y))\otimes\Theta_{(\wt{w}')^{-1}\cdot \s_y}\mathcal{A}.
\end{aligned}
\end{equation}
Now it follows from \eqref{uppertri} and Lemma \ref{WeqSplitting} (iii) that the map $f$ is upper triangular. Furthermore, Lemma \ref{WeqSplitting} (iii) also implies that the diagonal entries are nonzero. Thus the map $f$ is an isomorphism of $\mca{H}_{I}$-modules.
\end{proof}

The trivial orbit $\mca{O}_0$ and any free orbit in $\msc{X}_{Q,n}$ are always splitting. In this case, Theorem \ref{Wequi} gives the following.
\begin{cor}\label{O-triv}
Let $\mathcal{O}_0$ be the orbit corresponding to the trivial character of $T_\kappa$. Then we have the isomorphism
\begin{equation*}
\varepsilon_W \otimes_{\mca{H}_W}\mca{H}_{I} \simeq \mca{V}_{\kappa, \mathcal{O}_0}\otimes_{\mca{H}_\kappa}\mca{H}*\mbm{1}_{I}\simeq \mca{V}^I_{\mca{O}_0}.
\end{equation*}
of $\mca{H}_I$-modules.
\end{cor}

\begin{cor} \label{O-free}
Let $\mathcal{O} \subset \msc{X}_{Q,n}$ be a free $W$-orbit. Then 
\begin{equation*}
\mca{H}_I \simeq \mca{V}_{\kappa, \mathcal{O}}\otimes_{\mca{H}_\kappa}\mca{H} *\mbm{1}_{I} \simeq \mca{V}^I_\mca{O}
\end{equation*}
as $\mca{H}_I$-modules.
\end{cor}

\begin{dfn} \label{D:S-ppty}
A $W$-orbit $\mca{O} \subset \msc{X}_{Q,n}$ is said to have the S-property if there exists $y\in Y$ with $\hat{y}\in \mca{O}$ and a character $\mu_{y}: \mca{H}_{I, y} \to \C^\times$ such that
$$\mca{V}^I_{\mca{O}_{\hat{y}}} \simeq \mu_{y} \otimes_{\mca{H}_{I, y}} \mca{H}_I.$$
\end{dfn}
Theorem \ref{Wequi} states that every splitting orbit of $\msc{X}_{Q,n}$ has the S-property.

\section{Subclasses of covers and splitting $W$-orbits} \label{S:splO}

 In this section, we give some detailed discussion on the $W$-orbits in $\msc{X}_{Q,n}$ which are splitting. In particular, we determine some covers  for which every $W$-orbit in $\msc{X}_{Q,n}$ is splitting: such examples include the ``oasitic" covers of semisimple simply-connected groups and the Kazhdan--Patterson and Savin covers.

\subsection{Saturated and oasitic covers}
We briefly discuss several special subclasses of covering groups, which exhibit better properties.
For $Q: Y \to \Z$ and the associated bilinear form $B_Q$, we denote
$$\det(B_Q)=\det [B_Q(v_i, v_j)]_{1\lest i, j \lest r} \in \Z,$$
where $\set{v_i: 1\lest i \lest r} \subset Y$ is any choice of $\Z$-basis. Here $\det(B_Q) \in \Z$ is well-defined and if $G$ is almost simple and simply-connected, then
\begin{equation} \label{F:detBQ}
\det(B_Q) = \msc{I}_\Delta \cdot \prod_{\alpha \in \Delta} Q(\alpha^\vee),
\end{equation}
where $\msc{I}_\Delta$ is the index of the root system of $G$.

\begin{dfn} \label{D:sat}
A covering group $\wt{G}^{(n)}$ of $G$ is called
\begin{enumerate}
\item[--] saturated (cf. \cite[Definition 2.1]{Ga6}) if  $Y^{sc} \cap Y_{Q,n} = Y_{Q,n}^{sc}$ (note, the inclusion $\supset$ always holds);
\item[--] aligned if $Y_{Q,n}^{sc} =  n' \cdot Y^{sc}$ for some $n' \in \N$;
\item[--] very saturated if it is both saturated and aligned;
\item[--] oasitic if $\gcd(n, c_\alpha^\sharp)=1=\gcd(n, \det(B_Q)) $ for every $\alpha\in \Delta$, where $\alpha_\sharp^\vee = \sum_{\alpha \in \Delta} c_\alpha^\sharp \alpha^\vee$ is the highest coroot.
\end{enumerate}
\end{dfn}

If $G$ is almost simple and simply-connected, then $\wt{G}$ is saturated if and only if its dual group $\wt{G}^\vee$ is of adjoint type, since $Z(\wt{G}^\vee)=\Hom(Y_{Q,n}/Y_{Q,n}^{sc}, \C^\times)$ by definition.

\begin{lm}  \label{L:v-sat}
Let $\wt{G}^{(n)}$ be an $n$-fold cover of an almost simple simply-connected $G$ with $Q(\alpha^\vee)=1$ for any short $\alpha^\vee$. Then $\wt{G}^{(n)}$ is very saturated if and only if the following holds:
\begin{enumerate}
\item[(i)] $\gcd(n, \msc{I}_\Delta)=1$ for type $A_r, B_r, C_r, D_r, E_6, E_7$;
\item[(ii)] $2\nmid n$ for type $F_4$ and $3\nmid n$ for $G_2$;
\item[(iii)] any $n$ for $E_8$.
\end{enumerate}
When such $\wt{G}^{(n)}$ is very saturated, we have $Y_{Q,n} =Y_{Q,n}^{sc} = nY$. Moreover, except for type $F_4$ and $G_2$ we have that $\wt{G}^{(n)}$ is very saturated if and only if $\wt{G}^\vee \simeq G^\vee$.
\end{lm}
\begin{proof}
Assertion (i) follows from \cite[\S 2.7]{We6}. For (ii), we note that $\wt{F}_4^{(n)}$ (resp. $\wt{G}_2$) is aligned if and only if $2\nmid n$  (resp. $3\nmid n$). Statement (iii) is clear. Last assertion follows from loc. cit. and the fact that $\wt{E}_8^{(n)}$ is always aligned and that $\wt{E}_8^\vee = E_8 = E_8^\vee$ always hold.
\end{proof}

Consider a cover $\wt{G}^{(n)}$ of $G$ as in Lemma \ref{L:v-sat}. It is easy to check when $\wt{G}^{(n)}$ is oasitic by using \eqref{F:detBQ}. For such $G$, we tabulate saturated, very saturated and oasitic covers in Table 1 and Table 2.

\begin{table}[H]  \label{T:3T1}
\caption{For classical $G$}
\vskip 5pt
\renewcommand{\arraystretch}{1.3}
\begin{tabular}{|c|c|c|c|c|c|c|c|c|c|c|}
\hline
 & $A_r$  &  $B_r$ & $C_r$  & $D_r$   \\
\hline
saturated & $\gcd(n, r+1)=1$ & $n$ odd, or  & $n$ odd & $n$ odd    \\ 
& & $n\in 4\Z +2$ for $r$ odd & & \\
\hline
very saturated & $\gcd(n, r+1)=1$ & $n$ odd  & $n$ odd & $n$ odd    \\ 
\hline
oasitic & $\gcd(n, r+1)=1$ & $n$ odd  & $n$ odd & $n$ odd    \\ 
\hline
\end{tabular}
\end{table}
\vskip 10pt

\begin{table}[H]  \label{T:3T2}
\caption{For exceptional  $G$}
\vskip 5pt
\renewcommand{\arraystretch}{1.3}
\begin{tabular}{|c|c|c|c|c|c|c|c|c|c|c|}
\hline
 & $E_6$  &  $E_7$ & $E_8$  & $F_4$ & $G_2$   \\
\hline
saturated & $3\nmid n$ & $2\nmid n$& all $n$ & all $n$ & all $n$   \\ 
\hline
very saturated & $3\nmid n$ & $2\nmid n$& all $n$ & $2\nmid n$ & $3\nmid n$   \\ 
\hline
oasitic & $2, 3\nmid n$ & $2, 3 \nmid n$& $2, 3, 5 \nmid n$ & $2, 3 \nmid n$ & $2, 3\nmid n$   \\  
\hline
\end{tabular}
\end{table}
\vskip 10pt

It is clear that if $G$ is almost simple and simply-connected, we have the following inclusions for covers of $G$ with $Q(\alpha^\vee)=1$: 
\begin{equation} \label{C:chain1}
\set{\text{oasitic covers}} \subseteq \set{\text{very saturated covers}} \subseteq \set{\text{saturated covers}} .
\end{equation}

\subsection{Splitting $W$-orbits}
To determine the splitting $W$-orbits, we start with the following:
\begin{prop} \label{P:Osplt}
Let $\wt{G}^{(n)}$ be a cover of an almost simple simply-connected $G$ with $Q(\alpha^\vee)=1$ for any short $\alpha^\vee$. 
Assume $\gcd(n, \msc{I}_\Delta)=1$. Then every $W$-orbit $\mca{O}_{\hat{y}} \subset \msc{X}_{Q,n}$ such that the stabilizer $W_{\hat{y}} \subset W$ is a parabolic subgroup is splitting.
\end{prop}
\begin{proof}
Let $\mca{O}_{\hat{y}} \subset \msc{X}_{Q,n}$ be an orbit such that
$$W_{\hat{y}} = W_J,$$
the Weyl group generated by a certain subset $J \subset \Delta$ depending on $\hat{y}$.

Taking an arbitrary representative $y\in Y$ of $\hat{y}$, we have 
$$w_\alpha(y) - y \in Y_{Q,n}=n\cdot Y.$$
Let $\set{\omega_\alpha:\ \alpha \in \Delta}$ be the set of coweights, which constitutes a $\Z$-basis of the coweight lattice $P$. If we write $y=\sum_{\alpha \in \Delta} y_\alpha \omega_\alpha$, then we have
$$n|y_\alpha \text{ for every } \alpha \in J.$$
We want to find a $y_{Q,n} \in Y_{Q,n}$ such that
\begin{equation} \label{E:yQn}
w_\alpha(y - y_{Q,n}) = y- y_{Q,n} \text{ for every } \alpha \in J,
\end{equation}
in which case we have $W_J \subseteq W_{y - y_{Q,n}} \subseteq W_{\hat{y}} = W_J$ and therefore $W_{y - y_{Q,n}} = W_J$ as desired.

If we write $y_{Q,n} = \sum_{\alpha\in \Delta} d_\alpha \omega_\alpha \in Y_{Q,n}$, then \eqref{E:yQn} entails that 
$$d_\alpha = y_\alpha \text{ for every } \alpha \in J.$$
Thus, it suffices to show that there exists $d_\alpha \in \Z$ for $\alpha \notin J$ such that
$$\sum_{\alpha \in J} y_\alpha \omega_\alpha + \sum_{\alpha \notin J} d_\alpha \omega_\alpha \in Y_{Q,n}.$$
Since $\gcd(n, \msc{I}_\Delta) =1$, there exists $a, b \in \Z -\set{0}$ such that
$$na - \msc{I}_\Delta \cdot b =1.$$
Consider the sublattice 
$$P_J =\set{\sum_{\alpha} y_\alpha \omega_\alpha: n|y_\alpha \text{ for every } \alpha \in J} \subset P.$$
Define a $\Z$-homomorphism
$$\varphi_J: P \longrightarrow P\otimes \Z[1/n]$$
given by $\varphi_J(\omega_\alpha)=w_\alpha/n, \alpha \in J$ and $\varphi_J(\omega_\alpha) = a\cdot \omega_\alpha$ for $\alpha\notin J$.
Clearly, $\varphi_J|_{P_J}: P_J \to P$ is well-defined. Moreover, since $n$ and $\msc{I}_\Delta$ are coprime, it is easy to see that $\varphi_J(P_J \cap Y) \subset Y$, i.e., the following commutative diagram
$$\begin{tikzcd}
P_J \cap Y \ar[r, hook] \ar[d, "{\varphi_J}"] & P_J \ar[r, hook] \ar[d, "{\varphi_J}"] & P \ar[d, "{\varphi_J}"] \\
Y \ar[r, hook] & P \ar[r, hook] & P\otimes \Z[1/n]. 
\end{tikzcd}$$
is well-defined. Now, if we take $d_\alpha = n a y_\alpha$ for $\alpha \notin J$, then
$$y_{Q,n} = n\cdot \left( \sum_{\alpha \in J} (y_\alpha/n) \omega_\alpha + \sum_{\alpha \notin J} ay_\alpha \omega_\alpha \right) = n\cdot \varphi_J(y) \in n\cdot Y;$$
this completes the proof.
\end{proof}

It follows from Table 1 and Table 2 that the covers of $G$ in  Proposition \ref{P:Osplt} satisfying $\gcd(n, \msc{I}_\Delta)=1$ are exactly the saturated covers, except the only case of $\wt{B}_r^{(n)}$ with  $n\in 4\Z +2$ and $r$ odd, for which $\gcd(n, \msc{I}_\Delta)=2$.

\begin{cor} \label{C:aOsplt}
Let $\wt{G}^{(n)}$ be an oasitic cover of an almost simple simply connected $G$ with $Q(\alpha^\vee)=1$ for every short coroot $\alpha^\vee$. Then every $W$-orbit $\mca{O} \subset \msc{X}_{Q,n}$ is splitting.
\end{cor}
\begin{proof}
For an oasitic cover $\wt{G}^{(n)}$ we have 
$$\gcd(n, \msc{I}_\Delta) =1 =\gcd(n, c_\alpha^\sharp)$$
for all $c_\alpha^\sharp$ as in Definition \ref{D:sat}, and it thus follows from \cite[Proposition 4.1]{Som1} that for every $\hat{y}\in \msc{X}_{Q,n}$ the stabilizer subgroup $W_{\hat{y}}$ is conjugate to  a parabolic Weyl subgroup. The result then follows from Proposition \ref{P:Osplt}.
\end{proof}

\begin{eg} \label{E:KPS}
Consider $\GL_r$ and its cocharacter lattice $Y$ given with the standard $\Z$-basis $\set{e_i: 1\lest i \lest r}$. A Brylinski--Deligne cover $\wt{\GL}_r^{(n)}$ of $\GL_r$ is associated to $\mbf{p}, \mbf{q} \in \Z$ such that
$$B(e_i, e_j) =\begin{cases}
2\mbf{p} & \text{ if } i=j,\\
\mbf{q} & \text{ if } i \ne j.
\end{cases}
$$
We have $Q(\alpha^\vee)=2\mbf{p} - \mbf{q}$. Every BD cover $\wt{\GL}_r$ is saturated, i.e., $Y_{Q,n} \cap Y^{sc} = Y_{Q,n}^{sc} = n_\alpha \cdot Y^{sc}$. The Kazhdan--Patterson covers are those with $Q(\alpha^\vee)=-1$, while Savin's nice cover is the one with $\mbf{p}=-1, \mbf{q}=0$.

Consider a  BD cover $\wt{\GL}_r$ such that 
\begin{equation} \label{I:nice}
n_\alpha Y \subseteq Y_{Q,n}.
\end{equation}
 We have the natural quotient maps
 $$\begin{tikzcd}
 Y \ar[r, two heads, "f"] & Y/n_\alpha Y \ar[r, two heads, "h"] & Y/Y_{Q,n} =\msc{X}_{Q,n}.
 \end{tikzcd}$$
 For every $y^* \in Y/n_\alpha Y$, it is easy to see that ${\rm Stab}_W(y^*) = {\rm Stab}_W(h(y^*))$, since $\wt{\GL}_r$ is always saturated. Thus, every $W$-orbit $\mca{O} \subset \msc{X}_{Q,n}$ has a $W$-equivariant section of $h$. It is also clear that every $W$-orbit in $Y/n_\alpha Y$ has a $W$-equivariant section of $f$. Hence, for such $\wt{\GL}_r$, every orbit of $\msc{X}_{Q,n}$ is splitting. The Kazhdan--Patterson covers and Savin covers both satisfy \eqref{I:nice} above.
 
 We also note that if \eqref{I:nice} fails, then it is possible that certain $\mca{O} \subset \msc{X}_{Q,n}$ is not splitting. The cover $\wt{\GL}_2^{(4)}$ associated with $\mbf{p}=1, \mbf{q}=2$ is such an example.
\end{eg}

\begin{rmk}
If $\wt{G}$ is very saturated but not oasitic, then it is possible that certain $\mca{O} \subset \msc{X}_{Q,n}$ is not splitting. For example, consider the double cover $\wt{G}_2^{(2)}$ of the exceptional $G_2$. In the Bourbaki notations, the element $\hat{\omega}_1 \in \msc{X}_{Q,n}$ of the coweight (also a coroot) $\omega_1\in Y$ has stabilizer subgroup
$$W_{\hat{\omega}_1} \simeq \set{1, w_{\alpha_2}} \times \set{1, w_{\omega_1}},$$
where the right hand is direct product. Since $W_{\hat{\omega}_1}$ is not isomorphic to a parabolic Weyl subgroup, we see that $\mca{O}_{\hat{\omega}_1}$ is not splitting.
\end{rmk}

\subsection{Covers of $\SL_2$}
In this subsection we focus on covers of $\SL(2)$ and determine the $\mca{H}_{I}$-module structure of $\mca{V}_\mca{O}^I \simeq \mca{V}_{\kappa, \mathcal{O}}\otimes_{\mca{H}_\kappa}\mca{H}*\mbm{1}_{I}$ for any $W$-orbit $\mathcal{O}\subseteq \mathscr{X}_{Q,n}$.

Let $\wt{\SL}_2$ be the $n$-fold cover with $Q(\alpha^\vee)=-1$. In this case
$$Y_{Q,n}=\Z[\alpha_{Q,n}^\vee]=\Z[n^* \alpha^\vee] \text{ where } n^*=\frac{n}{\gcd(n,2Q(\alpha^{\vee}))}.$$
Moreover,
$$\msc{X}_{Q,n}=\set{j\alpha^{\vee}: j\in[-n^*/2, n^*/2]\cap\Z}.$$
Note that if $n^{*}$ is even, then $n^*/2= - n^*/2 \in \msc{X}_{Q,n}$. With these representatives we can easily describe the $W$-orbits. We consider two cases according to the parity of $n^*$.

If $n^{*}$ is odd, then the $W$-orbits of $\mathscr{X}_{Q,n}$ consist of the trivial orbit $\set{0}$ and $\floor{n^*/2}$-many free orbits of the form
$$\set{j\alpha^\vee, -j\alpha^\vee}, \text{ where } j \in [1, n^*/2] \cap \Z.$$
In this case, by Propositions \ref{O-triv} and \ref{O-free} we have a complete description of the $\mca{H}_{I}$-module structure of $\mca{V}^I_\mca{O} \simeq \mca{V}_{\kappa, \mathcal{O}}\otimes_{\mca{H}_\kappa}\mca{H}*\mbm{1}_{I}$ for any $W$-orbit $\mathcal{O} \subset \msc{X}_{Q,n}$.

Now assume $n^{*}$ is even. There are two trivial $W$-orbits in $\mathscr{X}_{Q,n}$:
$$\set{0} \text{ and } \set{n^*\alpha^\vee/2},$$
while
$$\set{j\alpha^{\vee}, -j\alpha^\vee}, j\in [1,n^*/2-1]\cap \Z$$
constitute the free $W$-orbits in $\msc{X}_{Q,n}$.  Again, Propositions \ref{O-triv} and \ref{O-free} gives a complete description of the $\mca{H}_I$-module  $\mca{V}^I_\mca{O} \simeq \mca{V}_{\kappa, \mathcal{O}}\otimes_{\mca{H}_\kappa}\mca{H}*\mbm{1}_{I}$ except for the orbit 
$$\mca{O}_{n^*/2}:=\set{n^*\alpha^\vee/2} \subset \msc{X}_{Q,n}.$$

We will describe $\mca{V}^I_{\mca{O}_{n^*/2}}\simeq \mca{V}_{\kappa, \mathcal{O}_{n^*/2}}\otimes_{\mca{H}_\kappa}\mca{H}*\mbm{1}_{I}$ directly. The result of our analysis uncovers a new phenomenon that does not follow from our general results proved above. Note that if $n^*$ is even, then necessarily $2n^{*}=n_{\alpha}$.

\begin{lm}\label{finAff}
The element $h:=\mathcal{T}_{\alpha}^{I}\Theta_{\s_{-n^*\alpha^{\vee}}}^{I}\in\mca{H}_{I}$ satisfies $h^{2}=q$.
\end{lm}

\begin{proof}
Let $y=-n^{*}\alpha^{\vee}$. By Lemma \ref{I*BR3*I} we have
\begin{equation*}
\Theta_{\s_y}^{I}\mathcal{T}_{\alpha}^{I}=\mathcal{T}_{\alpha}^{I}\Theta_{w_{\alpha}(-1)\cdot\s_y}^{I}-(q-1)\Theta_{\s_y\wt{h}_{\alpha}(\varpi^{2n^{*}})}^{I}.
\end{equation*} 
Thus 
\begin{equation*}
h^{2}=(T_{\alpha}^{I})^{2}\Theta_{(w_{\alpha}(-1)\cdot\s_y) \s_y}^{I}-(q-1)\mathcal{T}_{\alpha}^{I}\Theta_{\s_y\wt{h}_{\alpha}(\varpi^{2n^{*}})\s_y}^{I}.
\end{equation*}
The identity $w_{\alpha}(-1)\cdot\s_y=\s_y\wt{h}_{\alpha}(\varpi^{2n^{*}})$ from \eqref{Rel7} implies that
\begin{align*}
h^{2}=&[(\mathcal{T}_{\alpha}^{I})^{2}-(q-1)\mathcal{T}_{\alpha}^{I}]\Theta_{\s_y\wt{h}_{\alpha}(\varpi^{2n^{*}})\s_y}^{I}\\
=&q\Theta_{\s_y\wt{h}_{\alpha}(\varpi^{2n^{*}})\s_y}^{I}=q,
\end{align*}
where the last equality holds since $\s_y\wt{h}_{\alpha}(\varpi^{2n^{*}})\s_y=1\in \overline{T}$ by \eqref{Rel2}.
\end{proof}

Let $\mathcal{T}':=q^{-1/2} \cdot T_{\alpha}^{I}\Theta_{\s_{-n^*\alpha^\vee}}^{I}$ and consider the $\C$-subalgebra 
$$(\mca{H}_I)_0':=\langle \mca{T}' \rangle \subset \mca{H}_I$$
generated by $\mca{T}'$.  Lemma \ref{finAff} implies $(\mathcal{T}')^{2}=1$ and thus
$$(\mca{H}_I)_0' \simeq \C[\Z/2\Z].$$
Let $\sigma: (\mca{H}_I)_0' \to \C$ be the character given by
$$\sigma(\mathcal{T}')=q^{-1/2} \cdot \g_{\alpha}(\psi,\varphi(-n^*\alpha^\vee/2)) \in \set{\pm1}.$$ 

\begin{prop}\label{WaffSplit}
Let $y=\frac{-n^{*}}{2}\alpha^{\vee}$. The $\C$-vector space map 
$$\sigma\longrightarrow \mca{V}_{\kappa, \mathcal{O}_{n^*/2}}\otimes_{\mca{H}_\kappa}\mca{H}*\mbm{1}_{I}$$
defined by $1\mapsto c(\varphi(y))\otimes \Theta_{\s_y}*\mbm{1}_{I}$ is a map of $(\mca{H}_{I})_{0}'$-modules. Moreover, this map induces an isomorphism 
$$\sigma\otimes_{(\mca{H}_{I})_{0}'}\mca{H}_{I}\rightarrow \mca{V}_{\kappa, \mathcal{O}_{n^*/2}}\otimes_{\mca{H}_\kappa}\mca{H}*\mbm{1}_{I} \stackrel{\pmb{\gamma}}{\rightarrow} \mca{V}^I_{\mca{O}_{n^*/2}}$$
of $\mca{H}_{I}$-modules.
\end{prop}
\begin{proof}
Using the Bernstein relation in \eqref{BR3}, we have
$$
c(\varphi(y))\otimes \Theta_{\s_y}*\mbm{1}_{I}*\mathcal{T}_{\alpha}^{I}=\g_\alpha(\psi,\varphi(y))\cdot c(\varphi(y))\otimes \Theta_{w_{\alpha}(-1)\cdot \s_y}*\mbm{1}_{I}.
$$
Thus
$$
c(\varphi(y))\otimes \Theta_{\s_y}*\mbm{1}_{I}*\mathcal{T}'=\frac{\g_{\alpha}(\psi,\varphi(y))}{q^{1/2}}c(\varphi(y))\otimes \Theta_{(w_{\alpha}(-1)\cdot \s_y)\s_{-n^*\alpha^\vee}}*\mbm{1}_{I}.
$$
Since $w_{\alpha}(-1)\cdot \s_y=\s_yh_{\alpha}(\varpi^{n^{*}})$ by \eqref{Rel7} and $h_{\alpha}(\varpi^{n^{*}})=\s(\alpha^{\vee}(\varpi^{n^{*}}))$, it follows that $(w_{\alpha}(-1)\cdot \s_y) \s(-n^{*}\alpha^{\vee})(\varpi))=\s_y$.
Hence,
\begin{equation*}
c(\varphi(y))\otimes \Theta_{\s_y}*\mbm{1}_{I}*\mathcal{T}'=\frac{\g_{\alpha}(\psi,\varphi(y))}{q^{1/2}}c(\varphi(y))\otimes \Theta_{\s_y}*\mbm{1}_{I}.
\end{equation*}
This shows that the map $\sigma\rightarrow \mca{V}_{\kappa, \mathcal{O}}\otimes_{\mca{H}_\kappa}\mca{H}*\mbm{1}_{I}$ defined by $1\mapsto c(\varphi(y))\otimes \Theta_{\s_y}*\mbm{1}_{I}$ is a map of $(\mca{H}_{I})_{0}'$-modules. The induced map of $\mca{H}_{I}$-modules is an isomorphism because it takes the $\mathcal{A}$-basis $\set{1\otimes \Theta_{\s_{kn^*\alpha^\vee}}^I: k\in \Z }$ to the $\mathcal{A}$-basis $\set{c(\varphi(y))\otimes \Theta_{\s_y}*\mbm{1}_{I}*\Theta_{\s_{kn^*\alpha^\vee}}^{I} }$.
\end{proof}

The orbit $\mca{O}_{n^*/2} \subset \msc{X}_{Q,n}$ does not possess a $W$-equivariant splitting. Thus, Proposition \ref{WaffSplit} does not follow from Theorem \ref{Wequi}. This suggests that splittings of other finite subgroups of the modified extended Weyl group $\tilde{W}_{\ex} = W\ltimes Y_{Q,n}$ should also be considered, besides $W$. We make a few remarks in this direction and further elaborate on the example of Proposition \ref{WaffSplit}.

For any $z\in Y\otimes \Q$, we set $\tilde{W}_{\ex, z}:={\rm Stab}_{\tilde{W}_\ex}(z)$. In particular, $\tilde{W}_{\ex, 0}=W$. For any $y\in Y$ we have
$$W_y = \tilde{W}_{\ex, 0} \cap \tilde{W}_{\ex, y}.$$
Now we may consider
$$\tilde{W}_{\ex, z, y}:=\tilde{W}_{\ex, z} \cap \tilde{W}_{\ex, y}.$$
Let $\eta:\tilde{W}_{\ex}= W\ltimes Y_{Q,n}\rightarrow W$ be the projection map. The quotient map 
$$Y \onto \msc{X}_{Q,n}$$
is equivariant with respect to $\eta: \tilde{W}_\ex \onto W$.  For every $z\in Y\otimes \Q$, the group $\tilde{W}_{\ex, z}$ is finite and $\eta: \tilde{W}_{\ex, z} \to \eta(\tilde{W}_{\ex, z} )$ is an isomorphism. Then an orbit $\mca{O}_y \subset \msc{X}_{Q,n}$ is said to be $\tilde{W}_{\ex, z}$-splitting if 
 \begin{equation} \label{E:gen-spl}
 \eta(\tilde{W}_{\ex, z, y}) = \eta(\tilde{W}_{\ex, z})_{\hat{y}}.
 \end{equation}
As we vary $z$, there are more possibilities for the above equality to hold.

Consider for instance the case of $n^* \in 2\Z$ and $\hat{y}=n^*\alpha^\vee/2$ for $\wt{\SL}_2$ as in Proposition \ref{WaffSplit}. If we take $z=0$, then the two sides of \eqref{E:gen-spl} are $\set{1}$ and $W$ respectively. On the other hand, if we take $z=-n^*\alpha^\vee/2$, then $\tilde{W}_{\ex, z} \subset \tilde{W}_{\ex}$ is of order 2, generated by the reflection fixing $-n^*\alpha^\vee/2$. In this case, the two sides of \eqref{E:gen-spl} are both equal to $\eta(\tilde{W}_{\ex, z, y})$, and there is a $\tilde{W}_{\ex, z}$-splitting for the orbit $\mca{O}_{n^*/2}$.

Such a  $\tilde{W}_{\ex, z}$-splitting is only implicitly used in the proof of Proposition \ref{WaffSplit}. For a general $\wt{G}$, it is not easy to identify explicitly (with concrete relations among the generators) the finite subalgebra of $\mca{H}_{I}$ associated to $\tilde{W}_{\ex, z}$. For $\wt{\SL}_2$, we essentially realized this by using Lemma \ref{finAff}. We hope to return to this problem in a future work.

\section{Metaplectic representation of Sahi--Stokeman--Venkateswaran} \label{S:SSV}

In this section we show that the representation 
$$\mca{V}^I\simeq \mca{V}_\kappa \otimes_{\mca{H}_\kappa}\mca{H}*\mbm{1}_{I}$$ 
embeds naturally into the metaplectic representation constructed by Sahi--Stokeman--Venkateswaran \cite{SSV21} for a particular choice of what they call representation parameters (see \cite[Definition 3.5]{SSV21}).

The proof of Lemma $\ref{OrbitA}$ gives that if $\set{y_j}_j \subset Y$ is a set of representatives for $\mathscr{X}_{Q,n}$, then $$\set{c(\varphi(y_{j}))\otimes \Theta_{\s_{y_j}}*\mbm{1}_{I}}_j$$
is an $\mathcal{A}$-basis for $\mca{V}_\kappa \otimes_{\mca{H}_\kappa}\mca{H}*\mbm{1}_{I}$. Using the Bernstein relations for $\mca{H}$ we can compute the action of $\mca{H}_{I}$ on $\mca{V}_\kappa \otimes_{\mca{H}_\kappa}\mca{H}*\mbm{1}_{I}$ in terms of the above basis $\set{c(\varphi(y_{j}))\otimes \Theta_{\s_{y_j}}*\mbm{1}_{I}}_j$. We begin with a description of the action of $\mathcal{A}$.

\begin{lm}\label{SSVA}
Let $y\in Y$ and $z\in Y_{Q,n}$. Then
\begin{equation*}
c(\varphi(y))\otimes \Theta_{\s_y}*\mbm{1}_{I}*\Theta_{\s_z}*\mbm{1}_{I}=\varepsilon^{D(y,z)}c(\phi(y))\cdot \otimes \Theta_{\s_{y+z}}*\mbm{1}_{I}.
\end{equation*}
\end{lm}

Now we describe the action of $\HH_W \subset \HH_I$. It suffices to consider the action of the elements $\mathcal{T}_{\alpha}=\mathcal{T}_{\wt{w}_{\alpha}(1)}$, where $\alpha\in \Delta$.

\begin{lm}\label{SSVB}
Let $y\in Y$ and $\alpha\in \Delta$. We write $\theta_{\s_y}:=c(\varphi(y))\otimes \Theta_{\s_y}$. Then
\begin{equation*}
\begin{aligned}
& \theta_{\s_y}*\mbm{1}_{I}*\mathcal{T}_{\alpha}*\mbm{1}_{I} \\
= & \begin{cases}
\g_{\alpha}(\psi,\chi)\theta_{w_{\alpha}(-1)\cdot \s_y}*\mbm{1}_{I}+(q-1) \sum_{ \substack{1-\angb{y}{\alpha}\lest j \lest 0\\ j\equiv 0\mod n_{\alpha} } } \varepsilon^{jQ(\alpha^{\vee})}\theta_{\s_yh_{\alpha}(\varpi^{j})}*\mbm{1}_{I} & \text{ if } \angb{y}{\alpha}> 0,\\
\g_{\alpha}(\psi,\chi)\theta_{w_{\alpha}(-1)\cdot \s_y}*\mbm{1}_{I} & \text{ if } \angb{y}{\alpha}= 0,\\
\g_{\alpha}(\psi,\chi)\theta_{w_{\alpha}(-1)\cdot \s_y}*\mbm{1}_{I}+(q-1) \sum_{ \substack{1\lest j \lest -\angb{y}{\alpha} \\ j\equiv 0\,\text{mod }n_{\alpha} } } \varepsilon^{jQ(\alpha^{\vee})}\theta_{\s_yh_{\alpha}(\varpi^{j})}*\mbm{1}_{I} & \text{ if } \angb{y}{\alpha}< 0.
\end{cases}
\end{aligned}
\end{equation*}
\end{lm}

\subsection{The SSV representation}
We briefly recall the construction of the metaplectic representation of $\mca{H}_I$ in \cite{SSV21} using notations in our paper.

For the rest of this section, we assume that $G$ is semisimple and almost simple. Denote by $P$ the lattice of coweights  with
$$P \supset Y \supset Y^{sc}.$$
Fix $Q: Y \to \Z$ an integral-valued Weyl-invariant quadratic form, and extend it uniquely to 
a rational-valued form
$$Q: P \to \Q.$$
One has a sublattice $P_{Q, n} \subset P$ given in \cite[\S 2.2]{SSV21}. The relations among these sublattices are as in
$$\begin{tikzcd}
P_{Q, n} \ar[r, hook] & P \\
Y_{Q,n} \ar[r, hook] \ar[u, hook] & Y \ar[u, hook] .
\end{tikzcd}$$
In particular, 
$$Y_{Q,n } \subseteq P_{Q, n} \cap Y,$$
where equality holds if $Y=Y^{sc}$, i.e., if $G$ is simply-connected. Furthermore, if $Y = Y^{sc}$, we have an embedding
$$Y/Y_{Q,n} \into P/P_{Q, n}.$$

One has the extended affine Hecke algebra $\mca{H}(P_{Q,n})$ associated with $P_{Q,n}$, and as a $\C$-isomorphism we get
$$\mca{H}(P_{Q,n}) \simeq \mca{H}_W \otimes_\C \C[P_{Q,n}],$$
where $\mca{H}_W$ is the finite Hecke algebra. One has an inclusion of algebras
$$\mca{H}_I \subset \mca{H}(P_{Q,n}).$$
There is a natural action of $\mca{H}(P_{Q,n})$ on $\C[P_{Q,n}]$ via transporting the $\mca{H}(P_{Q,n})$-action from the left hand side of the $\C[P_{Q,n}]$-module isomorphism
$$\mca{H}(P_{Q,n}) \otimes_{\mca{H}_W} \C \simeq \C[P_{Q,n}].$$
It is shown in \cite[Theorem 3.7]{SSV21} that this action of $\mca{H}(P_{Q,n})$ can be extended to the bigger space $\C[P] \supset \C[P_{Q,n}]$.  We let $(\pi, \C[P])$ denote this $\mca{H}(P_{Q,n})$-module. In brief, it is constructed as follows:
\begin{enumerate}
\item[(i)] Choose a $W$-stable subset $C \subset P$ such that $0 \in C$ and that the quotient 
$$P \onto P/P_{Q,n}$$
is surjective when restricted to $C$. In \cite{SSV21}, an explicit $C$ is given.
\item[(ii)] Consider the $\mca{H}(P_{Q,n})$-module
$$N_C:=\mca{H}(P_{Q,n}) \otimes_{\mca{H}_W} V_C,$$
where $V_C$ is the $\val{C}$-dimensional space affording the $\mca{H}_W$-action ``deformed" from the Weyl action, see \cite[\S 3.1]{SSV21}. 
Depending on certain parameters $\cc$, there is a natural $\C[P]$-module surjection
$$\Psi^\cc_C: N_C  \onto \C[P].$$
\item[(iii)] By choosing $\cc$ properly, it is shown that ${\rm Ker}(\Psi^\cc_C)$ is actually $\mca{H}(P_{Q,n})$-stable, and thus one has a $\mca{H}(P_{Q,n})$-module structure on $\C[P]$ by transporting that from $N_C$: this is the representation $\pi$.
\end{enumerate}
In the notation of \cite{SSV21}, the action of the generators of $\mca{H}(P_{Q,n})$ on $\C[P]$ are given as follows:
\begin{equation} \label{F:SSV}
\begin{aligned}
\pi(T_\alpha)x^\lambda & = (\mbf{k}_\alpha - \mbf{k}_\alpha^{-1}) \wt{\nabla}_\alpha(x^\lambda) + \mbf{p}_\alpha(\wt{\lambda}) x^{s_\alpha(\lambda)}, \\
\pi(x^\nu) x^\lambda & = x^{\lambda + \nu}
\end{aligned}
\end{equation}
for $\alpha\in \Delta, \lambda \in P$ and $\nu\in P_{Q,n}$.

By restricting the action of $\mca{H}(P_{Q,n})$ to the subalgebra $\mca{H}_I$, we get an $\mca{H}_I$-module still denoted by $(\pi, \C[P])$. For every $z\in P$, it follows from \eqref{F:SSV} that the subspace 
$$\C[Y]\cdot x^z \subset \C[P]$$
is stable under the $\mca{H}_I$-action and thus gives a $\mca{H}_I$-submodule $(\pi^z, \C[Y]\cdot x^z)$. Moreover, one has a decomposition
$$\pi = \bigoplus_{z\in P/Y} \pi^z$$
as $\mca{H}_I$-modules.

The representation $\pi^0$ has a further decomposition as follows. For each $W$-orbit $\mca{O} \subset \msc{X}_{Q,n}$, consider the vector space
$$\C[Y]_\mca{O}:= \bigoplus_{y\in \mca{O}} \C[Y_{Q,n}]\cdot x^y,$$
which is $\mca{H}_I$-stable by \eqref{F:SSV}. We let $(\pi^0_\mca{O}, \C[Y]_\mca{O})$ denote this, which gives a decomposition
$$\pi^0 = \bigoplus_{\mca{O} \subset \msc{X}_{Q,n}} \pi^0_\mca{O}.$$

\begin{thm}\label{T:GG=SSV0}
Let $\wt{G}$ be an $n$-fold cover of a semisimple group $G$. Assume $\varepsilon=(-1, \varpi)_n= 1$. Then we have an isomorphism 
$$\mca{V}^I_\mca{O} \simeq \pi^0_\mca{O}$$
of $\mca{H}_I$-modules for every orbit $\mca{O} \subset \msc{X}_{Q,n}$.
\end{thm}
\begin{proof}
In view of  Lemma \ref{SSVA} and Lemma \ref{SSVB}, checking the result essentially amounts to matching the notations used in our paper and those in \cite[Theorem 3.7]{SSV21}, i.e., \eqref{F:SSV} above. More precisely, the correspondence is given as follows:
\begin{itemize}
\item $\mbm{1}_{I}*\mathcal{T}_{\alpha}*\mbm{1}_{I}\longleftrightarrow \mathbf{k_{i}}T_\alpha$ for simple root $\alpha$,
\item $\mbm{1}_{I}*\Theta_{\s_{y'}}*\mbm{1}_{I}\longleftrightarrow x^{\nu}$ where $y' \leftrightarrow \nu\in Y_{Q,n}$,
\item $c(\varphi(y))\otimes \Theta_{\s_y}\longleftrightarrow x^{\lambda}$ where $y \leftrightarrow\lambda\in Y$.
\end{itemize}
Note that the assumption $\varepsilon=1$ is a technical assumption, enforced by the relation that the parameters $\mbf{k}_\alpha$ (essentially the Gauss sums) need to satisfy.
\end{proof}

One of the applications of the SSV representation $\pi$ is to provide a more natural proof of the Chinta--Gunnells action \cite{CG10}. This latter work assumes that $\mu_{2n}\subset F^\times$, which in particular implies that $\varepsilon = 1$. Using our formulas in Lemma \ref{SSVA} and \ref{SSVB}, the method of Sahi--Stokman--Venkateswaran can be easily adjusted to incorporate the slightly more general case without the assumption $\mu_{2n}\subset F^\times$.

\begin{rmk}
It is possible to prove Theorem \ref{Wequi} by using Theorem \ref{T:GG=SSV0} and refining several steps in the construction of $\pi$ in \cite{SSV21}, as outlined in (i)--(iii) above. Indeed, if $\mca{O}:=\mca{O}_{\hat{y}} \subset \msc{X}_{Q,n}$ is a splitting $W$-orbit, then one may take its splitting $C_\mca{O} \subset Y$. Examining the argument in \cite{SSV21} will give that $\pi_\mca{O}^0 \simeq \varepsilon_y \otimes_{\mca{H}_{I, y}} \mca{H}_I$, where $\mca{H}_{I, y} \subset \mca{H}_W$ is the subalgebra associated with the parabolic Weyl subgroup $W_{\hat{y}} \subset W$.
\end{rmk}

\subsection{A speculation}
It is natural to ask what role $\pi^z$ plays for general $z\in P$. We give a speculation as follows. Consider the vector space isomorphism 
$${\rm sh}_z: \C[Y] \longrightarrow \C[Y] \cdot x^z$$
given by $f \mapsto f\cdot x^z$. By transport of structure, one obtains a representation
$$({\rm sh}_z^*(\pi^z), \C[Y])$$
of $\mca{H}_I$.

For $z\in P$, let $\varpi^{z}=z(\varpi)$ be the element of the adjoint torus. 

\begin{conj} \label{C:pi-z}
For every $z\in P$, consider the character $({}^z\psi)(u)=\psi( \varpi^{-z} u \varpi^z)$. Then we have an isomorphism
$$({\rm ind}_{\mu_n U^-}^{\wt{G}} \epsilon \otimes {}^z\psi)^I \simeq {\rm sh}_z^*(\pi^z)$$
of $\mca{H}_I$-modules.
\end{conj}
The above holds for $z=0$. For $z=\rho$ it was also motivated from the following consideration. If $z=\rho$, then ${}^\rho\psi$ has conductor $O_F$. In this case, work from \cite{Ga7} suggests that (for oasitic covers of simply-connected $G$ for example) one has a decomposition
$$({\rm ind}_{\mu_n U^-}^{\wt{G}} \epsilon \otimes {}^z\psi)^I = \bigoplus_{\mca{O} \subset \msc{X}_{Q,n}} \varepsilon_\mca{O} \otimes_{\mca{H}_{I, \mca{O}}} \mca{H}_I = \big( \varepsilon_W \cdot \sigma_{[z]}^{\msc{X}}  \big) \otimes_{\mca{H}_W} \mca{H}_I,$$
where $\mca{O}$ is taken over the $(W, z)$-orbits in $\msc{X}_{Q,n}$. Also, $\varepsilon_W \cdot \sigma_{[z]}^
\msc{X}$ is a $\mca{H}_W$-module deformed from the corresponding representation of $W$. This seems to be compatible with ${\rm sh}_z^*(\pi^z)$ in view of \eqref{F:SSV}, and gives another motivation for Conjecture \ref{C:pi-z}.

\section{Explicit Whittaker dimensions} \label{S:Wdim}

The goal of this section and the next one is to apply Theorem \ref{Wequi} to compute explicit Whittaker dimensions of some Iwahori-spherical representations. We will eventually focus on irreducible constituents of regular unramified principal series in this section and those of a unitary unramified principal series in the next.

For every lattice $L \subset \mathscr{A}$ closed under the action of $W$, one has the canonical surjection
\begin{equation} \label{D:eta}
\eta: L \rtimes W \onto W.
\end{equation}
Though we might use the same $\eta$ for different $L$, there is no risk of confusion from the context. If $L$ is a root lattice and $W$ the associated Weyl group, then we write
$$W_{\aff} = L \rtimes W$$
for the affine Weyl group. We also denote 
$$\tilde{W}_{\aff} = L' \rtimes W,$$
where $L' \subset L$ is a modified root lattice (i.e., generated by $\set{k_\alpha \alpha: \alpha \text{ simple root of }  L}$ for some $k_\alpha \in \N$) giving rise to the same Weyl group as $L$ does. Again, the context makes it clear what $L'$ refers to.
 
 
 \subsection{Some permutation representations}
 Recall from \S \ref{SS:sdes} that for $z\in P$ and any Weyl-invariant sublattice $L\subset Y$ of the same rank as $Y$, we have the well-defined action $w[y]_z:=w(y+z) - z$ for $y\in Y/L$. Recall that  this gives a finite dimensional permutation representation
$$\sigma_{[z]}^{Y/L}: W \longrightarrow {\rm Perm}(Y/L).$$
It is easy to see that if $z - z' \in P \cap Y$, then 
$$\sigma_{[z]}^{Y/L} \simeq \sigma_{[z']}^{Y/L}.$$
Thus, there is a well-defined action of $P/(P\cap Y)$ on $\sigma_{[0]}^{Y/L}$. For later purpose, we are interested in the orbit of $\sigma_{[0]}^{Y/L}$ under this action of $P/(P \cap Y)$ and also the associated stabilizer subgroup.  A special case is given as follows:

\begin{lm} \label{L:01}
Assume $Y=Y^{sc}$ is the root lattice of a root system $R$ with simple roots $\Delta=\set{\alpha_i: \ 1\lest i \lest r}$. Let $S \subset Y$ be a sublattice associated to a root system $R'$ with simple roots $\Delta'=\set{k_i \alpha_i: \alpha_i \in \Delta'}$ for some $k_i \in \N$. Then for the above action of $P/Y$ on $\sigma_{[0]}^{Y/S}$, one has
$${\rm Stab}_{P/Y}(\sigma_{[0]}^{Y/S}) = \frac{P \cap (P(S) + Y)}{Y} \subseteq P/Y,$$
where $P(S) \subset Y\otimes \Q$ denotes the weight lattice of $S$. 
\end{lm}
\begin{proof}
Write $\sigma_{[z]}$ for $\sigma_{[z]}^{Y/S}$ for simplicity. Consider its character $\chi_{\sigma_{[z]}}$ which is integer-valued and satisfies $\chi_{\sigma_{[z]}}(w) \gest 0$ for every $w\in W$. We observe that if $\chi_{\sigma_{[z]}}(w) \ne 0$, then 
$$\chi_{\sigma_{[z]}}(w) =\chi_{\sigma_{[0]}}(w).$$
Thus, the following are equivalent:
\begin{enumerate}
\item[(i)] $\chi_{\sigma_{[z]}}(w) \ne 0$ for every $w\in W$;
\item[(ii)] $\chi_{\sigma_{[z]}}(w) = \chi_{\sigma_{[0]}}(w) \ne 0$ for every $w \in W$;
\item[(iii)] $\sigma_{[z]} \simeq \sigma_{[0]}$.
\end{enumerate}
One has $z\in \Stab_{P/Y}(\sigma_{[0]})$ if and only if
$$W \subseteq \bigcup_{y + z \in Y + z} \eta( {\rm Stab}_{S \rtimes W}(y + z)),$$
where $\eta: S\rtimes W \onto W$ is the natural surjection as in \eqref{D:eta}.
However, since the right hand is a union of reflection subgroups of $W$, and the Coxeter element of $W$ does not lie in any proper reflection subgroup (see \cite[\S 5]{Stem94}), the above inclusion holds if and only if $Y +z$ has a special point of the affine Weyl group $S \rtimes W$, i.e., if and only if 
$$z\in P(S) + Y.$$
This completes the proof.
\end{proof}

Specializing to  $L=Y_{Q,n}$, we have a permutation representation
$$\sigma_{[z]}^\msc{X}: W \longrightarrow {\rm Perm}(\msc{X}_{Q,n}).$$

\begin{prop} \label{P:O1}
Let $\wt{G}^\vee$ be a very saturated cover of an almost simple simply-connected $G$ with $Q(\alpha^\vee)=1$ for any short coroot $\alpha^\vee$. Then
$$\sigma_{[z]}^\msc{X} \simeq \sigma_{[0]}^\msc{X}$$
for every $z \in P/Y$. In fact, for every $z\in P$, there exists $y_z \in Y$ such that the bijective map 
$$\mfr{m}_z: \msc{X}_{Q,n} \longrightarrow \msc{X}_{Q,n}, \quad y \mapsto y + y_z$$
is equivariant with respect to the $w[-]_0=w(-)$ and $w[-]_z$ on the domain and codomain respectively. In particular, $\mfr{m}_{z}$ induces a bijection between the $W$-orbits  and $(W, z)$-orbits in $\msc{X}_{Q,n}$.
\end{prop}
\begin{proof}
For $G=F_4, E_8, G_2$ one has $P/Y=\set{0}$ and thus it suffices to consider $G$ of other types. Since $\wt{G}$ is very saturated, Lemma \ref{L:v-sat} gives
$$Y_{Q,n}=Y_{Q,n} = n \cdot Y$$
and also $\gcd(n, \msc{I}_\Delta)=1$. Note $\msc{I}_\Delta = \val{P/Y}$. Now  $P(Y_{Q,n}^{sc}) = nP$
and it follows from Lemma \ref{L:01} that
\begin{equation} \label{E:S-sig0}
{\rm Stab}_{P/Y}(\sigma_{[0]}) = (P \cap (nP + Y))/Y = P/Y,
\end{equation}
where the last equality follows from the fact that $\gcd(n, \msc{I}_\Delta)=1$. This shows that
$$\sigma^\msc{X}_{[z]} \simeq \sigma^\msc{X}_{[0]}.$$

Now for the second assertion, $\mfr{m}_z$ (depending on $y_z$) is equivariant if and only if
$$w[y+ y_z]_z = y_z + w(y) \in \msc{X}_{Q,n}$$
for all $w\in W$ and $y\in \msc{X}_{Q,n}$. This equality is equivalent to
$$w(y_z + z) \equiv y_z + z \mod Y_{Q,n}^{sc}=nY,$$
i.e., $y_z + z \in P(Y_{Q,n}^{sc})=nP$.
However, it follows from \eqref{E:S-sig0} that for every $z\in P$ there always exists $y_z\in Y$ such that $y_z + z \in P(Y_{Q,n}^{sc})$. This proves the equivariance of $\mfr{m}_z$, from which the last assertion is clear.
\end{proof}

\subsection{$z$-persistent orbits and covers}

We also recall the notion of a persistent cover as follows (see \cite[Definition 2.3]{Ga6}). In fact, we introduce a slightly more general version of $z$-persistency for any element $z \in P$ in the coweight lattice.  These $z$-persistent covers will give a refinement of the chain in \eqref{C:chain1} since they always contain the saturated covers, see \eqref{C:chain2} below. It is expected that $z$-persistent covers exhibit better behavour if we consider ${}^z\psi$-Whittaker space of Iwahori-spherical representations.

Consider the sublattices $Y_{Q,n}, Y_{Q,n}^{sc} \subset Y$ and
$$\msc{X}_{Q,n}^{sc}:=Y/Y_{Q,n}^{sc}, \ \msc{X}_{Q,n}=Y/Y_{Q,n},$$
on which the action $w[-]_z, z \in P/(P\cap Y)$ is well-defined. 
For any $y\in Y$, let $y^\dag$ and $\hat{y}$ denote its image in $\msc{X}_{Q,n}^{sc}$ and $\msc{X}_{Q,n}$ respectively. With respect to the action $w[-]_z$, we have the $(W, z)$-orbits $\mca{O}_{y^\dag} \subset \msc{X}_{Q,n}^{sc}$ and $\mca{O}_{\hat{y}} \subset \msc{X}_{Q,n}$, which are the images of $\mca{O}_y$ in the respective quotient spaces.

\begin{dfn} \label{D:z-per}
A $(W, z)$-orbit $\mca{O}_y \subset Y$ is called $z$-persistent if 
$${\rm Stab}_W(y^\dag; \msc{X}_{Q,n}^{sc}) = {\rm Stab}_W(\hat{y}; \msc{X}_{Q,n}).$$
A covering group $\wt{G}$ is called $z$-persistent if every $(W, z)$-orbit $\mca{O}_y$ is $z$-persistent.
\end{dfn}
If $\mca{O}_{\hat{y}} = \mca{O}_{\hat{x}} \subset \msc{X}_{Q,n}$, then $\mca{O}_y \subset Y$ is persistent if and only if $\mca{O}_x\subset Y$ is. The proof of this is identical to \cite[Corollary 2.5]{Ga6}. Thus, checking if a covering group $\wt{G}$ is $z$-persistent amounts to checking for a finite set of representative orbits of those in $\msc{X}_{Q,n}$.

 While persistency is a slightly technical condition, we note that a saturated cover is always $z$-persistent for any $z\in P$. Indeed, if $w[\hat{y}]_z = \hat{y}$, then 
 $$w[y]_z - y \in Y_{Q,n} \cap Y^{sc};$$
 but if $G$ is saturated, then $w[y]_z - y \in Y_{Q,n}^{sc}$, i.e., $w$ fixes $y^\dag \in \msc{X}_{Q,n}^{sc}$. This shows $\mca{O}_y$ is $z$-persistent. Thus, if $G$ is almost simple and simply-connected, then the chain \eqref{C:chain1} is refined to be:
\begin{equation} \label{C:chain2}
\set{\text{oasitic covers}} \subseteq \set{\text{very saturated covers}} \subseteq \set{\text{saturated covers}} \subseteq \set{\text{$z$-persistent covers}}.
\end{equation}
We note that the last three inclusions actually hold for arbitrary $G$.

\begin{eg}
Every Brylinski--Deligne cover of $\GL_r$ is saturated and thus $z$-persistent. On the other hand, the cover $\wt{\SL}_2^{(n)}$ associated with 
$$Q(\alpha) = -1$$
is saturated if and only if $n$ is odd. For $n\in 4\Z +2$, the cover $\wt{\SL}_2^{(n)}$ is  $0$-persistent but not $\rho$-persistent. On the other hand, if $4|n$, then $\wt{\SL}_2^{(n)}$ is $\rho$-persistent but not $0$-persistent.
\end{eg}
We expect that saturated and $z$-persistent covers form subclasses of covers whose representation theory is more accessible, especially the part pertaining to Whittaker models. For instance, 
the ${}^\rho\psi$-Whittaker dimension of a theta representation  of $\wt{\SL}_2^{(n)}$ (for $n\in 4\Z +2$) depends sensitively on the choice of central characters of $Z(\wt{T})$, since it is not $(W, \rho)$-persistent. For odd-fold covers of $\SL_2$, such subtle dependence disappears. For simply-connected $G$, exhibiting even better properties are the very saturated covers or even the oasitic covers.

\subsection{$\mca{O}$-Whittaker space}
Recall that we have a decomposition
$$\mca{V}^I = \bigoplus_{\mca{O} \subset \msc{X}_{Q,n}} \mca{V}^I_\mca{O},$$
where for every splitting $W$-orbit $\mca{O}_{\hat{y}} \subset \msc{X}_{Q,n}$ one has
\begin{equation} \label{E:gGG}
\mca{V}^I_{\mca{O}_{\hat{y}}} \simeq \varepsilon \otimes_{\HH_{I, \hat{y}}} \HH_I,
 \end{equation}
where the stabilizer subgroup $W_{\hat{y}} = W_y \subset W$ is a parabolic subgroup for a certain representative $y$ of $\hat{y}$. 

For an Iwahori-spherical representation $\pi$, we define
$$\Wh_\psi(\pi):= \Hom_{\wt{G}}( \mca{V}, \check{\pi} )=\Hom_{\mca{H}_I} (\mca{V}^I, \check{\pi}^I)$$
and
$$\Wh_\psi(\pi)^\sharp:= \Hom_{\mca{H}_I} (\pi^I, ({\rm Ind}_{U^-}^{\wt{G}} \psi^{-1})^I),$$
where $\check{\pi}$ denotes the contragredient representation of $\pi$.
By the perfect  $\wt{G}$-pairing between ${\rm ind}_{U^-}^{\wt{G}} \psi$ and ${\rm Ind}_{U^-}^{\wt{G}} \psi^{-1}$, we get a canonical isomorphism
$$\iota_\psi: \Wh_\psi(\pi) \longrightarrow \Wh_\psi(\pi)^\sharp.$$

\begin{dfn} \label{D:O-Wh}
For every Iwahori-spherical representation $\pi$ and every orbit $\mca{O} \subset \msc{X}_{Q,n}$, the subspace
$$\Whc(\pi)_\mca{O} := \Hom_{\mca{H}_I} (\mca{V}^I_\mca{O}, \check{\pi}^I) \subset \Wh_\psi(\pi),$$
 is called the $\mca{O}$-Whittaker subspace of $\pi$.
\end{dfn}
One has a decomposition  as
$$\Wh_\psi(\pi)= \bigoplus_{\mca{O} \subset \msc{X}_{Q,n}} \Whc(\pi)_\mca{O}.$$
If $T \in \Hom_{\wt{G}}(\pi_1, \pi_2)$ is an intertwining operator, then it induces a well-defined $\C$-homomorphism
$$T_{\psi, \mca{O}}: \Wh_\psi(\pi_2)_\mca{O} \longrightarrow \Wh_\psi(\pi_1)_\mca{O}.$$

\begin{lm}\label{L:Proj}
For every orbit $\mca{O}\subset \msc{X}_{Q,n}$, the $\mca{H}_I$-module $\mca{V}^I_\mca{O}$ is projective. Hence $\Whc(-)_\mca{O}$ is an exact functor.
\end{lm}
\begin{proof}
Since the functor $\Hom(-, \Ind_{U^-}^{\wt{G}} \psi)$ of taking Whittaker model  is exact, we get that $\mca{V}^I$ is  a projective $\mca{H}_I$-module. Any of its direct summand $\mca{V}^I_\mca{O}$ is thus also projective.
\end{proof}

\subsection{A concrete realization} \label{SS:con-Wh}
If $\pi = I(\chi)$ is an unramified genuine principal series, then we expect to have a more concrete description of $\Whc(I(\chi))_\mca{O}$ which arises essentially from composing the Jacquet integral with functionals of $i(\chi)$, see \cite{KP, GSS2}.

Let $\Ftn(i(\chi))$ be the vector space of  functions $\cc$ on $\wt{T}$  satisfying
$$\cc(\wt{t} \cdot \wt{z}) =  \cc(\wt{t}) \cdot \chi(\wt{z}), \quad \wt{t} \in \wt{T} \text{ and } \wt{z} \in \wt{A}=Z(\wt{T})\mathbf{T}(O_{F}).$$
The support of $\cc \in \Ftn(i(\chi))$ is a disjoint union of cosets in $\wt{T}/\wt{A}$.
For every $\gamma \in \wt{T}$, let $\cc_\gamma \in \Ftn(i(\chi))$ be the unique element satisfying
$$\text{supp}(\cc_{\gamma})=\gamma \cdot \wt{A} \text{ and } \cc_{\gamma}(\gamma)=1.$$
Clearly, $\cc_{\gamma \cdot a} = \chi(a)^{-1} \cdot \cc_\gamma$ for every $a\in \wt{A}$. If $\set{\gamma_i}\subset \wt{T}$ is a set of representatives of $\wt{T}/\wt{A}$, then $\set{\cc_{\gamma_i}}$ forms a basis for $\Ftn(i(\chi))$. Let $i(\chi)^\vee$ be the vector space of functionals of $i(\chi)$, which affords the contragredient representation of $i(\chi)$.
 The set $\set{\gamma_i}$ gives rise to linear functionals $l_{\gamma_i} \in i(\chi)^\vee$ such that $l_{\gamma_i}(\phi_{\gamma_j})=\delta_{ij}$, where $\phi_{\gamma_j}\in i(\chi)$ is the unique element such that
$$\text{supp}(\phi_{\gamma_j})=\wt{A}\cdot \gamma_j^{-1} \text{ and  } \phi_{\gamma_j}(\gamma_j^{-1})=1.$$
It is easy to see that for every $\gamma\in \wt{T}$ and $a\in \wt{A}$, one has
$$\phi_{\gamma a}= \chi(a)\cdot \phi_\gamma, \quad  l_{\gamma a} = \chi(a)^{-1} \cdot l_\gamma.$$
Moreover, there is a natural isomorphism of vector spaces
$$\Ftn(i(\chi)) \simeq i(\chi)^\vee$$
 given by
$$ \cc  \mapsto   l_\cc:= \sum_{\gamma_i \in \wt{T}/\wt{A}} \cc(\gamma_i) \cdot l_{\gamma_i}.
$$
It can be checked easily that this isomorphism does not depend on the choice of representatives for $\wt{T}/\wt{A}$.

For any $z\in Y_{\rm ad}=P$, consider again
$${}^z \psi(x)= \psi(z^{-1} x z).$$
There is an isomorphism between $i(\chi)^\vee$ and the space $\Wh_{^z\psi} (I(\chi))^\sharp$ of $(U^-, (^z\psi)^{-1})$-Whittaker functionals  on $I(\chi)$  given by
$$l \mapsto \lambda_l^\sharp$$
with
$$\lambda_l^\sharp:  I(\chi) \to \C, \quad f \mapsto l \circ J^\sharp_\psi(f) \text{ where } J^\sharp_\psi(f)= \int_{U^-} f(u) \psi(u) du  \in i(\chi).$$
Here $f\in I(\chi)$ is an $i(\chi)$-valued function on $\wt{G}$. For any $\cc\in \Ftn(i(\chi))$, write $\lambda_\cc^\sharp \in \Wh_{^z\psi}(I(\chi))^\sharp$ for the $(U^-, (^z\psi)^{-1})$-Whittaker functional of $I(\chi)$ associated to $l_\cc$. Therefore, $\cc \mapsto \lambda_\cc^\sharp$ gives an isomorphism between $\Ftn(i(\chi))$ and $\text{Wh}_{^z\psi}(I(\chi))^\sharp$. For $\gamma \in \wt{T}$, we will write
$$\lambda_\gamma^\sharp:=\lambda_{\cc_\gamma}^\sharp.$$

Let $w_{G}$ be the long element in the Weyl group $W$. To describe the above $\Wh_{^z\psi}(I(\chi))^\sharp$ further, we relate it to the $(U, ({}^{w_G z}\psi)^{-1})$-Whittaker model of $I(\chi)$ which we denote by $\Wh_{{}^{w_G z}\psi}(I(\chi))^\std$ and appeared more frequently in the literature.
More precisely, for 
$${}^{w_G}({}^z\psi) = {}^{w_G z w_G^{-1}} ({}^{w_G} \psi): U \longrightarrow \C^\times$$
we consider the $(U, ({}^{w_G z}\psi)^{-1})$-functional 
$$J^\std_{^{w_Gz}\psi}: I(\chi) \to i(\chi) \text{ given by } J^\std_{^{w_Gz}\psi}(f)=\int_{U} f(\wt{w}_G^{-1}u) (^{w_G z}\psi)(u) du \in i(\chi).$$
We define $\lambda_l^\std:= l \circ J^\std_{^{w_G z}\psi}$ and also
$$\Wh_{^{w_G z}\psi}(I(\chi))^\std = \set{\lambda_l^\std: \ l \in i(\chi)^\vee}.$$
It is clear that 
\begin{equation} \label{E:2J}
J^\sharp_{^z\psi}(f) = J_{{}^{w_Gz}\psi}^\std(R_{w_G}(f)),
\end{equation}
where $R_{g}$ means the right translation action of $I(\chi)$ by $g\in \wt{G}$. One has a commutative diagram
$$\begin{tikzcd}
I(\chi) \ar[d, "{R_{\wt{w}_G}}"]  \ar[rr, "{T(w, \chi)}"] & &  I({}^w \chi) \ar[d, "{R_{\wt{w}_G}}"] \\
I(\chi) \ar[rr, "{T(w, \chi)}"] & &  I({}^w \chi),
\end{tikzcd}$$
which induces the diagram
$$\begin{tikzcd}
\Wh_{^z\psi}(I(\chi))^\sharp   & & \Wh_{^z\psi}(I({}^w \chi))^\sharp \ar[ll, "{T(w, \chi)_{^z\psi}}"'] \\
\Wh_{^{w_Gz}\psi}(I(\chi))^\std \ar[u, "{R_{\wt{w}_G}^*}"]  & &\Wh_{^{w_Gz}\psi}(I(^w \chi))^\std \ar[ll, "{T(w, \chi)_{^{w_Gz}\psi}^\std}"]   \ar[u, "{R_{\wt{w}_G}^*}"],
\end{tikzcd}$$


Now for every $(W, w_G(z))$-orbit $\mca{O} \subset \msc{X}_{Q,n}$, consider
$$\Ftn(i(\chi))_{\mca{O}}=\set{\cc \in \Ftn(i(\chi)):\ {\rm supp}(\cc) \subset \bigcup_{y\in \mca{O}} \s_y\cdot \wt{A} }.$$  
This gives via the above natural isomorphisms of vector spaces
$$\Ftn(i(\chi)) \longrightarrow i(\chi)^\vee \longrightarrow \Wh_{{}^z\psi}(I(\chi))^\sharp$$
a natural subspace 
$$\Wh_{{}^z\psi}(I(\chi))_\mca{O}^\sharp \subset \Wh_{{}^z\psi}(I(\chi))^\sharp.$$
Similarly, one has
$$\Wh_{^{w_G z}\psi}(I(\chi))^\std_\mca{O} \subset \Wh_{^{w_G z}\psi}(I(\chi))^\std$$
with
\begin{equation} \label{psWd}
\dim \Wh_{{}^z\psi}(I(\chi))_\mca{O}^\sharp= \dim \Wh_{^{w_G z}\psi}(I(\chi))^\std_\mca{O} =  \val{\mca{O}}.
\end{equation}
As a consequence of \eqref{E:2J}, the restriction
\begin{equation} \label{F:RwG}
R_{w_G}^*: \Wh_{^{w_G z}\psi}(I(\chi))^\std_\mca{O}  \longrightarrow \Wh_{{}^z\psi}(I(\chi))_\mca{O}^\sharp
\end{equation}
is a  well-defined vector space isomorphism.

We want to show that 
\begin{equation} \label{F:TO}
T(w, \chi)_{{}^z\psi, \mca{O}}^\sharp: \Wh_{{}^z\psi}(I({}^w \chi))_\mca{O}^\sharp \longrightarrow \Wh_{{}^z\psi}(I(\chi))_\mca{O}^\sharp
\end{equation}
is well-defined for every $\mca{O} \subset \msc{X}_{Q,n}$. In view of \eqref{F:RwG}, it suffices to prove that by restricting $T(w, \chi)_{^{w_Gz}\psi}^\std$ the map
\begin{equation} \label{F:TOstd}
T(w, \chi)_{^{w_Gz}\psi, \mca{O}}^\std: \Wh_{^{w_Gz}\psi}(I(^w \chi))^\std \longrightarrow \Wh_{^{w_Gz}\psi}(I(\chi))^\std
\end{equation}
is well-defined.

For notational convenience, we write
$$z^*:=w_G(z) \text{ and } \quad \psi^*={}^{w_G z}\psi ={}^{z^*}({}^{w_G}\psi): U \to \C^\times.$$
If we choose bases $\set{\lambda_{\gamma'}^{{}^w\chi}}_{\gamma' \in \wt{T}/\wt{A}}$ and $\set{\lambda_{\gamma}^\chi}_{\gamma \in \wt{T}/\wt{A}}$ for $\Wh_{\psi^*}(I({}^w \chi))_\mca{O}^\std$ and $\Wh_{\psi^*}(I(\chi))_\mca{O}^\std$, then $T(w, \chi)_{\psi^*}^\std$ is naturally represented by the so-called scattering matrix
$$[\tau_{\psi^*}(w, \chi, \gamma', \gamma)]_{\gamma', \gamma\in \wt{T}/\wt{A}}$$
satisfying
\begin{equation} \label{F:tau}
T(w, \chi)_{\psi^*}^\std (\lambda_{\gamma'}^{{}^w\chi}) = \sum_{\gamma \in \wt{T}/\wt{A}}  \tau_{\psi^*}(w, \chi, \gamma', \gamma) \cdot \lambda_\gamma^\chi.
\end{equation}
Furthermore, if one chooses a set $\mfr{R} \subset Y$ of representatives of the set $\msc{X}_{Q,n}$, it gives a natural basis for the domain and codomain of $T(w, \chi)_{\psi^*}^\std$, which is then represented by $[\tau_{\psi^*}(w, \chi, \s_{y'}, \s_y)]_{y', y\in \mfr{R}}$.
In any case, $\tau_{\psi^*}(w, \chi, \gamma', \gamma)$ satisfies the cocycle relation reflecting the decomposition of the intertwining operator $T(w, \chi)$ into rank-one intertwining operators. Thus, it suffices to determine $\tau_{\psi^*}(w_\alpha, \chi, \gamma', \gamma)$ for a simple reflection $w_\alpha$, which is given as follows:
\begin{enumerate}
\item[--] $\tau_{\psi^*}(w_\alpha, \chi,\gamma', \gamma) = \tau_{\psi^*}^1(w_\alpha, \chi,\gamma', \gamma) + \tau_{\psi^*}^2(w_\alpha, \chi,\gamma', \gamma)$ with
$$\tau_{\psi^*}^i(w_\alpha, \chi,\gamma' \cdot a', \gamma \cdot a)=({}^{w_\alpha} \chi)^{-1}(a') \cdot \tau_{\psi^*}^i(w_\alpha, \chi, \gamma', \gamma) \cdot \chi(a)$$
for every  $a, a'\in \wt{A}$ and $1\lest i \lest 2$;
\item[--] $\tau_{\psi^*}^1(w_\alpha, \chi, \s_{y'}, \s_y)=0$  unless  $y' \equiv y \mod Y_{Q,n}$, and
$$\tau_{\psi^*}^1(w_\alpha, \chi, \s_y, \s_y) = (1-q^{-1}) \frac{\chi (\wt{h}_\alpha(\varpi^{n_\alpha}))^{k_{z, y,\alpha}}}{1-\chi (\wt{h}_\alpha(\varpi^{n_\alpha}))},$$
 where $k_{z, y,\alpha}=\ceil{\frac{1+\angb{y+z^*}{\alpha}}{n_\alpha}}$;
\item[--] $\tau_{\psi^*}^2(w_\alpha, \chi, \s_{y'}, \s_y)=0$   unless $y' \equiv w_\alpha[y]_{z^*} \mod Y_{Q,n}$ and
$$\tau_{\psi^*}^2(w_\alpha, \chi, \s_{w_\alpha[y]_{z^*}}, \s_y) = (-1, \varpi)_n^{\angb{y+z^*}{\alpha}\cdot D(y,\alpha^\vee)} \cdot \g_{^{w_G}\psi}(\angb{y+z^*}{\alpha}Q(\alpha^\vee)).$$
\end{enumerate}
In the above, the Gauss sum $\mbf{g}_{^{w_G}\psi}(k)$ is given as in \cite[\S 3.6]{GSS2}, and is essentially $\g_\alpha(\psi, (-,\varpi)_n^k)$ in \eqref{D:GS}. We see that in particular if $z \in P\cap Y$, then 
$$\tau_{\psi^*}^i(w_\alpha, \chi, \s_{y'}, \s_y) = \tau_{\psi^*}^i(w_\alpha, \chi, \s_{y' + z^*}, \s_{y +z^*})$$
for $1 \lest i \lest 2$. We also remark that  for $z=\rho$ and thus $\mfr{f}(\psi^*) = O_F$, the scattering matrix was given in \cite{KP, Mc2, Ga2}.  For $z=0$, the formula for $\tau_{\psi^*}^i(w_\alpha, \chi, \s_{y'}, \s_y)$ is given in \cite[\S 5.3.3]{GSS3}. 

From \eqref{F:tau} and the above description of the $\tau$-functions, we see that $T(w, \chi)_{\psi^*, \mca{O}}^\std$ in \eqref{F:TOstd} is indeed well-defined, and so is \eqref{F:TO}. In fact, this is the  raison d'\^etre for considering the $(W, z^*)$ orbits in $\msc{X}_{Q,n}$.

When specialized to $z=0$ and thus for $W$-orbits $\mca{O}$, one has the natural embedding
$$\iota(\chi)_{\psi, \mca{O}}: \Wh_\psi(I(\chi))_\mca{O} \into \Wh_\psi(I(\chi))^\sharp.$$

\begin{conj} \label{C:iden}
Keep notations as above. Then $\iota(\chi)_{\psi, \mca{O}}$ gives an isomorphism
$$\iota(\chi)_{\psi, \mca{O}}: \Whc(I(\chi))_\mca{O} \simeq \Whc(I(\chi))_\mca{O}^\sharp$$
for every $W$-orbit $\mca{O} \subset \msc{X}_{Q,n}$. Moreover, $\iota(\chi)_{\psi, \mca{O}}$ is equivariant with respect to the two homomorphisms on Whittaker spaces induced from $T(w, \chi)$, i.e., 
$$\iota({}^w \chi)_{\psi, \mca{O}} \circ T(w, \chi)_{\psi, \mca{O}} = T(w, \chi)_{\psi, \mca{O}}^\sharp \circ \iota(\chi)_{\psi, \mca{O}}.$$
\end{conj}

The remainder of this paper is devoted to determining $\dim \Whc(\pi)_\mca{O}$  for constituents of a regular unramified principal series $I(\chi)$ or a unitary unramified principal series. We prove analogues of certain conjectural formulas for such genuine Whittaker dimensions in \cite{Ga6, Ga7} -- we deal with the case $\mfr{f}(\psi)=\mfr{p}$ in this paper while that in loc. cit. assumes $\mfr{f}(\psi)= O_F$.
In fact, for unitary unramified $I(\chi)$ we also investigate partially the relation between the two cases for different conductors, see Corollary \ref{C:Wh-equi}.

\subsection{Regular unramified $I(\chi)$}
Consider an unramified $\chi$ satisfying the following:
\begin{enumerate}
\item[--] $\chi$ is regular, that is, its stabilizer subgroup of $W$ is trivial,
\item[--] the set $\Phi(\chi):=\set{\alpha\in \Phi: \chi( \wt{h}_\alpha(\cdot)^{n_\alpha} ) =\val{\cdot}_F}$ is a subset of $\Delta$.
\end{enumerate}
Such an exceptional $\chi$ (following \cite{KP}) gives a regular unramified genuine principal series $I(\chi)$ that decomposes
\begin{equation} \label{E:Rod}
I(\chi)^{\rm ss} = \bigoplus_{S \subset \Phi(\chi)} \pi_S,
\end{equation}
where the left hand side denotes the semisimplification of $I(\chi)$ and $\pi_S \in \Irrg(\wt{G})$. The decomposition is multiplicity-free and the irreducible constituent $\pi_S$ is characterized by its Jacquet module, see \cite{Rod4} and \cite[\S 3]{Ga6}. For example, if $\Phi(\chi) =\Delta$, then $\pi_\Delta = \Theta(\chi)$ is a theta representation and $\pi_\emptyset$ is a covering analogue of the Steinberg representation.

Recall that for each $S \subset \Phi(\chi) \subset \Delta$, there is a representation $\sigma_S$ (possibly reducible) of the Weyl group.
More precisely, one has for the Jacquet module
$$(\pi_S)_U = \bigoplus_{w\in W_S} \delta_B^{1/2} \cdot i({}^{w^{-1}} \chi),$$
where $W_S \subset W$ is a union of Kazhdan--Lusztig right cells of $W$. The representation $\sigma_S$ is then the direct sum of the cell representations associated to $W_S$. In a more concrete form (see \cite[Corollary 6.5]{Ga6}), one has
\begin{equation} \label{E:al-sum}
\sigma_S = \sum_{S':\ S \subseteq S' \subseteq \Phi(\chi)} (-1)^{\val{S'- S}} \cdot \Ind_{W(S')}^W \varepsilon_{W(S')},
\end{equation}
where $W(S') \subset W$ is the Weyl subgroup generated by $S'$.

Also recall the permutation representation 
$$\sigma^\msc{X}_{\mca{O}}: W \longrightarrow {\rm Perm}(\mca{O})$$
 for each $\mca{O} \subset \msc{X}_{Q,n}$, which is given by $\sigma^\msc{X}_\mca{O}(w)(y) = w(y)$.

\begin{thm} \label{T:reg-ps}
Let $I(\chi)$ be a regular unramified genuine principal series such that $\Phi(\chi) \subset \Delta$. Let $S\subseteq \Phi(\chi)$. Then for every splitting orbit $\mca{O} \subset \msc{X}_{Q,n}$ one has
$$\dim \Whc(\pi_S)_\mca{O} = \angb{\sigma_S}{\sigma^\msc{X}_\mca{O}}_W.$$
Hence, for $\wt{G}$ such that every orbit $\mca{O}$ is splitting (for example, those as in Corollary \ref{C:aOsplt} and Example \ref{E:KPS}) one has 
$$\dim \Whc(\pi_S) = \angb{\sigma_S}{\sigma^\msc{X}}_W.$$
\end{thm}
\begin{proof}
The idea is the same as in \cite{Ga6} where $\mfr{f}(\psi)= O_F$. Let $\mca{O} =\mca{O}_{\hat{y}}$ be a splitting orbit with $W_{\hat{y}} \subset W$ a parabolic subgroup.

First, we consider the case when $S= \Phi(\chi) = \Delta$, and thus $\pi_\Delta = \Theta(\chi)$ is a theta representation of $\wt{G}$. Here $\chi$ is an exceptional character and thus $\Theta(\chi)$ is the Langlands quotient of $I(\chi)$, and is also the image of the intertwining operator $T(w_G, \chi)$, where $w_G \in W$ is the longest Weyl element. Its contragredient $\check{\pi}_\Delta$ is the irreducible subrepresentation of $I(\chi^{-1})$, and is also the theta representation associated to the exceptional genuine character ${}^{w_G}(\chi^{-1})$ of $Z(\wt{T})$. In any case, $\check{\pi}_\Delta$ is the unique unramified representation of $I({}^{w_G}(\chi^{-1}))$, and we have that
$$(\check{\pi}_\Delta)^I |_{ \HH_{\tilde{W}_{\aff}} }$$
is the trivial representation. That is, assuming $\tilde{W}_{\aff}$ is generated by $\set{ \tilde{\alpha}_0^\vee } \cup \set{ \tilde{\alpha}_i^\vee: 1\lest i \lest r}$, then each $T_{\tilde{\alpha}_i^\vee} \in \mca{H}_I, 0\lest i \lest r$ acts on $(\check{\pi}_\Delta)^I$ by $q$. This gives
$$\begin{aligned}
 \dim \Wh_\psi(\pi_\Delta)_\mca{O} & =\dim  \Hom_{\HH_I}( \mca{V}^I_\mca{O}, (\check{\pi}_\Delta)^I) \\
 & =\dim \Hom_{\HH_I}( \varepsilon \otimes_{\HH_{W_{\hat{y}}}} \HH_I, (\check{\pi}_\Delta)^I)  \text{ by \eqref{E:gGG}} \\
& =\dim \Hom_{\HH_{W_{\hat{y}}}}( \varepsilon, (\check{\pi}_\Delta)^I |_{\HH_{W_{\hat{y}}}}) \\
& = \begin{cases}
1 & \text{ if } W_{\hat{y}} = \set{1}, \\
0 & \text{ otherwise}.
\end{cases} 
 \end{aligned} $$
This shows that
\begin{equation} \label{E:dWT}
 \dim \Whc(\pi_\Delta)_\mca{O}= \angb{\varepsilon_W}{ \sigma^\msc{X}_\mca{O} }_W 
= \begin{cases}
 1 & \text{ if  $\mca{O}$ is a free $W$-orbit},\\
 0 & \text{ otherwise.}
 \end{cases}
 \end{equation}
 
 Second, assume in general $S \subseteq \Phi(\chi) \subseteq \Delta$. Denote by 
 $$\pi(w, \chi) \subset I(w, \chi)$$
  the image of intertwining operator $T(w, \chi)$. We have 
 $$\pi_S =\sum_{S':\ S \subset S' \subset \Phi(\chi)} (-1)^{\val{S'- S}} \cdot \pi(w_{S'}, \chi) \in \msc{R}(\Irrg(\wt{G})),$$
 where $w_{S'}$ means the longest element in the Weyl subgroup $W(S')$ generated by $S'$. Since $\Whc(-)_\mca{O}$ is an exact functor, we have
$$\dim \Whc(\pi_S)_\mca{O} = \sum_{S':\ S \subset S' \subset \Phi(\chi)} (-1)^{\val{S'- S}} \cdot \dim \Whc(\pi(w_{S'}, \chi))_\mca{O}.$$
However, since $\pi(w_{S'}, \chi)$ is equal to the full representation parabolically induced from the theta representation $\pi_{S'}^{M_{S'}}$ of the Levi subgroup $M_{S'}$, we have
$$\Whc(\pi(w_{S'}, \chi))_\mca{O} = \angb{\varepsilon_{W(S')}}{ \sigma^\msc{X}_\mca{O} }_{W(S')}= \angb{\Ind_{W(S')}^W \varepsilon_{W(S')}}{ \sigma^\msc{X}_\mca{O} }_{W},$$
where the first equality follows from \eqref{E:dWT}. In view of \eqref{E:al-sum}, we get that
$$\Whc(\pi_S)_\mca{O} = \angb{\sigma_S}{\sigma^\msc{X}_\mca{O}}_W$$
for every splitting orbit $\mca{O} \subset \msc{X}_{Q,n}$. This completes the proof.
\end{proof}

We can  verify for $\mfr{f}(\psi)=\mfr{p}$ an analogue of \cite[Conjecture 6.9 (i)]{Ga6}.

\begin{cor} \label{C:lbd}
Let $\Theta(\chi)=\pi_\Delta$ be an unramified theta representation associated to $\chi$ with $\Phi(\chi)=\Delta$. For every splitting orbit $\mca{O} \subset \msc{X}_{Q,n}$, one has
$$\dim \Whc(\pi)_\mca{O} \gest \dim \Whc(\Theta(\chi))_\mca{O}$$
for every irreducible Iwahori-spherical representation $\pi$ of $\wt{G}$. In particular, if every orbit $\mca{O} \subset \msc{X}_{Q,n}$ is splitting, then $\dim \Whc(\pi) \gest \dim \Whc(\Theta(\chi))$ for every such $\pi$ as well.
\end{cor}
\begin{proof}
Let $\pi$ be any Iwahori-spherical representation of $\wt{G}$. From \eqref{E:dWT}, we see that $\dim \Whc(\Theta(\chi))_\mca{O}=1$ if and only if $\mca{O} \subset \msc{X}_{Q,n}$ is a free $W$-orbit. In this case,
$$\mca{V}^I_\mca{O} \simeq \mca{H}_I$$
as $\mca{H}_I$-module. It gives that 
$$\dim \Whc(\pi)_\mca{O} = \dim \Hom_{\mca{H}_I}(\mca{H}_I, \check{\pi}^I) \gest 1$$
and thus the result follows.
\end{proof}

Corollary \ref{C:lbd} in particular applies to Kazhdan--Patterson covers and Savin covers by Example \ref{E:KPS}. In fact, for $G=\GL_r$, we expect the inequality in Corollary \ref{C:lbd} holds for every irreducible $\pi \in \Irrg(\wt{G})$ (not necessarily Iwahori-spherical), which however may fail for general $\wt{G}$. For more discussion, see \cite[Page 339]{Ga6}.

\begin{rmk}
First, retaining the notation $\pi(w, \chi) \subset I({}^w \chi)$ as in the proof of Theorem \ref{T:reg-ps}, we can define the space $\Wh_\psi(\pi(w, \chi))_\mca{O}^\sharp$ as the image of $T(w, \chi)_{\psi, \mca{O}}^\sharp$. If we assume Conjecture \ref{C:iden}, then $\Wh_\psi(\pi(w, \chi))_\mca{O} \simeq \Wh_\psi(\pi(w, \chi))_\mca{O}^\sharp$ for every orbit $\mca{O} \subset \msc{X}_{Q,n}$; in particular \cite[Conjecture 4.7]{Ga6} holds. In this case, we have
$$\dim \Whc(\pi_S)_\mca{O} = \dim \Whc(\pi_S)_\mca{O}^\sharp = \angb{\sigma_S}{\sigma^\msc{X}_\mca{O}}_W,$$
where the second equality follows from \cite[Theorem 6.6]{Ga6}. We note that the proof in \cite{Ga6} actually applies to all $0$-persistent covers, and thus we expect that Theorem \ref{T:reg-ps} holds for all orbits $\mca{O} \subset \msc{X}_{Q,n}$ as long as $\wt{G}$ is a $0$-persistent cover.

Second, if we consider an orbit $\mca{O}$ with the S-property (see Definition \ref{D:S-ppty}) and assume further that $\mu_{\hat{y}}|_{ \HH_{\tilde{W}_{\aff, y}} }$ is nontrivial, then the proof of Theorem \ref{T:reg-ps} applies in the same way for such $\mca{O}$. In view of the first remark above, it is plausible that for $0$-persistent covers every orbit satisfies these two assumptions, i.e., the S-property and non-triviality of $\mu_{\hat{y}}|_{ \HH_{\tilde{W}_{\aff, y}} }$.
\end{rmk}

\section{Unitary unramified genuine principal series} \label{S:uniPS}
In this section, we assume that $G$ is almost simple and simply-connected. Let $\wt{G}$ be the $n$-fold cover of $G$ associated with
$$Q(\alpha^\vee)=1$$
for any short coroot $\alpha^\vee$. The general case of $Q(\alpha^\vee)$ imposes no extra difficulties, except for notational complications.

\subsection{$R$-group and decomposition of $I(\chi)$}
Consider a unitary unramified principal series $I(\chi)$. Let $R_\chi \subset W_\chi$ be the R-group associated to $I(\chi)$ satisfying 
$$\C[R_\chi] \simeq \text{End}(I(\chi)).$$
One has a natural correspondence between $\text{Irr}(R_\chi)$ and the constituents $\Pi(\chi)$ of $I(\chi)$:
$$\Irr(R_\chi) \longleftrightarrow  \Pi(\chi), \quad \sigma \leftrightarrow \pi_\sigma,$$
which is normalised such that $\pi_{\mbm{1}}$ is the unique unramified constituent of $I(\chi)$. It is known that $R_\chi$ is abelian (see \cite{Luo3}), and therefore the decomposition of $I(\chi)$ is multiplicity-free. The nontrivial $R_\chi$ is given in Table 3.

\begin{table}[H]  \label{T3}
\caption{Nontrivial $R_\chi$}
\vskip 5pt
\begin{tabular}{|c|c|c|c|c|c|c|c|c|c|c|}
\hline
 & $A_r$  &  $B_r$ & $C_r$  & $D_r, r $ even &  $D_r, r $ odd  & $E_6$  &  $E_7$  \\
\hline
$ R_\chi $ & $\Z/d\Z, d|(r+1)$ & $\Z/2\Z$  & $\Z/2\Z$ & $\Z/2\Z$ or $(\Z/2\Z)^2$  & $\Z/2\Z$ or $\Z/4\Z$  & $\Z/3\Z$ & $\Z/2\Z$  \\ 
\hline
\end{tabular}
\end{table}
\vskip 10pt

For every $w \in R_\chi$, we have a well-defined isomorphism
$$\msc{A}(w, \chi)= \gamma(w, \chi) \cdot T(w,\chi):  I(\chi) \longrightarrow  I(\chi),$$
where  $\gamma(w, \chi)$ is the gamma-factor associated with $w$ and $\chi$, i.e.,
$$\gamma(w, \chi)^{-1}= \prod_{\alpha \in \Phi_\w}   \frac{1 - q^{-1} \chi_\alpha^{-1}  }{ 1-\chi_\alpha }.$$
More explicitly, one has
$$\msc{A}(w, \chi)|_{\pi_\sigma} = \sigma(w) \cdot \text{id}$$
for every irreducible constituent $\pi_\sigma \subset I(\chi)$. It induces an isomorphism
$$\msc{A}(w, \chi)_\psi:  \Whc(I(\chi)) \longrightarrow  \Whc( I(\chi) ),$$
where $\dim \Whc(I(\chi)) =\val{ \msc{X}_{Q,n} }$. This gives a $\val{\msc{X}_{Q,n}}$-dimensional representation 
$$\sigma^\Whc:   R_\chi  \longrightarrow   \GL( \Whc(I(\chi)) )$$
of $R_\chi$ given by 
$$\sigma^\Whc(w):=  \msc{A}(w, \chi)_\psi.$$
Indeed, the proof in \cite{Ga7} for $\sigma^{ \Wh^\std_{{}^{w_G \rho}\psi} }$ afforded by $\Wh_{{}^{w_G\rho}\psi}(I(\chi))^\std$ relies only on the fact that $R_\chi$ is abelian, thus it applies to $\Wh_\psi(I(\chi))$ here. 

For every $W$-orbit $\mca{O} \subset \msc{X}_{Q,n}$, we have from restriction
$$\msc{A}(w,  \chi)_{\psi, \mca{O}}: \Whc(I({}^w\chi))_{\mca{O}} \longrightarrow \Whc(I(\chi))_{\mca{O}}$$
 This gives rise to the following:
\begin{enumerate}
\item[$\bullet$] a subrepresentation
$$\sigma_{\mca{O}}^\Whc:  R_\chi \longrightarrow  \GL(\Whc(I(\chi))_{\mca{O}}), \quad \sigma_{\mca{O}}^\Whc(w):= \msc{A}(w,  \chi)_{\psi, \mca{O}}$$
such that the decomposition
$$\sigma^\Whc = \bigoplus_{ \mca{O} \subset \msc{X}_{Q,n}  } \sigma_{\mca{O}}^\Whc$$
holds;
\item[$\bullet$] for every $\pi_\sigma \subset I(\chi)$ the decomposition
$$\Whc(\pi_\sigma)= \bigoplus_{ \mca{O} \subset \msc{X}_{Q,n} }  \Whc(\pi_\sigma)_{\mca{O}},$$
where $\Whc(\pi_\sigma)_{\mca{O}}$ is isomorphic to the subspace of $\Whc(I(\chi))_\mca{O}$ consisting of Whittaker functionals of $I(\chi)$ restricted to $\pi_\sigma$.
\end{enumerate}
It follows from Fourier inversion (cf. \cite[Theorem 5.6]{Ga7}) that for every $\sigma\in \Irr(R_\chi)$ we have
$$\dim \Whc(\pi_\sigma)_{\mca{O}_y} = \angb{\sigma}{ \sigma^\Whc_{\mca{O}_y} }_{R_\chi}$$
for every orbit $\mca{O}_y \subset \msc{X}_{Q,n}$. We thus get
$$\dim \Whc(\pi_\sigma) = \angb{\sigma}{ \sigma^\Whc }_{R_\chi}$$
for every $\pi_\sigma \subset I(\chi)$.
The representation $\sigma^\Whc$ is in general difficult to compute. Indeed, one has
$$\sigma_{\mca{O}_y}^\Whc(w) = \gamma(w, \chi)^{-1} \cdot \text{Tr}(T(w, \chi)_{\psi, \mca{O}}),$$
and $T(w, \chi)_{\psi, \mca{O}}$ is not easily computable.

If $\wt{G}$ is a very saturated cover of  $G$, then it was speculated (see \cite{Ga7, Ga8}) that 
$$\sigma^{{\rm Wh}^\std_{{}^{w_G\rho}\psi}} \simeq \sigma_{[-\rho]}^\msc{X},$$
the second of which is more accessible for computation. The goal of this section is to prove an analogue of such equality for $\psi$ (see Theorem \ref{T:uni-ps} below). Moreover, as a consequence of our discussion, we also prove a result on the variation of Whittaker dimensions of representations inside the same $L$-packet for a unitary unramified representation, see Corollary \ref{C:Wh-equi}. This latter result was also speculated in \cite[Conjecture 5.7]{GSS3}.


\subsection{Explicit Whittaker dimension of $\pi_\sigma$}
Henceforth, we assume that $G$ is almost simple and simply-connected, and that $\wt{G}^{(n)}$ is a very saturated cover of $G$ with
$Q(\alpha^\vee)=1$ for any short simple coroot $\alpha^\vee$. We have in this case
$$Y_{Q,n} = Y_{Q,n}^{sc} = n \cdot Y^{sc}.$$
and
$$\wt{G}^\vee \simeq G^\vee.$$
The modified affine Weyl group is
$$\tilde{W}_{\aff} = Y_{Q,n}^{sc} \rtimes W = (nY) \rtimes W,$$
which acts on $Y \subset Y\otimes \R$ naturally. For any $y \in Y$, we have
$$\tilde{W}_{\aff, y} \subset \tilde{W}_\aff, $$
the stabilizer subgroup of $y$. 

For every root $\alpha$ and $k\in \Z$, we have
$$w_{\alpha, k}  = (k\alpha^\vee, w_\alpha) \in W_{\aff}= Y \rtimes W,$$
where $w_\alpha(y) = y - \angb{y}{\alpha} \alpha^\vee$ the usual reflection. Similarly, we write
$$\tilde{w}_{\alpha, k} =  (k\alpha_{Q,n}^\vee, w_\alpha) =  (k \cdot n \alpha^\vee, w_\alpha)   \in \tilde{W}_{\aff}= Y_{Q,n} \rtimes W.$$
We have
$$w_\alpha = \tilde{w}_{\alpha, 0} = w_{\alpha, 0}$$
for every root $\alpha \in \Phi$.

Let $\alpha^\dag  \in \Phi^+$ be the highest root of $\Phi$. Then $\tilde{\alpha}^\dag = \alpha^\dag/n \in \Phi_{Q,n} = \Phi/n$ is the highest root of the modified root system. 
Recall the set
$$S_\aff= \set{\tilde{w}_{\alpha^\dag, 1}} \cup \set{w_{\alpha_i}: \ 1\lest i \lest r}$$
of generators of $\tilde{W}_\aff$. For every $y\in Y$ there exists a proper subset $S_{\aff, y} \subsetneq S_\aff$ such that
$$\tilde{W}_{\aff, y} \simeq W(S_{\aff, y}),$$
where the right hand sides denotes the subgroup of $\tilde{W}_\aff$ generated by $S_{\aff, y}$.
The surjective map
$$\eta: \tilde{W}_\aff \onto W$$
restricts to give an isomorphism 
$$\tilde{W}_{\aff, y} \simeq \eta(\tilde{W}_{\aff, y}).$$
If we set $\alpha_0:=-\alpha^\dag$ and define as usual
$$\Delta_\aff = \set{\alpha_0} \cup \Delta,$$
then it is clear that
$$\eta(\tilde{W}_{\aff, y}) = \langle w_\alpha:  \alpha \in \Delta_{\aff, y} \rangle \subset W$$ 
for a subset $\Delta_{\aff, y} \subsetneq \Delta_\aff$.

The following result will play an important role in the proof of Theorem \ref{T:uni-ps} below.
\begin{prop} \label{P:uni-key}
Let $\wt{G}^{(n)}$ be a very saturated cover of an almost simple simply-connected group $G$ such that $Q(\alpha^\vee)=1$ for any short simple coroot $\alpha^\vee$. Let $y\in Y$. Assume $\alpha_0 \in \Delta_{\aff, y}$. Then 
$$\eta(\tilde{W}_{\aff, y}) \cap R_\chi = \set{1}$$
for any unitary unramified genuine character $\chi$ of $Z(\wt{T})$.
\end{prop}
\begin{proof}
We verify this from a case by case checking.

First, for type $A_r$, i.e., $G=\SL_{r+1}$, the cover $\wt{\SL}_{r+1}^{(n)}$ is  very saturated  if and only if
$$\gcd(n, r+1)=1.$$
Identifying $W\simeq S_{r+1}$, we know that $R_\chi$ is $W$-conjugate to $\Z/d\Z$ generated by
$$w:=(1, 2, ..., d)(d+1, d+2, ..., 2d) .... (kd + 1, ..., r+1)$$
for some $d$ with $r+1=dk$. Note that $w$ is the Coxeter element of the parabolic Weyl subgroup $\prod_{i=1}^k S_d \subset W$.  Since a conjugate of $w$ lies in $W(\Delta_{\aff, y})$, it follows that the set $\Delta_{\aff, y}$ contains $k$-many connected components of the extended Dynkin diagram, each of which is of size at least $d-1$ and one of which contains $\alpha_0$. 
Writing $y=\sum_{i=1}^{r+1} y_i e_i$ with $y_i \in \Z$ and $\sum_i y_i =0$, together with the two equalities
$$\angb{y}{\tilde{\alpha}^\dag} = (y_1 - y_{r+1})/n = 1 \text{ and } \sum_i y_i =0,$$
this shows that  there exist distinct $1\lest i_j \lest r+1$ with $1\lest j \lest k$ such that
$$d\cdot \left( \sum_{1\lest j \lest k} y_{i_j} \right) = n.$$
Since $\gcd(r+1, n) =1$ and $d|(r+1)$, it follows that $d=1$ and thus $w=1$.

Second, for type $B_r$, the cover $\wt{\Spin}_{2r+1}^{(n)}$ is very saturated if and only if $n$ is odd. Write 
$$y=\sum_{i=1}^r y_i e_i \in Y \text{ with } y_i \in \Z \text{ and } 2|(\sum_i y_i).$$
Depending on the parity of $r$, we have the following two cases of $R_\chi$:
\begin{enumerate}
\item[--] If $r$ is even, then $R_\chi \simeq \Z/2\Z$ is $W$-conjugate to the group generated by
$$w=(12)(34)... (r-1, r) = w_{\alpha_1} w_{\alpha_3} ... w_{\alpha_{r-1}}.$$
Suppose $w \in \eta(\tilde{W}_{\aff, y})$ with $\alpha_0 \in \Delta_{\aff, y}$. Now if $\alpha_1 \in \Delta_{\aff, y}$, then the two equalities
$$\angb{y}{\tilde{\alpha}^\dag} = 1, \quad \angb{y}{\alpha_1} = 0$$
gives that $y_1 = y_2 = n/2$, which is a contradiction since $n$ is odd. Thus, we see that $\alpha_1 \notin \Delta_{\aff, y}$ and this gives that
$$\set{\alpha_0, \alpha_3, \alpha_5, ..., \alpha_{r-1}} \subset \Delta_{\aff, y}.$$
(In fact, the above inclusion also follows from \cite{Osh}.) It then follows that
$$y_1 + y_2 = n \text{ and } y_i= y_{i+1} \text{ for all odd } i \in [3, r-1].$$
However, this contradicts the fact that $\sum_i y_i$ is even.
\item[--] If $r$ is odd, then $R_\chi \simeq \Z/2\Z$ and $W$-conjugate to the group generated by
$$w=w_{\alpha_1} w_{\alpha_3} ... w_{\alpha_{r-2}} w_{\alpha_r}.$$
The argument is similar to the case above and we have $w \notin \eta(\tilde{W}_{\aff, y})$.
\end{enumerate}

Third, covers of $G$ of type $C_r$ are in fact simpler. More precisely, the cover $\wt{\Sp}_{2r}^{(n)}$ is very saturated if and only if $n$ is odd.  We claim that there is no $y \in Y$ such that $\alpha_0 \in \Delta_{\aff, y}$. Indeed, in the standard coordinates, one has
$$\tilde{\alpha}^\dag = 2e_1/n.$$
Thus, the equality $\angb{y}{\tilde{\alpha}^\dag} =1$ has no solution for $y\in Y$. Hence, the desired equality holds vacuously.

Fourth, for type $D_r$, the cover $\wt{\Spin}_{2r}^{(n)}$ is very saturated if and only if $n$ is odd. Using the Bourbaki notations, we could write every $y\in Y$ as
$$y=\sum_{i=1}^r y_i e_i \text{ with } y_i \in \Z \text{ and } 2|\sum_i y_i.$$
We have $\alpha_0 = -\alpha^\dag = - (e_1 + e_2)$. Now according to the parity of $r$ and the possible $R_\chi$, we have the following cases:
\begin{enumerate}
\item[--] $r$ is odd and $R_\chi \simeq \Z/2\Z$ generated by $w:=w_{\alpha_{r-1}} w_{\alpha_r}$. If $w$ is $W$-conjugate to an element in $\eta(\tilde{W}_{\aff, y})$, then it is easy to see (cf. \cite{DyMi} as well)
$$\alpha_1 \in \Delta_{\aff, y}$$
as well. Note that $\alpha_0\in \Delta_{\aff, y}$ by assumption. These two imply that $y_1 = y_2 = n/2$, which is a contradiction. 
\item[--] $r$ is odd and $R_\chi \simeq \Z/4\Z$ generated by
$$w:=w_{\alpha_1} w_{\alpha_3} ... , w_{\alpha_{r-2}} \cdot w_{\alpha_{r-1}} w_{\alpha_r}.$$
If $w$ is $W$-conjugate to an element in $\eta(\tilde{W}_{\aff, y})$, then we get (see \cite{DyMi})
$$\set{\alpha_0, \alpha_1, \alpha_2, \alpha_4, \alpha_6, ..., \alpha_{r-5}, \alpha_{r-3}, \alpha_{r-1}} \subset \Delta_{\aff, y}$$
or $\set{\alpha_0, \alpha_1, \alpha_2, \alpha_4, \alpha_6, ..., \alpha_{r-5}, \alpha_{r-3}, \alpha_{r-1}} \subset \Delta_{\aff, y}$. In either case, we have $\set{\alpha_0, \alpha_1} \subset \Delta_{\aff, y}$, which gives again $y_1= y_2 =n/2$ and is a contradiction.
\item[--] $r$ is even and any nontrivial $R_\chi$ is isomorphic to either $\Z/2\Z$ or $\Z/2\Z \times \Z/2\Z$. If $R_\chi \simeq \Z/2\Z$, it is generated by either $w_{\alpha_{r-1}} w_{\alpha_r}$ or by
$$w:= w_{\alpha_1} w_{\alpha_3} w_{\alpha_5} ... w_{\alpha_{r-3}} w_{\alpha_{r-1}}.$$
If $R_\chi \simeq \Z/2\Z\times \Z/2\Z$, then 
$$R_\chi \simeq \langle w, w_{\alpha_{r-1}} w_{\alpha_r}  \rangle.$$
For either such $R_\chi$, the argument is similar to the case when $r$ is odd, and there exists no $y\in Y$ satisfying both $-\alpha_0 \in \Delta_{\aff, y}$ and $\eta(\tilde{W}_{\aff, y}) \cap R_\chi \ne \set{1}$.
\end{enumerate}

Fifth, a cover of $\wt{E}_6^{(n)}$ is very saturated if and only if $3\nmid n$. The nontrivial $R_\chi \simeq \Z/3\Z$ is generated by 
$$w:= w_1 w_3 \cdot w_6 w_5$$
in the notations of Bourbaki. If $\eta(\tilde{W}_{\aff, y}) \cap R_\chi \neq\set{1}$, then (see \cite[Theorem 3.4]{DyMi}) we have either
$$\set{\alpha_0, \alpha_2, \alpha_5, \alpha_6} \subset \Delta_{\aff, y} \text{ or } \set{\alpha_0, \alpha_2, \alpha_1, \alpha_3} \subset \Delta_{\aff, y}.$$
In the first case, one has
$$\angb{y}{\alpha^\dag}=n \text{ and } \angb{y}{\alpha_i}=0 \text{ for } i=2, 5, 6.$$
In terms of the coordinates $e_i$, we can write $y=\sum_{i=1}^8 y_i e_i$ with $y_i\in \Z/2$ and $y_6=y_7=-y_8$. Then above conditions give that
$$3(y_3 + y_8) = 2n$$
and  in particular $3|n$, which is a contradiction. The second case is similar, and thus $\eta(\tilde{W}_{\aff, y}) \cap R_\chi = \set{1}$.

Lastly, we consider cover $\wt{E}_7^{(n)}$ of the simply-connected $E_7$, which is very saturated if and only if $n$ is odd.
Again, in the Bourbaki notations, the nontrivial $R_\chi \simeq \Z/2\Z$ is generated by
$$w:=w_{\alpha_2} w_{\alpha_5} w_{\alpha_7}.$$
It then follows from \cite{DyMi} (see also \cite{Osh}) that if $w$ is $W$-conjugate to an element in $\eta(\tilde{W}_{\aff, y})$, then necessarily 
$$\set{\alpha_0, \alpha_2, \alpha_3} \subset \Delta_{\aff, y}.$$
If we write $y = \sum_{i=1}^7 k_i \alpha_i^\vee$ with $k_i\in \Z$, then this gives
$$k_1 = n, k_2=0, k_3=n/2,$$
which is a contradiction since $n$ is odd.

All the above completes the proof.
\end{proof}

For such $\wt{G}$, the unique distinguished genuine character of $Z(\wt{T})$ is used to construct a natural algebra isomorphism (see \cite[\S 15.4]{GG})
$$\mca{H}(\wt{G}, I) \simeq \mca{H}(G, I).$$
We continue to write $\mca{H}_I$ when there is no chance of confusion. 
Also, $\mca{H}_I$ can be identified with the affine Hecke algebra $\HH_{\tilde{W}_\aff}$ associated with the affine Weyl group $W_\aff = Y_{Q,n}^{sc} \rtimes W$ of the dual group $\wt{G}^\vee$.
By Borel's theorem, every Iwahori-spherical representation $\pi$ of $\wt{G}$ corresponds uniquely to a finite-dimensional irreducible $\mca{H}_I$-module.  

If a linear $G$ has connected center, then a Langlands correspondence was established by Kazhdan and Lusztig \cite{KL2} between the  ireducible $\mca{H}_I$-modules and certain Kazhdan--Lusztig parameters of triples 
$(\tau, u,  \rho)$. Such a correspondence was extended by Reeder \cite{Ree4} to include the case when the center of $G$ may not be connected. In fact, the work in loc. cit. also includes ramified $I(\chi)$ when $G$ is assumed to have connected center. In recent works by Aubert--Baum--Plymen--Solleveld (for example \cite{ABPS17}), a Langlands correspondence for irreducible constituents of general $I(\chi)$ was investigated and a link to the geometric side of the parameter space in terms of the ``extended quotients" was established.

Thus one has a local Langlands correspondence
$$\mca{L}: \Irr_\epsilon(\wt{G})^I \longrightarrow \Irr(\mca{H}_I) \longrightarrow \set{(\tau, u, \rho)}, \quad \pi\mapsto \pi^I \mapsto (\tau, u, \rho)$$
from the Iwahori-spherical representations to the set of Kazhdan--Lusztig--Reeder (KLR) parameters. Every KLR parameter that arises in the correspondence above satisfies the following:
\begin{enumerate}
\item[--] $\tau \in \wt{T}^\vee$ is a semisimple element in the dual torus inside $\wt{G}^\vee$;
\item[--] $u \in \wt{G}^\vee$ is a nilpotent element such that $\tau u \tau^{-1} = q^{-1}\cdot u$;
\item[--] $\rho \in \Irr_{\rm geom}(\pi_0(\wt{G}^\vee_{\tau, u}))$ is a certain irreducible representation of connected component group of the mutual centralizer subgroup $\wt{G}^\vee_{\tau, u} \subset \wt{G}^\vee$ that appears in the homology $H_*(\mfr{B}^{\tau, u}_{\wt{G}^\vee}, \C)$.
\end{enumerate}
Here $\mfr{B}_{\wt{G}^\vee}^{\tau, u}$ denotes the variety of Borel subgroups of $\wt{G}^\vee$ containing both $\tau$ and $u$. In particular, the irreducible constituent $\pi_\sigma \subset I(\chi)$ has a KLR parameter
$$(\tau = s_\chi, u=0, \rho=\sigma),$$
where $s_\chi \in \wt{G}^\vee$ is the Satake parameter associated to $I(\chi)$.

Before stating our main result on the $\psi$-Whittaker dimension of $\pi_\sigma$, we first recall a homomorphism (depending on the half sum $\tilde{\rho} := n\rho$ of positive roots of $\wt{G}^\vee$) 
$$\zeta_{\tilde{\rho}}: R_\chi \longrightarrow \C^\times$$ given as follows.
Consider the simply-connected cover  $\wt{G}_{sc}^\vee$ of $\wt{G}^\vee$ as in
$$\begin{tikzcd}
Z \ar[r, hook] & \wt{G}^\vee_{sc} \ar[r, two heads] & \wt{G}^\vee.
\end{tikzcd}$$
Attached to $I(\chi)$ is the Satake parameter $s_\chi \in \wt{G}^\vee$.  For any $w\in R_\chi \subset W_\chi$ and any lifting $s_\chi' \in \wt{G}^\vee_{sc}$ of $s_\chi$, one has
$$w(s_\chi')/s_\chi' \in Z,$$
which is independent of the particular lifting $s_\chi'$ chosen. Since
$$Z = \Hom(P_{Q,n}/Y_{Q,n}, \C^\times) \text{ with } P_{Q,n}=nP,$$
this gives a natural pairing
$$\phi: (P_{Q,n}/Y_{Q,n}) \times R_\chi \longrightarrow \C^\times$$
given by
$$\phi(z, w): =\left( \frac{w(s_\chi')}{ s_\chi' } \right)(z) =(w(z) - z)(s_\chi) \in \C^\times.$$
Now we just set
\begin{equation} \label{D:z-rho}
\zeta_{\tilde{\rho}}(-):= \phi(\tilde{\rho}, -).
\end{equation}
Since $2\tilde{\rho} \in Y_{Q,n}^{sc} \subset Y_{Q,n}$, we see that
\begin{equation} \label{E:zeta-quad}
\zeta_{\tilde{\rho}}^2 = \mbm{1},
\end{equation}
that is, $\zeta_{\tilde{\rho}}$ is a (possibly trivial) quadratic character.

\begin{thm} \label{T:uni-ps}
Let $\wt{G}$ be a very saturated cover of an almost simple simply-connected  $G$ associated with $Q(\alpha^\vee)=1$ for any short simple coroot $\alpha^\vee$. Let $I(\chi)$ be a unitary $(K, s_K)$-unramified genuine principal series of $\wt{G}$. Let $\mca{O} \subset \msc{X}_{Q,n}$ be a $W$-orbit satisfying the S-property (see Definition \ref{D:S-ppty}). Then, as representations of $R_\chi$ one has
$$\sigma_\mca{O}^\Whc \simeq  \zeta_{\tilde{\rho}}^{-1} \otimes \sigma_\mca{O}^\msc{X};$$
equivalently,
$$\dim \Whc(\pi_\sigma)_\mca{O} =  \angb{\sigma \otimes \zeta_{\tilde{\rho}} }{  \sigma_\mca{O}^\msc{X} }_{R_\chi}$$ 
for every $\sigma \in \Irr(R_\chi)$.
\end{thm}
\begin{proof}
Let $\mca{O} =\mca{O}_y \subset \msc{X}_{Q,n}$ be a $W$-orbit with the S-property. It suffices to prove the second equality.

Since $Y_{Q,n}=Y_{Q,n}^{sc} = nY$ for very saturated covers, it follows from Definition \ref{D:S-ppty} that
$$\dim \Wh_\psi(\pi_\sigma)_\mca{O} = \dim \Hom_{\mca{H}_I}( \mu_y \otimes_{ \mca{H}_{\tilde{W}_{\aff, y}} } \mca{H}_I, \check{\pi}_\sigma^I) = \dim \Hom_{\mca{H}_{\tilde{W}_{\aff, y}}}(\mu_y, \check{\pi}_\sigma^I|_{ \mca{H}_{\tilde{W}_{\aff, y}} }).$$
Here $\mca{H}_{\tilde{W}_{\aff, y}} \subset \mca{H}_I$ is the associated/deformed subalgebra of  the group $\tilde{W}_{\aff, y}$, and we have the $q\to 1$ operation of sending $\mca{H}_{\tilde{W}_{\aff, y}}$-modules to $\tilde{W}_{\aff, y}$-modules. Moreover, since $\tilde{W}_{\aff, y}$ is a finite Weyl subgroup of $\tilde{W}_\aff$, this operation is an isometry, see \cite[Proposition 10.11.4]{Car}. That is, we have
$$\Hom_{\mca{H}_{\tilde{W}_{\aff, y}}}(\mu_y, \check{\pi}_\sigma^I|_{ \mca{H}_{\tilde{W}_{\aff, y}} }) = \Hom_{\tilde{W}_{\aff, y}}(\mu_y, (\check{\pi}_\sigma^I|_{ \mca{H}_{\tilde{W}_{\aff, y}} })_{q\to 1}),$$
where we view $\mu_y$ as a character of $\tilde{W}_{\aff, y}$ as well. For $\mca{H}_I$-modules, one also has the operation $q\to 1$ by Lusztig's theory (see \cite{Lu-CA4}, \cite[\S 6]{Ree-00} or \cite[\S 12]{ABPS17}), which commutes with restriction, and thus
$$(\check{\pi}_\sigma^I|_{ \mca{H}_{\tilde{W}_{\aff, y}} })_{q\to 1} = (\check{\pi}_\sigma^I)_{q \to 1} |_{ \tilde{W}_{\aff, y} }.$$

Consider the homology $H_*(\mfr{B}_{\wt{G}^\vee}^{s_\chi}, \C)$ of the variety of Borel subgroups of $\wt{G}^\vee$ containing $s_\chi$. It is naturally a $ \tilde{W}_\aff \times \pi_0(Z_{\wt{G}^\vee}(s_\chi))$-module. Identifying
$$R_\chi = \pi_0(Z_{\wt{G}^\vee}(s_\chi)),$$
we can view $H_*(\mfr{B}_{\wt{G}^\vee}^{s_\chi}, \C)$ as a $\tilde{W}_\aff \times R_\chi$-module. It follows from \cite[Proposition 6.2]{Kat-83}  that
\begin{equation} \label{E:Kato}
H_*(\mfr{B}_{\wt{G}^\vee}^{s_\chi}, \C) =  \Ind_{(Y_{Q,n} \rtimes W_{s_\chi}^0)\rtimes R_\chi}^{\tilde{W}_\aff \times R_\chi} H_*(\mfr{B}_{\wt{G}^{\vee, 0}_{s_\chi}}, \C ),
\end{equation}
where
\begin{enumerate}
\item[--] $W_{s_\chi} \simeq W_\chi \subset W$ is the stabilizer subgroup of $s_\chi$ with $W_{s_\chi} = W_{s_\chi}^0 \rtimes R_\chi$;
\item[--] $\wt{G}^\vee_{s_\chi} \subset \wt{G}^\vee$ is the stabilizer subgroup of $s_\chi$,  its connected component is denoted by $\wt{G}_{s_\chi}^{\vee, 0}$ whose Weyl group is $W_{s_\chi}^0$;
\item[--] the inclusion $(Y_{Q,n} \rtimes W_{s_\chi}^0) \rtimes R_\chi \into \tilde{W}_\aff \times R_\chi$ is given by
$$(y, w, r) \mapsto ( (y, w, r) ,r),$$
where we note the canonical $(Y_{Q,n} \rtimes W_{s_\chi}^0) \rtimes R_\chi = Y_{Q,n} \rtimes W_\chi \subset \tilde{W}_\aff$;
\item[--] the action of $Y_{Q,n} \rtimes (W_{s_\chi}^0 \rtimes R_\chi)$ on $H_*(\mfr{B}_{\wt{G}^{\vee, 0}_{s_\chi}}, \C ) $ is given by $(s_\chi \otimes \rho_{\rm Spr}) \otimes \mbm{1}$, where $s_\chi$ is viewed as a character of $Y_{Q,n}$, and $W_{s_\chi}^0$ acts on $H_*(\mfr{B}_{\wt{G}^{\vee, 0}_{s_\chi}}, \C )$ by the Springer correspondence $\rho_{\rm Spr}$, and $R_\chi$ acts trivially (see \cite[Page 199]{Kat-83}).
\end{enumerate}
It follows from \cite[Theorem 9.1]{ABPS17} and its proof that as $\tilde{W}_\aff = Y_{Q,n} \rtimes W$-modules 
\begin{equation} \label{E:q1}
(\check{\pi}_\sigma^I)_{q\to 1} = \Hom_{R_\chi}(\sigma, H_*(\mfr{B}_{\wt{G}^\vee}^{s_\chi}, \C))^\vee.
\end{equation}
It is a classical result of Borel that
\begin{equation} \label{E:Borel}
H_*(\mfr{B}_{\wt{G}^{\vee, 0}_{s_\chi}}, \C ) \simeq \C[W_{s_\chi}^0]
\end{equation}
as $W_{s_\chi}^0$-representations.  Now \eqref{E:Kato} \eqref{E:q1} and \eqref{E:Borel} together give that
$$(\check{\pi}_\sigma^I)_{q\to 1} = (\Ind_{Y_{Q,n} \rtimes R_\chi}^{Y_{Q,n} \rtimes W} (s_\chi \rtimes \sigma^{-1}))^\vee = \Ind_{Y_{Q,n} \rtimes R_\chi}^{Y_{Q,n} \rtimes W} (s_\chi^{-1} \rtimes \sigma).$$

It follows
$$(\check{\pi}_\sigma^I|_{ \mca{H}_{\tilde{W}_{\aff, y}} })_{q\to 1} = \Res_{\tilde{W}_{\aff, y}} \Ind_{Y_{Q,n} \rtimes R_\chi}^{Y_{Q,n} \rtimes W} (s_\chi^{-1} \rtimes \sigma).$$
Since we need to consider the restriction  to $\tilde{W}_{\aff, y}$, which may not be a parabolic subgroup,  we ``pull-back" all the  representations to be those of subgroups $W$ as in \cite[Page 54]{Ree-00}. 

For any $w\in W$ and $\varsigma:={}^w \chi$, we have the following diagram:
$$\begin{tikzcd}
\eta(\tilde{W}_{\aff, y}) \ar[r, "{\eta^{-1}}"] & \tilde{W}_{\aff, y} \\
R_{\varsigma} \cap \eta(\tilde{W}_{\aff, y}) \ar[r, hook, "{\eta^{-1}}"] \ar[u, hook] \ar[rrr, bend right=15, "{s_\varsigma^y}"'] & (Y_{Q,n} \rtimes R_\varsigma) \cap \tilde{W}_{\aff, y}  \ar[u, hook] \ar[r, hook] & (Y_{Q,n} \rtimes R_\varsigma) \ar[r, "{s_\varsigma^{-1}}"]   & \C^\times.
\end{tikzcd}$$
Define 
$$s_\varsigma^y= s_\varsigma \circ \eta^{-1}: R_\varsigma \cap \eta(\tilde{W}_{\aff, y}) \longrightarrow \C^
\times.$$
For simplicity of notation, we write for the rest of the proof that 
$$W_{\aff, y}^\eta:= \eta(\tilde{W}_{\aff, y}) \text{ and } W_{y, \varsigma} :=W_{\aff, y}^\eta \cap R_\varsigma.$$
Mackey's formula gives that as representations of $\tilde{W}_{\aff, y}^\eta$ one has
$$\begin{aligned}
\eta_* \left( (\check{\pi}_\sigma^I|_{ \mca{H}_{\tilde{W}_{\aff, y}} })_{q\to 1} \right)   
& = \bigoplus_{w \in W_{\aff, y}^\eta \backslash W /R_\chi} \eta_*\left( \Ind_{\tilde{W}_{\aff, y} \cap R_{{}^w \chi} }^{ \tilde{W}_{\aff, y}  }  {}^w (s_\chi^{-1}) \otimes  {}^w \sigma\right) \\
& =  \bigoplus_{w \in W_{\aff, y}^\eta \backslash W /R_\chi}  \Ind_{ W_{y, {}^w \chi} }^{ W_{\aff, y}^\eta } ( s_{{}^w \chi}^y \otimes {}^w\sigma).
\end{aligned}
$$
This gives
$$\begin{aligned}
\dim \Hom_{\tilde{W}_{\aff, y}}(\mu_y, (\check{\pi}_\sigma^I|_{ \mca{H}_{\tilde{W}_{\aff, y}} })_{q\to 1}) & = \angb{ \eta_*(\mu_y)  }{ \eta_* \left( (\check{\pi}_\sigma^I|_{ \mca{H}_{\tilde{W}_{\aff, y}} })_{q\to 1} \right) }_{W_{\aff, y}^\eta} \\
& =  \bigoplus_{w \in W_{\aff, y}^\eta \backslash W /R_\chi} \angb{\eta_*(\mu_y)}{ \Ind_{ W_{y, {}^w \chi} }^{ W_{\aff, y}^\eta } (s_{{}^w \chi}^y \otimes {}^w\sigma)  }_{ W_{\aff, y}^\eta } \\
& =  \bigoplus_{w \in W_{\aff, y}^\eta \backslash W /R_\chi} \angb{\eta_*(\mu_y)}{  s_{{}^w \chi}^y \otimes {}^w\sigma  }_{ W_{y, {}^w \chi} }
\end{aligned}$$

On the other hand, since $\sigma_\mca{O}^\msc{X} =\Ind_{W_{\aff, y}^\eta}^W \mbm{1}$, one has
$$\begin{aligned}
\angb{\sigma \otimes \zeta_{\tilde{\rho}} }{  \sigma_\mca{O}^\msc{X} }_{R_\chi}  & = \bigoplus_{w\in R_\chi \backslash W / W_{\aff, y}^\eta} \angb{\mbm{1}}{ \zeta_{\tilde{\rho}} \otimes \sigma }_{R_\chi \cap {}^w (W_{\aff, y}^\eta)} \\
& = \bigoplus_{w \in W_{\aff, y}^\eta \backslash W /R_\chi} \angb{\mbm{1}}{ {}^w \zeta_{\tilde{\rho}} \otimes {}^w \sigma }_{W_{y, {}^w \chi}}.
\end{aligned}$$
We will verify the following equality
\begin{equation} \label{E:keq}
\eta_*(\mu_y)^{-1} \otimes s^y_{{}^w \chi} = {}^w \zeta_{\tilde{\rho}} \text{ on } W_{y, {}^w \chi} \text{  for every } w \in W_{\aff, y}^\eta \backslash W /R_\chi, 
\end{equation}
which clearly suffices to give the desired result. We discuss the two cases according to whether $\alpha_0 \in \Delta_{\aff, y}$ or not.

First, if $\alpha_0 \in \Delta_{\aff, y}$, then it follows from Proposition \ref{P:uni-key} that 
$$W_{y, {}^w \chi} =\set{1} \text{ for every } w,$$
and thus  equality \eqref{E:keq} holds trivially.

Second, we assume $\alpha_0 \notin \Delta_{\aff, y}$ and thus $\tilde{W}_{\aff, y} \subset W$ is isomorphic to a parabolic subgroup. 
Hence the orbit $\mca{O}_y \subset \msc{X}_{Q,n}$ is splitting. In this case, we simply write $W_y:=\tilde{W}_{\aff, y}$.
It follows that $\mu_y = \varepsilon_{W_y}$
and $\eta$ is the identity map on $W_y$. Also, in this case $s_{{}^w \chi}^y$ is trivial on $W_{y, {}^w \chi}$ for all $w$, and the equality \eqref{E:keq} is equivalent to
\begin{equation} \label{E:keq-2}
\varepsilon_{W_y} = {}^w \zeta_{\tilde{\rho}} \text{ on } W_{y, {}^w\chi}.
\end{equation}
Now we take ${}^w w_1 \in W_{y, {}^w \chi} = W_y \cap R_{{}^w \chi} = W_y \cap {}^w R_\chi$ with $w_1 \in R_{\chi}$ being the standard representative as in \cite{Key2}. We have 
$$\varepsilon_{W_y}({}^w w_1) = \varepsilon_W({}^w w_1) = \varepsilon_W(w_1).$$
On the other hand, 
$${}^w \zeta_{\tilde{\rho}}({}^w w_1) = \zeta_{\tilde{\rho}}(w_1) = s_\chi(w_1(\tilde{\rho}) - \tilde{\rho}).$$

Let $w_1 = \prod_{j=1}^k w_{\alpha_{i_j}}$ and set $\Delta_{w_1}=\set{\alpha_{i_j}: 1\lest j \lest k}$. We divide the computation of the above into two situations according whether $\Delta_{w_1}$ is totally disconnected in the Dynkin diagram of $G$. For every $\alpha \in \Phi$, we denote
$$\chi_\alpha:=\chi(\wt{h}_\alpha(\varpi^{n_\alpha})).$$

First, if $\Delta_{w_1}$ is totally disconnected, then by \cite[\S 3]{Key2} or \cite{Ga8}, we see that $l(w_1) = \val{\Delta_{w_1}} = k$ and $\chi_{\alpha_{i_j}} = -1$ for every $j$. Since in this case 
$$w_1(\tilde{\rho}) - \tilde{\rho} = - \sum_j \alpha_{i_j, Q, n}^\vee,$$
it is clear that
$$ {}^w \zeta_{ \tilde{\rho} }({}^w w_1) =  \varepsilon_{W_y}({}^w w_1)  = (-1)^{l(w_1)}.$$
This applies to the cases when $R_\chi \simeq \Z/2\Z$ or $\Z/2\Z \times \Z/2\Z$, in particular, to all covers of type $B_r, C_r, E_7$ and some covers of $A_r$ and $D_r$.

Second, we deal with the remaining cases. Again, we use the standard notations and coordinates from Bourbaki. We discuss case by case.
\begin{enumerate}
\item[$\bullet$] For $A_r$ with $R_\chi \simeq \Z/d\Z, (r+1) =dk$ we have
$$w_1 =(1, ..., d)(d+1, ..., 2d) ... (dk-d+1, ..., dk)$$
where $\chi_{\alpha_i} = \xi$, a primitive $d$-th roots of unity. We see that
$$s_\chi(w_1(\tilde{\rho}) - \tilde{\rho}) = \xi^{-d(d-1)/2}  = (-1)^{d-1} = \varepsilon_W(w_1).$$
\item[$\bullet$] For $D_r$ with $R_\chi \simeq \Z/4\Z$, we have necessarily $r$ is odd, and a generator for $R_\chi$ is
$$w_1 = w_{\alpha_1} w_{\alpha_3} ... w_{\alpha_{r-4}} \cdot w_{\alpha_{r-2}} w_{\alpha_{r-1}} w_{\alpha_r},$$
and 
$$\chi_{\alpha_i} = -1 \text{ for } i \in \set{1, 3, 5, ..., r-4, r-2}, \quad \chi_{\alpha_{r-1}}^2 = -1, \quad \chi_{\alpha_r} = \chi_{\alpha_{r-1}}^{-1}.$$
A straightforward computation gives that
$$s_\chi(w_1(\tilde{\rho}) - \tilde{\rho}) = (-1)^{(r-1)/2},$$
which is equal to $\varepsilon_W(w_1) = (-1)^{(r+3)/2}$.
\item[$\bullet$] The last case is for $E_6$ with $R_\chi \simeq \Z/3\Z$ with a generator 
$$w_1 = w_{\alpha_1} w_{\alpha_3} \cdot w_{\alpha_6} w_{\alpha_5},$$
and
$$\chi_{\alpha_1}= \chi_{\alpha_3} = \chi_{\alpha_5}^{-1} = \chi_{\alpha_6}^{-1} = \xi,$$
where $\xi$ is a primitive third root of unity. A direct computation gives that
$$s_\chi(w_1(\tilde{\rho}) - \tilde{\rho}) = 1 = \varepsilon_W(w_1),$$
where the first equality already follows from \eqref{E:zeta-quad}
\end{enumerate}
This completes the proof of the equality \eqref{E:keq-2} and also the theorem.
\end{proof}

For $\mca{O}$ with the S-property as in Theorem \ref{T:uni-ps}, we have that
$$\sigma_\mca{O}^\Whc \simeq  \zeta_{\tilde{\rho}}^{-1} \otimes \sigma_\mca{O}^\msc{X}  \simeq  \zeta_{\tilde{\rho}} \otimes \sigma_\mca{O}^\msc{X}.$$

\begin{cor} \label{C:uni-ps}
Let $\wt{G}$ be an oasitic cover of a simply-connected semisimple group $G$ associated with $Q(\alpha^\vee)=1$ for any short simple coroot $\alpha^\vee$. Then, one has
$$\dim \Whc(\pi_\sigma)_\mca{O} =  \angb{\sigma \otimes \zeta_{\tilde{\rho}} }{  \sigma_\mca{O}^\msc{X} }_{R_\chi} = \angb{\sigma \otimes \zeta_{\tilde{\rho}}^{-1} }{  \sigma_\mca{O}^\msc{X} }_{R_\chi}$$ 
for every $W$-orbit $\mca{O} \subset \msc{X}_{Q,n}$ and $\sigma \in \Irr(R_\chi)$. Hence,
$$\dim \Whc(\pi_\sigma) =  \angb{\sigma \otimes \zeta_{\tilde{\rho}} }{  \sigma^\msc{X} }_{R_\chi} = \angb{\sigma \otimes \zeta_{\tilde{\rho}}^{-1} }{  \sigma^\msc{X} }_{R_\chi}.$$
\end{cor}

\subsection{Conductor of $\psi$ varied}
Note that for any $z$ one clearly has a decomposition
$$\Wh_{{}^z\psi}(I(\chi))^\sharp = \bigoplus_{\sigma \in \Irr(R_\chi)} \Wh_{{}^z\psi}(\pi_\sigma)^\sharp.$$
We want to compare $\dim \Wh_\psi(\pi_\sigma)^\sharp$ against $\dim \Wh_{{}^{\rho}\psi} (\pi_\sigma)^\sharp$ with $\mfr{f}({}^{\rho}\psi) = O_F$, and verify some speculations given in \cite[Conjecture 5.7]{GSS3}.

\begin{cor} \label{C:Wh-equi}
Let $\wt{G}$ be an oasitic cover of an almost simple simply-connected $G$ with $Q(\alpha^\vee)=1$. As in Proposition \ref{P:O1}, let $y_{-\rho} \in Y$ be such that
$$\mfr{m}_{-\rho}: \msc{X}_{Q,n} \longrightarrow \msc{X}_{Q,n}, y\mapsto y+ y_{-\rho}$$
induces a bijective correspondence between $W$-orbits and $(W,-\rho)$-orbits of the domain and codomain of $\mfr{m}_{-\rho}$ respectively. Let $I(\chi)$ be a unitary unramified principal series. Assume that Conjecture \ref{C:iden} holds. Also assume $R_\chi \simeq \Z/2\Z$ or $\Z/2\Z \times \Z/2\Z$. Then for every $\pi_\sigma \subset I(\chi)$ and every $W$-orbit $\mca{O} \subset \msc{X}_{Q,n}$, one has
$$\dim \Wh_\psi(\pi_\sigma)^\sharp_\mca{O} = \dim \Wh_{ {}^{\rho} \psi} (\pi_{\sigma \otimes \zeta^{-1}_{\tilde{\rho}}})^\sharp_{\mfr{m}_{-\rho}(\mca{O})};$$
hence,
$$\dim \Wh_\psi(\pi_\sigma)^\sharp = \dim \Wh_{ {}^{\rho} \psi} (\pi_{\sigma \otimes \zeta^{-1}_{\tilde{\rho}}})^\sharp$$
as well.
\end{cor}
\begin{proof}
Assuming Conjecture \ref{C:iden}, one has
$$\dim \Wh_\psi(\pi_\sigma)^\sharp_\mca{O} = \dim \Wh_\psi(\pi_\sigma)_\mca{O} = \angb{\sigma \otimes \zeta_{\tilde{\rho}}^{-1} }{  \sigma_{[0]}^\mca{O} }_{R_\chi},$$
where the second equality follows from Corollary \ref{C:uni-ps}. (Here we use the notation $\sigma_{[0]}^\mca{O}$ for $\sigma_\mca{O}^\msc{X}$ to highlight the twisting parameter.)

On the other hand, from \eqref{F:RwG} we have
$$\dim \Wh_{{}^\rho\psi}(\pi_\sigma)_{\mca{O}'}^\sharp = \dim \Wh_{^{-\rho}(^{w_G}\psi)}(\pi_\sigma)^\std_{\mca{O}'}$$
for every $(W, -\rho)$-orbit $\mca{O}' \subset \msc{X}_{Q,n}$. Note that every $\mca{O}' = \mfr{m}_{-\rho}(\mca{O})$ for a unique $W$-orbit $\mca{O} \subset \msc{X}_{Q,n}$. Since $\mfr{f}({}^{-\rho}(^{w_G} \psi))= O_F$, then for $R_\chi \simeq (\Z/2\Z)^i, i=1, 2$ it is shown in  \cite[Theorem 1.4]{Ga8} that
\begin{equation} \label{E:c-O}
\dim \Wh_{ {}^{-\rho} (^{w_G}\psi)} (\pi_{\sigma \otimes \zeta_{\tilde{\rho}}^{-1}})^\std_{\mfr{m}_{-\rho}(\mca{O})} = \angb{\sigma_{[-\rho]}^{\mfr{m}_{-\rho}(\mca{O})}}{\sigma \otimes \zeta_{\tilde{\rho}}^{-1}}_{R_\chi}.
\end{equation}
However, it follows from Proposition \ref{P:O1} that for such covers $\wt{G}$, one has
$$\sigma_{[-\rho]}^{\mfr{m}_{-\rho}(\mca{O})} \simeq \sigma_{[0]}^\mca{O}.$$
The desired first equality then follows from comparing all the above. The second equality is clear from this.
\end{proof}

\begin{rmk}
The equality \eqref{E:c-O} is expected to hold for general $R_\chi$ (see \cite{Ga7, Ga8}) and if proven will give a stronger version of Corollary \ref{C:Wh-equi} for all $R_\chi$. 
Moreover, for very saturated covers, one can compute directly the two sides of the equality in Corollary \ref{C:Wh-equi} as in \cite{Ga8}. Indeed,  the results in loc. cit. was obtained by considering the scattering matrix $\tau_{\psi^*}$ in \S \ref{SS:con-Wh} with $\psi^*={}^{z^*}(^{w_G}\psi)$ and $z^*=-\rho$, while the method equally applies to $\tau_{^{w_G}\psi}$. In fact, even if $\chi$ is ramified, similar results are expect to hold following the work \cite{GoSz, Szp6} on the local coefficients matrices and scattering matrices.
\end{rmk}


\end{document}